\documentclass[12pt]{article}
\usepackage{amsmath,mathdots}
\usepackage{amssymb}
\usepackage{latexsym}
\usepackage{amsthm}
\usepackage{amsmath,amssymb,amsfonts,amsthm}
\usepackage{mdframed}
\usepackage{graphicx}
\usepackage{stmaryrd}
\usepackage{gensymb}
\usepackage{bm}
\usepackage{tikz}  

\usepackage{stmaryrd}
\usepackage{multirow}

\usepackage{graphicx}
\usepackage{tikz}
\usetikzlibrary{matrix,arrows}
\usetikzlibrary{positioning}
\usetikzlibrary{fit}
\usetikzlibrary{patterns}

\usepackage[colorlinks,
linkcolor=blue,
anchorcolor=blue,
citecolor=blue
]{hyperref}
\usepackage[margin=2.9cm]{geometry}
\usepackage[normalem]{ulem}

\newtheorem{thm}{Theorem}[section]

\newtheorem{lem}[thm]{Lemma}

\newtheorem{coro}[thm]{Corollary}
\newtheorem{conj}{Conjecture}

\newtheorem*{remark}{Remark}

\usepackage{graphicx}
\usepackage{tikz}
\usetikzlibrary{matrix,arrows}
\usetikzlibrary{positioning}
\usetikzlibrary{fit}
\usetikzlibrary{patterns}

\newcommand{\T}{\mathcal{T}}

\newcommand{\pattern}[4]{										
	\raisebox{0.6ex}{
		\begin{tikzpicture}[scale=0.35, baseline=(current bounding box.center), #1]
		\foreach \x/\y in {#4}		\fill[gray!50] (\x,\y) rectangle +(1,1);
		\draw (0.01,0.01) grid (#2+0.99,#2+0.99);
		\foreach \x/\y in {#3}		\filldraw (\x,\y) circle (6pt);
		\end{tikzpicture}}
}

\newcommand{\boks}[2]{({#1, #2})}   

\definecolor{red}{rgb}{1,0,0}
{}

 \allowdisplaybreaks

\begin{document}

\begin{center}
{\large \bf Joint equidistributions of mesh patterns 123 and 132 with minus antipodal shadings}
\end{center}

\begin{center}
Shuzhen Lv$^{a}$ and Philip B. Zhang$^{b*}$
\\[6pt]

$^{a,b}$College of Mathematical Sciences \\
Tianjin Normal University \\ Tianjin  300387, P. R. China\\[6pt]
Email:
 $^{a}${\tt lvshuzhenlsz@yeah.net},
      $^{b*}${\tt zhang@tjnu.edu.cn}
\end{center}

\noindent\textbf{Abstract.}
The study of  joint equidistributions of mesh patterns 123 and 132 with the same symmetric shadings was recently initiated by Kitaev and Lv, where 75 of 80 potential joint equidistributions were proven. In this paper, we prove 112 out of 126 potential joint equidistributions of mesh patterns 123 and 132 with the same minus antipodal shadings. As a byproduct, we present 562 joint equidistribution results for non-symmetric and non-minus-antipodal shadings. To achieve this, we construct bijections, find recurrence relations, and obtain generating functions. Moreover, we demonstrate that the joint distributions of several pairs of mesh patterns are related to the unsigned Stirling numbers of the first kind. \\



\section{Introduction}\label{intro}
An $n$-permutation is a bijection of the set $[n]:=\{1, 2, \ldots, n\}$.
We denote $S_n$ the set of all $n$-permutations. 
Patterns in permutations have garnered significant attention in the literature (see~\cite{Kit}), with numerous researchers actively exploring this area. The concept of a \emph{mesh pattern} was introduced by Br\"and\'en and Claesson~\cite{BrCl}. A \emph{mesh pattern} of length \( k \) is defined as a pair \((\tau, R)\), where \(\tau \in S_k\) and \( R \) is a subset of \([0,k] \times [0,k]\), with \([0,k]\) denoting the set of integers from 0 to \( k \).
Let \(\boks{i}{j}\) denote the box whose corners have coordinates \((i,j)\), \((i,j+1)\), \((i+1,j+1)\), and \((i+1,j)\). The horizontal lines (resp., vertical lines) represent the values (resp., positions) in the pattern. Mesh patterns are formed by shading boxes in \( R \). The figure
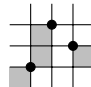
\[
\pattern{scale=0.8}{3}{1/1,2/3,3/2}{0/0,1/1,1/2,3/1}
\]
represents the mesh pattern with $\tau=132$ and $R = \{\boks{0}{0},\boks{1}{1},\boks{1}{2},\boks{3}{1}\}$. Many papers have focused on the investigation of mesh patterns, see~\cite{AKV,  Hilmarsson2015Wilf, KL, KZ, KZZ, KLv}. 

The {\em reverse, complement, and inverse} of $\pi=\pi_1\pi_2\cdots \pi_n \in S_n$, are respectively defined as $\pi^r=\pi_n\pi_{n-1}\cdots\pi_1$, $\pi^c= (n + 1 - \pi_1)(n + 1 - \pi_2)\cdots (n + 1-\pi_n)$, and $\pi^i=\pi^i_1\pi^i_2\cdots\pi^i_n$ such that $\pi^i_j=\pi_m$ if $\pi_j=m$, for any $1\leq j,m \leq n$. Given a mesh pattern $(\tau,R)$, the corresponding operations are defined as
\begin{align*}
(\tau,R)^r&=(\tau^r,R^r),\text{ where } R^r=\{(n - x,y): (x,y)\in R\};\\
(\tau,R)^c&=(\tau^c,R^c),\text{ where } R^c=\{(x,n - y): (x,y)\in R\};\\
(\tau,R)^i&=(\tau^i,R^i),\text{ where } R^i=\{(y,x): (x,y)\in R\}.
\end{align*}
For instance, for the mesh pattern $p=(213,\{(0,1),(1,3),(2,2)\})$, the specific operations are applied as follows.
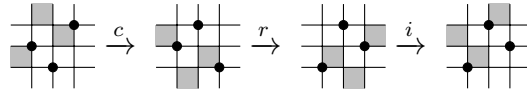
\[
\pattern{scale=0.8}{3}{1/2,2/1,3/3}{0/1,1/3,2/2}\xrightarrow{c}\pattern{scale=0.8}{3}{1/2,2/3,3/1}{0/2,1/0,2/1}\xrightarrow{r}\pattern{scale=0.8}{3}{1/1,2/3,3/2}{1/1,2/0,3/2}\xrightarrow{i}\pattern{scale=0.8}{3}{1/1,2/3,3/2}{1/1,0/2,2/3}
\]

A sequence $\pi'=\pi_{i_1}\pi_{i_2}\cdots\pi_{i_k}$ is an {\em occurrence} of a mesh pattern $p=(\tau, R)$ in a permutation $\pi=\pi_1\pi_2\cdots \pi_n$ if $(i)$ $\pi'$ is order-isomorphic to $\tau$, and $(ii)$ the shaded squares given by $R$ do not contain any elements of $\pi$ not appearing in $\pi'$. Otherwise, we say that $\pi$ {\em avoids} $p$, and we let $S_n(p)$ denote the set of such permutations.
For instance, the permutation 24513 contains 2 occurrences of the mesh pattern $(132, \{\boks{0}{0},\boks{1}{1},\boks{1}{2},\boks{3}{1}\})$ as the subsequence 241 and 243. 

Two patterns \( q_1 \) and \( q_2 \) are said to be \textit{Wilf-equivalent}, denoted \( q_1 \sim q_2 \), if \( |S_n(q_1)| = |S_n(q_2)| \). Moreover, \( q_1 \) and \( q_2 \) are called \textit{equidistributed}, denoted \( q_1 \sim_{d} q_2 \), if  for any nonnegative integers \( n \) and \( k \), the number of \( n \)-permutations with \( k \) occurrences of \( q_1 \) is equal to that with \( k \) occurrences of \( q_2 \). A stronger notion, \textit{joint equidistribution} of two mesh patterns, denoted \( (q_1, q_2) \sim_d (q_2, q_1) \) or \( q_1 \sim_{jd} q_2 \), means that the number of \( n \)-permutations with \( k \) occurrences of \( q_1 \) and \( \ell \) occurrences of \( q_2 \) is equal to that with \( \ell \) occurrences of \( q_1 \) and \( k \) occurrences of \( q_2 \) for any nonnegative integers \( n \), \( k \), and \( \ell \).

Patterns 123 and 132 remain invariant under both the group-theoretic inverse and the reflection of the elements across the diagonal. This motivates the authors of \cite{KLv} to study the joint equidistributions of mesh patterns 123 and 132 with certain shadings. We say that a mesh pattern has an \textit{minus antipodal shading} if, for any \( i \neq j \), precisely one of \( (i,j) \) and \( (j,i) \) is shaded. The notion of a minus antipodal shading of a mesh pattern was first introduced in the context of multidimensional permutations in \cite{S-S-J}. More relevant to our paper, Kitaev and Lv \cite{KLv} investigated the joint equidistributions of the mesh patterns 123 and 132 with the same symmetric shadings. Lv and Zhang \cite{L-Z-joint} extended the study to the joint equidistributions of the mesh patterns 123 and 321 with the same symmetric and minus antipodal shadings.

In this paper, we establish 112 joint equidistributions for mesh patterns 123 and 132 with minus antipodal shadings. Table~\ref{tab-sy-shadings} presents a summary of our results, which are classified into 8 disjoint classes of shadings considered in Sections~\ref{sec-the first element}--\ref{sec-X_4}. An additional four pairs with other minus antipodal shadings, considered in Section~\ref{sec-other}, can be found in Table~\ref{tab-other-proved-shadings}. Finally, in Section \ref{sec-con}, we give 14 conjectured joint equidistributions as presented in Table~\ref{tab-conj}, and 562 relevant joint equidistributions, which are given by non-minus-antipodal and non-symmetric shadings for the mesh patterns, as listed in Table \ref{tab-extend}.

To illustrate the information presented in Table~\ref{tab-sy-shadings}, consider the pair of mesh patterns
$$q_1=\pattern{scale = 0.5}{3}{1/1,2/2,3/3}{0/0,0/1,0/2,0/3,2/1,2/2,3/1,3/2},\hspace{0.3cm}q_2=\pattern{scale = 0.5}{3}{1/1,2/3,3/2}{0/0,0/1,0/2,0/3,2/1,2/2,3/1,3/2}$$
obtained from type $X^{(1)}$ shading of the form $\raisebox{0.5ex}{\begin{tikzpicture}[scale = 0.2, baseline=(current bounding box.center)]
	\foreach \x/\y in {0/0,0/1,0/2,0/3}		
		\fill[gray!50] (\x,\y) rectangle +(1,1); 	
        \draw (0,1)--(4,1);
        \draw (1,0)--(1,4);
        \draw (0,2)--(1,2);
         \draw (0,3)--(1,3);
         \draw (2,0)--(2,1);
         \draw (3,0)--(3,1);
 			\node  at (2.6,2.6) {$p$};
 			\filldraw (1,1) circle (4pt) node[above left] {};
 			\end{tikzpicture}}$ by inserting, respectively, the mesh patterns $p_1=\hspace{-0.1cm}\pattern{scale = 0.5}{2}{1/1,2/2}{1/0,1/1,2/0,2/1}$ and $p_2=\hspace{-0.1cm}\pattern{scale = 0.5}{2}{1/2,2/1}{1/0,1/1,2/0,2/1}$ instead of $p$. In general, throughout the paper, we use $q_1$ and $q_2$ to represent mesh patterns 123 and 132, respectively, with the same shadings depending on the context. Moreover, we always let $p_1$ (resp., $p_2$) be the length-2 mesh pattern obtained from $q_1$ (resp., $q_2$) by removing the bottom row and leftmost column, as in the example above. Also, by abusing the notation, all of our bijective maps used to prove joint equidistributions in this paper are denoted by $f$. Finally, for a pair of mesh patterns $\{q_1,q_2\}$, we often omit the brackets and the comma. 
			
For the patterns in Sections~\ref{sec-the first element}--\ref{sec-X_4}, we note that only the results for type \( X^{(i)} \) shadings, for \( i \in \{1,2,3,4\} \), require justifications, since all other results follow from the fact that type \( Y^{(i)} \) shading is obtained from type \( X^{(i)} \) shading by applying the inverse operation. Hence, we omit type \( Y^{(i)} \) shadings from our considerations. Also, in this paper, we number pairs of mesh patterns 123 and 132 with the same type \( X^{(i)} \) (resp., \( Y^{(i)} \)) shading by \( X^{(i)}_j \) (resp., \( Y^{(i)}_j \)) for \( i \in \{1,2,3,4\} \) and \( j \geq 1 \), and as noted above, we only provide proofs for the mesh patterns with type \( X^{(i)}_j \) shadings. 

\begin{table}[!ht]
 	{
 		\renewcommand{\arraystretch}{1.7}
 		\setlength{\tabcolsep}{5pt}
 \begin{center} 
 		\begin{tabular}{|c |c | c| c|}
 			\hline
 		\footnotesize{Type}	 & \footnotesize{Form}	& {\footnotesize Mesh pattern $p$} & {\footnotesize Reference}  \\
 			\hline
 			\hline

\multirow{2.6}{0.7cm}{$X^{(1)}$}  & 
\multirow{2.6}{1.1cm}{\raisebox{0.5ex}{\begin{tikzpicture}[scale = 0.2, baseline=(current bounding box.center)]
	\foreach \x/\y in {0/0,0/1,0/2,0/3}		
		\fill[gray!50] (\x,\y) rectangle +(1,1); 	
        \draw (0,1)--(4,1);
        \draw (1,0)--(1,4);
        \draw (0,2)--(1,2);
         \draw (0,3)--(1,3);
         \draw (2,0)--(2,1);
         \draw (3,0)--(3,1);
 			\node  at (2.6,2.6) {$p$};
 			\filldraw (1,1) circle (4pt) node[above left] {};
 			\end{tikzpicture}}}
 & \footnotesize{horizontal or vertical symmetric shadings} &\footnotesize{Subsection~\ref{subsec-X_1-1}}
 \\[5pt]
\cline{3-4}
 &  &  $\pattern{scale = 0.5}{2}{}{1/0,1/1,2/0,2/1}$ $\pattern{scale = 0.5}{2}{}{0/1,0/2,1/1,1/2}$  $\pattern{scale = 0.5}{2}{}{0/0,0/2,1/0,1/1,1/2,2/2}$ $\pattern{scale = 0.5}{2}{}{0/0,1/0,1/1,1/2,2/0,2/2}$ $\pattern{scale = 0.5}{2}{}{0/0,0/1,0/2,1/1,2/1,2/2}$ $\pattern{scale = 0.5}{2}{}{0/0,0/1,1/1,2/0,2/1,2/2}$ &  \footnotesize{Subsection~\ref{subsec-2-2}}
 \\[5pt]
\hline
\multirow{2.6}{0.7cm}{$Y^{(1)}$} & 
\multirow{2.6}{1.1cm}{\raisebox{0.5ex}{\begin{tikzpicture}[scale = 0.2, baseline=(current bounding box.center)]
	\foreach \x/\y in {0/0,1/0,2/0,3/0}		
		\fill[gray!50] (\x,\y) rectangle +(1,1); 	
        \draw (0,1)--(4,1);
        \draw (1,0)--(1,4);
        \draw (0,2)--(1,2);
         \draw (0,3)--(1,3);
         \draw (2,0)--(2,1);
         \draw (3,0)--(3,1);
 			\node  at (2.6,2.6) {$p$};
 			\filldraw (1,1) circle (4pt) node[above left] {};
 			\end{tikzpicture}}}
 & \footnotesize{vertical or horizontal symmetric shadings} &\footnotesize{Subsection~\ref{subsec-X_1-1}}
 \\[5pt]
\cline{3-4}
 &  & $\pattern{scale = 0.5}{2}{}{0/1,0/2,1/1,1/2}$ 
 $\pattern{scale = 0.5}{2}{}{1/0,1/1,2/0,2/1}$ $\pattern{scale = 0.5}{2}{}{0/0,0/1,1/1,2/0,2/1,2/2}$  $\pattern{scale = 0.5}{2}{}{0/0,0/1,0/2,1/1,2/1,2/2}$ $\pattern{scale = 0.5}{2}{}{0/0,1/0,1/1,1/2,2/0,2/2}$ $\pattern{scale = 0.5}{2}{}{0/0,0/2,1/0,1/1,1/2,2/2}$& \footnotesize{Subsection~\ref{subsec-2-2}}
 \\[5pt]
\hline
\multirow{2.6}{0.7cm}{$X^{(2)}$}  & \multirow{2.6}{1.1cm}{\raisebox{0.6ex}{\begin{tikzpicture}[scale = 0.2, baseline=(current bounding box.center)]
	\foreach \x/\y in {0/0,0/1,0/2,3/0}		
		\fill[gray!50] (\x,\y) rectangle +(1,1); 	
        \draw (0,1)--(4,1);
        \draw (1,0)--(1,4);
        \draw (0,2)--(1,2);
         \draw (0,3)--(1,3);
         \draw (2,0)--(2,1);
         \draw (3,0)--(3,1);
 			\node  at (2.6,2.6) {$p$};
 			\filldraw (1,1) circle (4pt) node[above left] {};
 			\end{tikzpicture}}}
 & \footnotesize{horizontal symmetric shadings} &  \footnotesize{Subsection~\ref{subsec-X_2-1}}
 \\[5pt]
\cline{3-4}
&  & $\pattern{scale = 0.5}{2}{}{0/0,1/0,1/1,1/2,2/0}$ $\pattern{scale = 0.5}{2}{}{0/2,1/0,1/1,1/2,2/2}$ $\pattern{scale = 0.5}{2}{}{0/0,0/1,1/1,2/0,2/1}$ $\pattern{scale = 0.5}{2}{}{0/1,0/2,1/1,2/1,2/2}$ $\pattern{scale = 0.5}{2}{}{0/1,0/2,1/1,1/2}$  $\pattern{scale = 0.5}{2}{}{0/0,0/2,1/0,1/1,1/2,2/2}$ $\pattern{scale = 0.5}{2}{}{0/0,1/0,1/1,1/2,2/0,2/2}$ $\pattern{scale = 0.5}{2}{}{0/0,0/1,0/2,1/1,2/1,2/2}$ $\pattern{scale = 0.5}{2}{}{0/0,0/1,1/1,2/0,2/1,2/2}$ & \footnotesize{Subsection~\ref{subsec-X_2-2}}
 \\[5pt]
\hline
\multirow{2.6}{0.7cm}{$Y^{(2)}$} & \multirow{2.6}{1.1cm}{\raisebox{0.6ex}{\begin{tikzpicture}[scale = 0.2, baseline=(current bounding box.center)]
	\foreach \x/\y in {0/0,1/0,2/0,0/3}		
		\fill[gray!50] (\x,\y) rectangle +(1,1); 	
        \draw (0,1)--(4,1);
        \draw (1,0)--(1,4);
        \draw (0,2)--(1,2);
         \draw (0,3)--(1,3);
         \draw (2,0)--(2,1);
         \draw (3,0)--(3,1);
 			\node  at (2.6,2.6) {$p$};
 			\filldraw (1,1) circle (4pt) node[above left] {};
 			\end{tikzpicture}}}
 & \footnotesize{vertical symmetric shadings} & \footnotesize{Subsection~\ref{subsec-X_2-1}}
 \\[5pt]
\cline{3-4}
&  & $\pattern{scale = 0.5}{2}{}{0/0,0/1,0/2,1/1,2/1}$ $\pattern{scale = 0.5}{2}{}{0/1,1/1,2/0,2/1,2/2}$ $\pattern{scale = 0.5}{2}{}{0/0,0/2,1/0,1/1,1/2}$ $\pattern{scale = 0.5}{2}{}{1/0,1/1,1/2,2/0,2/2}$ $\pattern{scale = 0.5}{2}{}{1/0,1/1,2/0,2/1}$  $\pattern{scale = 0.5}{2}{}{0/0,0/1,1/1,2/0,2/1,2/2}$  $\pattern{scale = 0.5}{2}{}{0/0,0/1,0/2,1/1,2/1,2/2}$ $\pattern{scale = 0.5}{2}{}{0/0,1/0,1/1,1/2,2/0,2/2}$ $\pattern{scale = 0.5}{2}{}{0/0,0/2,1/0,1/1,1/2,2/2}$& \footnotesize{Subsection~\ref{subsec-X_2-2}} 
 \\[5pt]
\hline
\multirow{2.6}{0.7cm}{$X^{(3)}$} & 
\multirow{2.6}{1.1cm}{\raisebox{0.5ex}{\begin{tikzpicture}[scale = 0.2, baseline=(current bounding box.center)]
	\foreach \x/\y in {0/0,0/1,0/3,2/0}		
		\fill[gray!50] (\x,\y) rectangle +(1,1); 	
        \draw (0,1)--(4,1);
        \draw (1,0)--(1,4);
        \draw (0,2)--(1,2);
         \draw (0,3)--(1,3);
         \draw (2,0)--(2,1);
         \draw (3,0)--(3,1);
 			\node  at (2.6,2.6) {$p$};
 			\filldraw (1,1) circle (4pt) node[above left] {};
 			\end{tikzpicture}}}
 & \footnotesize{ vertical symmetric shadings} & \footnotesize{Subsection~\ref{subsec-X_3-1}}
 \\[5pt]
\cline{3-4}
&  & $\pattern{scale = 0.5}{2}{}{0/0,0/2,1/0,1/1,1/2,2/2}$  & \footnotesize{Subsection~\ref{subsec-X_3-2}}
 \\[5pt]
\hline
\multirow{2.6}{0.7cm}{$Y^{(3)}$}& 
\multirow{2.6}{1.1cm}{\raisebox{0.5ex}{\begin{tikzpicture}[scale = 0.2, baseline=(current bounding box.center)]
	\foreach \x/\y in {0/0,1/0,3/0,0/2}		
		\fill[gray!50] (\x,\y) rectangle +(1,1); 	
        \draw (0,1)--(4,1);
        \draw (1,0)--(1,4);
        \draw (0,2)--(1,2);
         \draw (0,3)--(1,3);
         \draw (2,0)--(2,1);
         \draw (3,0)--(3,1);
 			\node  at (2.6,2.6) {$p$};
 			\filldraw (1,1) circle (4pt) node[above left] {};
 			\end{tikzpicture}}}
 & \footnotesize{horizontal symmetric shadings} & \footnotesize{Subsection~\ref{subsec-X_3-1}}
 \\[5pt]
\cline{3-4}
&  &$\pattern{scale = 0.5}{2}{}{0/0,0/1,1/1,2/0,2/1,2/2}$ & \footnotesize{Subsection~\ref{subsec-X_3-2}}
 \\[5pt]
\hline
\multirow{2.6}{0.7cm}{$X^{(4)}$}  &
\multirow{2.6}{1.1cm}{\raisebox{0.6ex}{\begin{tikzpicture}[scale = 0.2, baseline=(current bounding box.center)]
	\foreach \x/\y in {0/0,0/1,2/0,3/0}		
		\fill[gray!50] (\x,\y) rectangle +(1,1); 	
        \draw (0,1)--(4,1);
        \draw (1,0)--(1,4);
        \draw (0,2)--(1,2);
         \draw (0,3)--(1,3);
         \draw (2,0)--(2,1);
         \draw (3,0)--(3,1);
 			\node  at (2.6,2.6) {$p$};
 			\filldraw (1,1) circle (4pt) node[above left] {};
 			\end{tikzpicture}}}
 &  $\pattern{scale = 0.5}{2}{}{0/2,1/0,1/1,1/2,2/2}$ $\pattern{scale = 0.5}{2}{}{0/1,1/1,2/0,2/1,2/2}$   $\pattern{scale = 0.5}{2}{}{0/0,0/2,1/0,1/1,1/2,2/2}$ $\pattern{scale = 0.5}{2}{}{0/0,1/0,1/1,1/2,2/0,2/2}$ $\pattern{scale = 0.5}{2}{}{0/0,0/1,1/1,2/0,2/1,2/2}$ &
 \footnotesize{Subsection~\ref{subsec-5.1}}  \\[5pt]
\cline{3-4}
& & $\pattern{scale = 0.5}{2}{}{1/0,1/1,2/0,2/1}$ & \footnotesize{Subsection~\ref{subsec-X_4-2}} 
 \\[5pt]
\hline
\multirow{2.6}{0.7cm}{$Y^{(4)}$}  &
\multirow{2.6}{1.1
cm}{\raisebox{0.6ex}{\begin{tikzpicture}[scale = 0.2, baseline=(current bounding box.center)]
	\foreach \x/\y in {0/0,1/0,0/2,0/3}		
		\fill[gray!50] (\x,\y) rectangle +(1,1); 	
        \draw (0,1)--(4,1);
        \draw (1,0)--(1,4);
        \draw (0,2)--(1,2);
         \draw (0,3)--(1,3);
         \draw (2,0)--(2,1);
         \draw (3,0)--(3,1);
 			\node  at (2.6,2.6) {$p$};
 			\filldraw (1,1) circle (4pt) node[above left] {};
 			\end{tikzpicture}}}
 &   $\pattern{scale = 0.5}{2}{}{0/1,1/1,2/0,2/1,2/2}$ $\pattern{scale = 0.5}{2}{}{0/2,1/0,1/1,1/2,2/2}$ $\pattern{scale = 0.5}{2}{}{0/0,0/1,1/1,2/0,2/1,2/2}$ $\pattern{scale = 0.5}{2}{}{0/0,0/1,0/2,1/1,2/1,2/2}$  $\pattern{scale = 0.5}{2}{}{0/0,0/2,1/0,1/1,1/2,2/2}$&
 \footnotesize{Subsection~\ref{subsec-5.1}}
 \\[5pt]
\cline{3-4}
& & $\pattern{scale = 0.5}{2}{}{0/1,0/2,1/1,1/2}$ & \footnotesize{Subsection~\ref{subsec-X_4-2}}  
 \\[5pt]
\hline
	\end{tabular}
\end{center} 
}
\vspace{-0.5cm}
 	\caption{Shadings of mesh patterns 123 and 132 for which joint equidistributions  are established in this paper.}\label{tab-sy-shadings}
\end{table}

Let $\T(n,k,\ell)$ be the set of $n$-permutations with $k$ occurrences of $q_1$ and $\ell$ occurrences of $q_2$ in questions. We define $T_{n,k,\ell}=|\T(n,k,\ell)|$ and the generating function
\[T_n(x,y)=\sum_{k=0}^{n-2}\sum_{\ell=0}^{n-2}T_{n,k,\ell}\, x^k \,y^\ell.\]
Analogously, define the set $\mathcal{H}(n,k,\ell)$, $H_{n,k,\ell}=|\mathcal{H}(n,k,\ell)|$, and 
\[H_n(x,y)=\sum_{k=0}^{n-1}\sum_{\ell=0}^{n-1}H_{n,k,\ell}\, x^k \,y^\ell,\]
for the pair $\{p_1,\hspace{0.1cm}p_2\}$.

\section{Mesh patterns of types $X^{(1)}$ and $Y^{(1)}$}\label{sec-the first element}
In this section, we consider the patterns of types $X^{(1)}=\raisebox{0.5ex}{\begin{tikzpicture}[scale = 0.2, baseline=(current bounding box.center)]
	\foreach \x/\y in {0/0,0/1,0/2,0/3}		
		\fill[gray!50] (\x,\y) rectangle +(1,1); 	
        \draw (0,1)--(4,1);
        \draw (1,0)--(1,4);
        \draw (0,2)--(1,2);
         \draw (0,3)--(1,3);
         \draw (2,0)--(2,1);
         \draw (3,0)--(3,1);
 			\node  at (2.6,2.6) {$p$};
 			\filldraw (1,1) circle (4pt) node[above left] {};
 			\end{tikzpicture}}$ and $Y^{(1)}=\raisebox{0.5ex}{\begin{tikzpicture}[scale = 0.2, baseline=(current bounding box.center)]
	\foreach \x/\y in {0/0,1/0,2/0,3/0}		
		\fill[gray!50] (\x,\y) rectangle +(1,1); 	
        \draw (0,1)--(4,1);
        \draw (1,0)--(1,4);
        \draw (0,2)--(1,2);
         \draw (0,3)--(1,3);
         \draw (2,0)--(2,1);
         \draw (3,0)--(3,1);
 			\node  at (2.6,2.6) {$p$};
 			\filldraw (1,1) circle (4pt) node[above left] {};
 			\end{tikzpicture}}$.   The following result allows us to reduce considerations to length-2 mesh patterns in this case. 
			
\begin{lem}\label{reduction-lemma} For type  $X^{(1)}$ patterns, we have $q_1\sim_{jd}q_2$ if and only if $p_1\sim_{jd}p_2$. \end{lem}

\begin{proof}
Let $\pi=\pi_1\cdots\pi_n\in S_n$, $A=\{\pi_i:\pi_i>\pi_1, 2\leq i \leq n\}$ and $B=\{\pi_i:\pi_i<\pi_1, 2\leq i \leq n\}$. Each occurrence of $q_1$ (resp., $q_2$), if any, must begin with $\pi_1$, and $p_1$ (resp., $p_2$) will involve the elements in $A$, and no elements in $B$. But then it is clear that the joint equidistribution of $q_1$ and $q_2$ is equivalent to the joint equidistribution of $p_1$ and $p_2$, which completes our proof. 
\end{proof}

\subsection{Cases proved via complementation or reversal}\label{subsec-X_1-1}
This subsection is about the 32 pairs presented in Table~\ref{tab-X_1-1}. Each pair \( \{p_1, p_2\} \) associated with the mesh pattern \( X^{(1)}_j \), where \( j \in \{1, \ldots, 16\} \), exhibits symmetry either horizontally or vertically. Therefore, the map \( f \) that gives the joint equidistributions can be constructed by applying the complement or reverse operations. Moreover, the same joint distribution of the pairs in each row of Table~\ref{tab-X_1-1} follows from the complement, reverse, and inverse operations. Here, we demonstrate the operations for the pairs $\{X^{(1)}_i\}_{i=1}^{4}$ as an example, and the remaining three sets of pairs of mesh patterns are treated in the same manner:
\[
\pattern{scale = 0.5}{2}{1/1,2/2}{0/0,0/1,0/2,1/1,2/1} \pattern{scale = 0.5}{2}{1/2,2/1}{0/0,0/1,0/2,1/1,2/1} \xrightarrow{r} 
\pattern{scale = 0.5}{2}{1/2,2/1}{0/1,1/1,2/0,2/1,2/2} \pattern{scale = 0.5}{2}{1/1,2/2}{0/1,1/1,2/0,2/1,2/2} \xrightarrow{i} 
\pattern{scale = 0.5}{2}{1/2,2/1}{0/2,1/0,1/2,1/1,2/2} \pattern{scale = 0.5}{2}{1/1,2/2}{0/2,1/0,1/2,1/1,2/2} \xrightarrow{c} 
\pattern{scale = 0.5}{2}{1/1,2/2}{0/0,1/0,1/1,1/2,2/0} \pattern{scale = 0.5}{2}{1/2,2/1}{0/0,1/0,1/1,1/2,2/0}.
\]

\begin{table}[!ht]
 	{
 		\renewcommand{\arraystretch}{1.7}
 		\setlength{\tabcolsep}{5pt}
 \begin{center} 
 		\begin{tabular}{c | c || c | c || c | c || c | c}
 			\hline
 		\footnotesize{Nr.}	& \footnotesize {Patterns} & \footnotesize{Nr.}	& \footnotesize{Patterns} & \footnotesize {Nr.}  & \footnotesize{Patterns} &\footnotesize {Nr.}  & \footnotesize{Patterns}\\
 			\hline
 			\hline
$X^{(1)}_{1}$ & $\pattern{scale = 0.5}{3}{1/1,2/2,3/3}{0/0,0/1,0/2,0/3,1/1,1/2,1/3,2/2,3/2}\pattern{scale = 0.5}{3}{1/1,2/3,3/2}{0/0,0/1,0/2,0/3,1/1,1/2,1/3,2/2,3/2}$ &
$X^{(1)}_2$ & $\pattern{scale = 0.5}{3}{1/1,2/2,3/3}{0/0,0/1,0/2,0/3,1/2,2/2,3/1,3/2,3/3}\pattern{scale = 0.5}{3}{1/1,2/3,3/2}{0/0,0/1,0/2,0/3,1/2,2/2,3/1,3/2,3/3}$& 
$X^{(1)}_3$& $\pattern{scale = 0.5}{3}{1/1,2/2,3/3}{0/0,0/1,0/2,0/3,1/1,2/1,2/2,2/3,3/1}\pattern{scale = 0.5}{3}{1/1,2/3,3/2}{0/0,0/1,0/2,0/3,1/1,2/1,2/2,2/3,3/1}$ & 
$X^{(1)}_4$& $\pattern{scale = 0.5}{3}{1/1,2/2,3/3}{0/0,0/1,0/2,0/3,1/3,2/1,2/2,2/3,3/3}\pattern{scale = 0.5}{3}{1/1,2/3,3/2}{0/0,0/1,0/2,0/3,1/3,2/1,2/2,2/3,3/3}$ \\[5pt]
$X^{(1)}_5$ & $\pattern{scale = 0.5}{3}{1/1,2/2,3/3}{0/0,0/1,0/2,0/3,1/1,1/2,2/2,3/1,3/2}\pattern{scale = 0.5}{3}{1/1,2/3,3/2}{0/0,0/1,0/2,0/3,1/1,1/2,2/2,3/1,3/2}$ & 
$X^{(1)}_6$ & $\pattern{scale = 0.5}{3}{1/1,2/2,3/3}{0/0,0/1,0/2,0/3,1/2,1/3,2/2,3/2,3/3}\pattern{scale = 0.5}{3}{1/1,2/3,3/2}{0/0,0/1,0/2,0/3,1/2,1/3,2/2,3/2,3/3}$ &
$X^{(1)}_7$ & $\pattern{scale = 0.5}{3}{1/1,2/2,3/3}{0/0,0/1,0/2,0/3,1/1,1/3,2/1,2/2,2/3}\pattern{scale = 0.5}{3}{1/1,2/3,3/2}{0/0,0/1,0/2,0/3,1/1,1/3,2/1,2/2,2/3}$ & 
$X^{(1)}_8$ & $\pattern{scale = 0.5}{3}{1/1,2/2,3/3}{0/0,0/1,0/2,0/3,2/1,2/2,2/3,3/1,3/3}\pattern{scale = 0.5}{3}{1/1,2/3,3/2}{0/0,0/1,0/2,0/3,2/1,2/2,2/3,3/1,3/3}$ \\[5pt]
$X^{(1)}_9$ & $\pattern{scale = 0.5}{3}{1/1,2/2,3/3}{0/0,0/1,0/2,0/3,1/1,1/2,3/1,3/2}\pattern{scale = 0.5}{3}{1/1,2/3,3/2}{0/0,0/1,0/2,0/3,1/1,1/2,3/1,3/2}$ &
$X^{(1)}_{10}$ & $\pattern{scale = 0.5}{3}{1/1,2/2,3/3}{0/0,0/1,0/2,0/3,1/2,1/3,3/2,3/3}\pattern{scale = 0.5}{3}{1/1,2/3,3/2}{0/0,0/1,0/2,0/3,1/2,1/3,3/2,3/3}$ &
$X^{(1)}_{11}$ & $\pattern{scale = 0.5}{3}{1/1,2/2,3/3}{0/0,0/1,0/2,0/3,1/1,1/3,2/1,2/3}\pattern{scale = 0.5}{3}{1/1,2/3,3/2}{0/0,0/1,0/2,0/3,1/1,1/3,2/1,2/3}$ &
$X^{(1)}_{12}$ & $\pattern{scale = 0.5}{3}{1/1,2/2,3/3}{0/0,0/1,0/2,0/3,2/1,2/3,3/1,3/3}\pattern{scale = 0.5}{3}{1/1,2/3,3/2}{0/0,0/1,0/2,0/3,2/1,2/3,3/1,3/3}$ \\[5pt]
$X^{(1)}_{13}$ & $\pattern{scale = 0.5}{3}{1/1,2/2,3/3}{0/0,0/1,0/2,0/3,1/1,1/2,1/3,3/2}\pattern{scale = 0.5}{3}{1/1,2/3,3/2}{0/0,0/1,0/2,0/3,1/1,1/2,1/3,3/2}$ &
$X^{(1)}_{14}$ & $\pattern{scale = 0.5}{3}{1/1,2/2,3/3}{0/0,0/1,0/2,0/3,1/2,3/1,3/2,3/3}\pattern{scale = 0.5}{3}{1/1,2/3,3/2}{0/0,0/1,0/2,0/3,1/2,3/1,3/2,3/3}$ &
$X^{(1)}_{15}$ & $\pattern{scale = 0.5}{3}{1/1,2/2,3/3}{0/0,0/1,0/2,0/3,1/1,2/1,2/3,3/1}\pattern{scale = 0.5}{3}{1/1,2/3,3/2}{0/0,0/1,0/2,0/3,1/1,2/1,2/3,3/1}$ &
$X^{(1)}_{16}$ & $\pattern{scale = 0.5}{3}{1/1,2/2,3/3}{0/0,0/1,0/2,0/3,1/3,2/1,2/3,3/3}\pattern{scale = 0.5}{3}{1/1,2/3,3/2}{0/0,0/1,0/2,0/3,1/3,2/1,2/3,3/3}$ \\[5pt]
\hline
$Y^{(1)}_1$ & $\pattern{scale = 0.5}{3}{1/1,2/2,3/3}{0/0,1/0,1/1,2/0,2/1,2/2,2/3,3/0,3/1}\pattern{scale = 0.5}{3}{1/1,2/3,3/2}{0/0,1/0,1/1,2/0,2/1,2/2,2/3,3/0,3/1}$ &
 $Y^{(1)}_2$ & $\pattern{scale = 0.5}{3}{1/1,2/2,3/3}{0/0,1/0,1/3,2/0,2/1,2/2,2/3,3/0,3/3}\pattern{scale = 0.5}{3}{1/1,2/3,3/2}{0/0,1/0,1/3,2/0,2/1,2/2,2/3,3/0,3/3}$&
$Y^{(1)}_3$ & $\pattern{scale = 0.5}{3}{1/1,2/2,3/3}{0/0,1/0,1/1,1/2,1/3,2/0,2/2,3/0,3/2}\pattern{scale = 0.5}{3}{1/1,2/3,3/2}{0/0,1/0,1/1,1/2,1/3,2/0,2/2,3/0,3/2}$ & 
$Y^{(1)}_4$ & $\pattern{scale = 0.5}{3}{1/1,2/2,3/3}{0/0,1/0,1/2,2/0,2/2,3/0,3/3,3/1,3/2}\pattern{scale = 0.5}{3}{1/1,2/3,3/2}{0/0,1/0,1/2,2/0,2/2,3/0,3/3,3/1,3/2}$
\\[5pt]
$Y^{(1)}_5$ & $\pattern{scale = 0.5}{3}{1/1,2/2,3/3}{0/0,1/0,1/1,1/3,2/0,2/1,2/2,2/3,3/0}\pattern{scale = 0.5}{3}{1/1,2/3,3/2}{0/0,1/0,1/1,1/3,2/0,2/1,2/2,2/3,3/0}$&
$Y^{(1)}_6$ & $\pattern{scale = 0.5}{3}{1/1,2/2,3/3}{0/0,1/0,2/0,2/1,2/2,2/3,3/0,3/1,3/3}\pattern{scale = 0.5}{3}{1/1,2/3,3/2}{0/0,1/0,2/0,2/1,2/2,2/3,3/0,3/1,3/3}$ &
$Y^{(1)}_7$ & $\pattern{scale = 0.5}{3}{1/1,2/2,3/3}{0/0,1/0,1/1,1/2,2/0,2/2,3/0,3/1,3/2}\pattern{scale = 0.5}{3}{1/1,2/3,3/2}{0/0,1/0,1/1,1/2,2/0,2/2,3/0,3/1,3/2}$& 
$Y^{(1)}_8$ & $\pattern{scale = 0.5}{3}{1/1,2/2,3/3}{0/0,1/0,1/2,1/3,2/0,2/2,3/0,3/3,3/2}\pattern{scale = 0.5}{3}{1/1,2/3,3/2}{0/0,1/0,1/2,1/3,2/0,2/2,3/0,3/3,3/2}$
\\[5pt]
$Y^{(1)}_9$ & $\pattern{scale = 0.5}{3}{1/1,2/2,3/3}{0/0,1/0,1/1,1/3,2/0,2/1,2/3,3/0}\pattern{scale = 0.5}{3}{1/1,2/3,3/2}{0/0,1/0,1/1,1/3,2/0,2/1,2/3,3/0}$ &
$Y^{(1)}_{10}$ & $\pattern{scale = 0.5}{3}{1/1,2/2,3/3}{0/0,1/0,2/0,2/1,2/3,3/0,3/1,3/3}\pattern{scale = 0.5}{3}{1/1,2/3,3/2}{0/0,1/0,2/0,2/1,2/3,3/0,3/1,3/3}$ &
$Y^{(1)}_{11}$ & $\pattern{scale = 0.5}{3}{1/1,2/2,3/3}{0/0,1/0,1/1,1/2,2/0,3/0,3/1,3/2}\pattern{scale = 0.5}{3}{1/1,2/3,3/2}{0/0,1/0,1/1,1/2,2/0,3/0,3/1,3/2}$& 
$Y^{(1)}_{12}$ & $\pattern{scale = 0.5}{3}{1/1,2/2,3/3}{0/0,1/0,1/2,1/3,2/0,3/0,3/2,3/3}\pattern{scale = 0.5}{3}{1/1,2/3,3/2}{0/0,1/0,1/2,1/3,2/0,3/0,3/2,3/3}$ \\[5pt]
$Y^{(1)}_{13}$ & $\pattern{scale = 0.5}{3}{1/1,2/2,3/3}{0/0,1/0,1/1,2/0,2/1,2/3,3/0,3/1}\pattern{scale = 0.5}{3}{1/1,2/3,3/2}{0/0,1/0,1/1,2/0,2/1,2/3,3/0,3/1}$ &
$Y^{(1)}_{14}$ & $\pattern{scale = 0.5}{3}{1/1,2/2,3/3}{0/0,1/0,1/3,2/0,2/1,2/3,3/0,3/3}\pattern{scale = 0.5}{3}{1/1,2/3,3/2}{0/0,1/0,1/3,2/0,2/1,2/3,3/0,3/3}$ &
$Y^{(1)}_{15}$ & $\pattern{scale = 0.5}{3}{1/1,2/2,3/3}{0/0,1/0,1/1,1/2,1/3,2/0,3/0,3/2}\pattern{scale = 0.5}{3}{1/1,2/3,3/2}{0/0,1/0,1/1,1/2,1/3,2/0,3/0,3/2}$ &
$Y^{(1)}_{16}$ & $\pattern{scale = 0.5}{3}{1/1,2/2,3/3}{0/0,1/0,1/2,2/0,3/0,3/1,3/2,3/3}\pattern{scale = 0.5}{3}{1/1,2/3,3/2}{0/0,1/0,1/2,2/0,3/0,3/1,3/2,3/3}$
 \\[5pt]
\hline
	\end{tabular}
\end{center} 
}
\vspace{-0.5cm}
 	\caption{Jointly equidistributed patterns of types of $X^{(1)}$ and $Y^{(1)}$ via complementation or reversal.}\label{tab-X_1-1}
\end{table}

\subsection{Cases proved via swapping elements}\label{subsec-2-2}
This subsection is about the 12 pairs presented in Table~\ref{tab-X_1-2}.

\begin{table}[!ht]
 	{
 		\renewcommand{\arraystretch}{1.7}
 		\setlength{\tabcolsep}{5pt}
 \begin{center} 
 		\begin{tabular}{c | c || c | c || c | c || c | c}
 			\hline
 		\footnotesize{Nr.}	& \footnotesize {Patterns} & \footnotesize{Nr.}	& \footnotesize{Patterns} & \footnotesize {Nr.}  & \footnotesize{Patterns} &\footnotesize {Nr.}  & \footnotesize{Patterns}\\
 			\hline
 			\hline
$X^{(1)}_{17}$ & $\pattern{scale = 0.5}{3}{1/1,2/2,3/3}{0/0,0/1,0/2,0/3,2/1,2/2,3/1,3/2}\pattern{scale = 0.5}{3}{1/1,2/3,3/2}{0/0,0/1,0/2,0/3,2/1,2/2,3/1,3/2}$ &
$X^{(1)}_{18}$ & $\pattern{scale = 0.5}{3}{1/1,2/2,3/3}{0/0,0/1,0/2,0/3,1/2,1/3,2/2,2/3}\pattern{scale = 0.5}{3}{1/1,2/3,3/2}{0/0,0/1,0/2,0/3,1/2,1/3,2/2,2/3}$& 
$X^{(1)}_{19}$ & $\pattern{scale = 0.5}{3}{1/1,2/2,3/3}{0/0,0/1,0/2,0/3,1/1,1/3,2/1,2/2,2/3,3/3}\pattern{scale = 0.5}{3}{1/1,2/3,3/2}{0/0,0/1,0/2,0/3,1/1,1/3,2/1,2/2,2/3,3/3}$&
$X^{(1)}_{20}$ & $\pattern{scale = 0.5}{3}{1/1,2/2,3/3}{0/0,0/1,0/2,0/3,1/1,2/1,2/2,2/3,3/1,3/3}\pattern{scale = 0.5}{3}{1/1,2/3,3/2}{0/0,0/1,0/2,0/3,1/1,2/1,2/2,2/3,3/1,3/3}$\\[5pt]
$X^{(1)}_{21}$ & $\pattern{scale = 0.5}{3}{1/1,2/2,3/3}{0/0,0/1,0/2,0/3,1/1,1/2,1/3,2/2,3/2,3/3}\pattern{scale = 0.5}{3}{1/1,2/3,3/2}{0/0,0/1,0/2,0/3,1/1,1/2,1/3,2/2,3/2,3/3}$ &
$X^{(1)}_{22}$ & $\pattern{scale = 0.5}{3}{1/1,2/2,3/3}{0/0,0/1,0/2,0/3,1/1,1/2,2/2,3/1,3/2,3/3}\pattern{scale = 0.5}{3}{1/1,2/3,3/2}{0/0,0/1,0/2,0/3,1/1,1/2,2/2,3/1,3/2,3/3}$&& &
\\[5pt]
\hline
$Y^{(1)}_{17}$& $\pattern{scale = 0.5}{3}{1/1,2/2,3/3}{0/0,1/0,1/2,1/3,2/0,2/2,2/3,3/0}\pattern{scale = 0.5}{3}{1/1,2/3,3/2}{0/0,1/0,1/2,1/3,2/0,2/2,2/3,3/0}$ & 
$Y^{(1)}_{18}$& $\pattern{scale = 0.5}{3}{1/1,2/2,3/3}{0/0,1/0,2/0,2/1,2/2,3/0,3/1,3/2}\pattern{scale = 0.5}{3}{1/1,2/3,3/2}{0/0,1/0,2/0,2/1,2/2,3/0,3/1,3/2}$&
$Y^{(1)}_{19}$ &$\pattern{scale = 0.5}{3}{1/1,2/2,3/3}{0/0,1/0,1/1,1/2,2/0,2/2,3/0,3/1,3/2,3/3}\pattern{scale = 0.5}{3}{1/1,2/3,3/2}{0/0,1/0,1/1,1/2,2/0,2/2,3/0,3/1,3/2,3/3}$ & 
$Y^{(1)}_{20}$ & $\pattern{scale = 0.5}{3}{1/1,2/2,3/3}{0/0,1/0,1/1,1/2,1/3,2/0,2/2,3/2,3/0,3/3}\pattern{scale = 0.5}{3}{1/1,2/3,3/2}{0/0,1/0,1/1,1/2,1/3,2/0,2/2,3/2,3/0,3/3}$  \\[5pt]
$Y^{(1)}_{21}$ & $\pattern{scale = 0.5}{3}{1/1,2/2,3/3}{0/0,1/0,1/1,2/1,2/0,2/2,2/3,3/0,3/1,3/3}\pattern{scale = 0.5}{3}{1/1,2/3,3/2}{0/0,1/0,1/1,2/1,2/0,2/2,2/3,3/0,3/1,3/3}$ & 
$Y^{(1)}_{22}$ & $\pattern{scale = 0.5}{3}{1/1,2/2,3/3}{0/0,1/0,1/1,1/3,2/0,2/1,2/2,2/3,3/0,3/3}\pattern{scale = 0.5}{3}{1/1,2/3,3/2}{0/0,1/0,1/1,1/3,2/0,2/1,2/2,2/3,3/0,3/3}$ && & \\[5pt]
\hline
	\end{tabular}
\end{center} 
}
\vspace{-0.5cm}
 	\caption{Jointly equidistributed patterns of types $X^{(1)}$ and $Y^{(1)}$ via swapping elements.}\label{tab-X_1-2}
\end{table}

\begin{thm}[\cite{KLv}]\label{thm-KL-1}
    We have $\pattern{scale = 0.5}{2}{1/1,2/2}{0/1,0/0,1/1,1/0}\sim_{jd}\hspace{-1mm}\pattern{scale = 0.5}{2}{1/2,2/1}{0/1,0/0,1/1,1/0}$.
\end{thm}

\begin{thm}[Pairs $X^{(1)}_{17}, X^{(1)}_{18}$, $Y^{(1)}_{17}, Y^{(1)}_{18}$]\label{thm-box-all}
We have $\pattern{scale = 0.5}{3}{1/1,2/2,3/3}{0/0,0/1,0/2,0/3,2/1,2/2,3/1,3/2}\sim_{jd}\hspace{-1mm}\pattern{scale = 0.5}{3}{1/1,2/3,3/2}{0/0,0/1,0/2,0/3,2/1,2/2,3/1,3/2}$, $\pattern{scale = 0.5}{3}{1/1,2/2,3/3}{0/0,0/1,0/2,0/3,1/2,1/3,2/2,2/3}\sim_{jd}\hspace{-1mm}\pattern{scale = 0.5}{3}{1/1,2/3,3/2}{0/0,0/1,0/2,0/3,1/2,1/3,2/2,2/3}$, $\pattern{scale = 0.5}{3}{1/1,2/2,3/3}{0/0,1/0,1/2,1/3,2/0,2/2,2/3,3/0}\sim_{jd}\hspace{-1mm}\pattern{scale = 0.5}{3}{1/1,2/3,3/2}{0/0,1/0,1/2,1/3,2/0,2/2,2/3,3/0}$, and $\pattern{scale = 0.5}{3}{1/1,2/2,3/3}{0/0,1/0,2/0,2/1,2/2,3/0,3/1,3/2}\sim_{jd}\hspace{-1mm}\pattern{scale = 0.5}{3}{1/1,2/3,3/2}{0/0,1/0,2/0,2/1,2/2,3/0,3/1,3/2}$. Moreover, these four pairs have the same joint distribution.
\end{thm}

\begin{proof}
Applying the reverse and complement operations to the patterns in Theorem~\ref{thm-KL-1}, we obtain $\pattern{scale = 0.5}{2}{1/1,2/2}{1/0,1/1,2/0,2/1}\sim_{jd}\hspace{-1mm}\pattern{scale = 0.5}{2}{1/2,2/1}{1/0,1/1,2/0,2/1}$ and $\pattern{scale = 0.5}{2}{1/1,2/2}{0/1,0/2,1/1,1/2}\sim_{jd}\hspace{-1mm}\pattern{scale = 0.5}{2}{1/2,2/1}{0/1,0/2,1/1,1/2}$. Hence, by Lemma~\ref{reduction-lemma},  the mesh patterns in the pairs $X_{17}^{(1)}$ and  $X_{18}^{(1)}$ are jointly equidistributed, and the pairs have the same joint distribution, which completes our proof. \end{proof}

The computer also suggests that pairs $\{X^{(1)}_i\}_{i=9}^{18}$ and $\{Y^{(1)}_i\}_{i=9}^{18}$ have the same joint distribution. Theorem~\ref{coro-pairs-17-36} below  will prove this fact by deriving recurrence relations and generating functions for the joint distributions.

Recall that  the {\em unsigned Stirling number of the first kind} (\cite[A132393]{OEIS}) is defined by $c(n,k)=0$ if $n<k$ or $k=0$, except $c(0,0)=1$, and 
\[
c(n,k)=(n-1)c(n-1,k)+c(n-1,k-1).
\]

\begin{thm}\label{coro-pairs-17-36}
The pairs $\{X^{(1)}_i\}_{i=9}^{18}$ and $\{Y^{(1)}_i\}_{i=9}^{18}$ have the same joint distribution. 
Moreover, all pairs satisfy 
\begin{align}\label{recur-pair 8-17}
T_{n,k,\ell}=\left\{  
\begin{array}{ll}  
(n-1)!\,\binom{k+\ell}{k}\,\sum_{i=2}^{n-1}\,\frac{1}{i!}\,c(i-1,k+\ell), & \text{} k \neq 0 \text{ or } \ell \neq 0, \\[6pt]
2(n-1)!, & \text{} k=\ell=0.
\end{array}  
\right.
\end{align}
Equivalently,
\begin{align}\label{genera-pair 8 17}
T_n(x,y)=(n-1)!\sum_{k=1}^{n-1}\sum_{\ell=1}^{n-1}\sum_{i=2}^{n-1}\,\frac{1}{i!}\,\binom{k+\ell}{k}\,c(i-1,k+\ell)\,x^k\,y^\ell+2(n-1)!.
\end{align}
\end{thm}
By Subsection~\ref{subsec-X_1-1} and Theorem~\ref{thm-box-all}, it suffices to verify that the pairs $X^{(1)}_{9}$, $X^{(1)}_{13}$ and $X^{(1)}_{17}$ satisfy~\eqref{recur-pair 8-17} and~\eqref{genera-pair 8 17}. To achieve this, we will consider length-2 pairs corresponding to these three pairs in the next lemma.

\begin{lem}\label{lem-pairs-17-36}
    The pairs 
    $\{\hspace{-0.1cm}\pattern{scale = 0.5}{2}{1/1,2/2}{0/0,0/1,2/0,2/1},\hspace{-0.1cm}\pattern{scale = 0.5}{2}{1/2,2/1}{0/0,0/1,2/0,2/1}\hspace{0.1cm}\}$, 
    $\{\hspace{-0.1cm}\pattern{scale = 0.5}{2}{1/1,2/2}{0/0,0/1,0/2,2/1},\hspace{-0.1cm}\pattern{scale = 0.5}{2}{1/2,2/1}{0/0,0/1,0/2,2/1}\}$, and $\{\hspace{-0.1cm}\pattern{scale = 0.5}{2}{1/1,2/2}{1/0,1/1,2/0,2/1},\hspace{-0.1cm}\pattern{scale = 0.5}{2}{1/2,2/1}{1/0,1/1,2/0,2/1}\hspace{0.1cm}\}$ satisfy
\begin{align}\label{recur-pairs-len 2 8-17}
H_{n,k,\ell}=\binom{k+\ell}{k}\,c(n-1,k+\ell).
\end{align}
Equivalently,
\begin{align}\label{genera-pairs-len 2 8-17}
H_n(x,y)=\sum_{k=0}^{n-1}\sum_{\ell=0}^{n-1}\binom{k+\ell}{k}\,c(n-1,k+\ell)\,x^k\,y^\ell.
\end{align}
\end{lem}
\begin{proof} \ 

\begin{enumerate}
 \item[(i)] Consider the pair $\{p_1\hspace{-0.1cm}=\hspace{-0.1cm}\pattern{scale = 0.5}{2}{1/1,2/2}{0/0,0/1,2/0,2/1},\hspace{0.1cm}p_2\hspace{-0.1cm}=\hspace{-0.1cm}\pattern{scale = 0.5}{2}{1/2,2/1}{0/0,0/1,2/0,2/1}\}$. We obtain each $n$-permutation $\pi=\pi\cdots \pi_n$ counted by the left-hand side of the equation~\eqref{genera-pairs-len 2 8-17} by inserting the largest element $n$ into an  $(n-1)$-permutation $\pi'$. Specifically, inserting $n$ immediately before the first entry creates an additional occurrence of $p_2$, in the form of $n \pi_n$, which contributes to $y\,H_{n-1}(x,y)$. Conversely, inserting $n$ immediately after the last entry creates an extra occurrence of $p_1$ with $\pi_1 n$, which is counted by $x\,H_{n-1}(x,y)$. In any other slot, the insertion preserves the number of occurrences of both patterns, thereby contributing to $(n-2)H_{n-1}(x,y)$. Therefore, we have 
 \begin{align}\label{genera-pair len 2 12 15}
H_n(x,y)=(n+x+y-2)\,H_{n-1}(x,y). 
 \end{align}
It is easy to see that the initial conditions in \eqref{genera-pair len 2 12 15} are the same as those in \eqref{genera-pairs-len 2 8-17}. Moreover, we prove that \eqref{genera-pairs-len 2 8-17} satisfies \eqref{genera-pair len 2 12 15} as follows.
\begin{footnotesize}
 \begin{align*}
(n+x+y-2)\,H_{n-1}(x,y)=&(n+x+y-2)\,\sum_{k=0}^{n-2}\sum_{\ell=0}^{n-2}\binom{k+\ell}{k}\,c(n-2,k+\ell)\,x^k\,y^\ell
 \end{align*}
 
 \vspace{-0.8cm}
 
 \begin{align*}
=&\sum_{k=0}^{n-2}\sum_{\ell=0}^{n-2}(n-2)\,\binom{k+\ell}{k}\,c(n-2,k+\ell)\,x^k\,y^\ell+\sum_{k=1}^{n-1}\sum_{\ell=0}^{n-2}\,\binom{k+\ell-1}{k-1}\,c(n-2,k+\ell-1)\,x^k\,y^\ell\\
+&\sum_{k=0}^{n-2}\sum_{\ell=1}^{n-1}\,\binom{k+\ell-1}{k}\,c(n-2,k+\ell-1)\,x^k\,y^\ell\\
=&\sum_{k=0}^{n-1}\sum_{\ell=0}^{n-1}\Big(\binom{k+\ell}{k}\,(n-2)\,c(n-2,k+\ell)+\left(\binom{k+\ell-1}{k-1}+\binom{k+\ell-1}{k}\right)c(n-2,k+\ell-1)\Big)\,x^k\,y^\ell\\
=&\sum_{k=0}^{n-1}\sum_{\ell=0}^{n-1}\binom{k+\ell}{k}\,((n-2)\,c(n-2,k+\ell)+c(n-2,k+\ell-1))\,x^k\,y^\ell\\
=&\sum_{k=0}^{n-1}\sum_{\ell=0}^{n-1}\binom{k+\ell}{k}\,c(n-1,k+\ell)\,x^k\,y^\ell = H_n(x,y).
 \end{align*}
\end{footnotesize}
By extracting the coefficients from both sides of \eqref{genera-pair len 2 12 15}, we can directly obtain \eqref{recur-pairs-len 2 8-17}.

\item[(ii)] 
 For the pair 
$\{p_1\hspace{-0.1cm}=\hspace{-0.1cm}\pattern{scale = 0.5}{2}{1/1,2/2}{0/0,0/1,0/2,2/1},\hspace{0.1cm}p_2\hspace{-0.1cm}=\hspace{-0.1cm}\pattern{scale = 0.5}{2}{1/2,2/1}{0/0,0/1,0/2,2/1}\}$,
observe that occurrences of both patterns in a permutation \(\pi=\pi_1\cdots\pi_n\) begin with \(\pi_1\). Define 
$C=\{\pi_j: \pi_j<\pi_1\}$ and $D=\{\pi_j: \pi_j>\pi_1\}$, and suppose there are $i$ elements in $C$.
For each permutation counted by the left-hand side of \eqref{recur-pairs-len 2 8-17}, the permutation formed by \(C\) contains $\ell$ right-to-left maxima, that is, $\ell$ occurrences of the mesh pattern \(\pattern{scale = 0.5}{1}{1/1}{1/1}\), which is given by the unsigned Stirling numbers of the first kind, $c(i,\ell)$, as found in~\cite[Corollary 1.3.11]{Stanley} and mentioned in~\cite[Case 1]{S-S-J}. Moreover, the permutation formed by \(D\) contains \(k\) occurrences of the pattern \(\pattern{scale = 0.5}{1}{1/1}{1/0}\), contributing a factor of $c(n-i-1,k)$.
 Thus, we have 
 \begin{align*}
 H_{n,k,\ell}=&\sum_{i=0}^{n-1}\binom{n-1}{i}c(i,\ell)\,c(n-i-1,k)\\
 =& \binom{k+\ell}{k}\,c(n-1,k+\ell),
 \end{align*}
where the last line follows from a variation of the Chu-Vandermonde identity, see \cite[Equation (17)]{AM-S, mathworld}, also see \cite[Equation (5)]{L-Z-joint}. Furthermore, \eqref{genera-pairs-len 2 8-17} follows from~\eqref{recur-pairs-len 2 8-17}.

 \item[(iii)] For the pair $\{p_1\hspace{-0.1cm}=\hspace{-0.1cm}\pattern{scale = 0.5}{2}{1/1,2/2}{1/0,1/1,2/0,2/1},\hspace{0.1cm}p_2\hspace{-0.1cm}=\hspace{-0.1cm}\pattern{scale = 0.5}{2}{1/2,2/1}{1/0,1/1,2/0,2/1}\}$, our arguments for all the terms in \eqref{genera-pairs-len 2 8-17} are the same as those in $(i)$, except for the term $y\,H_{n-1}(x,y)$, which is replaced by the corresponding insertion immediately before the last entry, thereby creating an extra occurrence of $p_2$.
\end{enumerate}
Thus, the proof is complete.
  \end{proof}

 \begin{proof}[Proof of Theorem~\ref{coro-pairs-17-36}]
For each permutation counted by the left-hand side of \eqref{genera-pair 8 17}, suppose there are $i$ elements in $A$. For $i=0$ or $1$, the permutation formed by the remaining $n-1$ or $n-2$ elements in $B$ can be any permutation. For $i\geq 2$, the permutation formed by $A$ is counted by $H_i(x,y)$, and the permutation formed by $B$ can be any permutation. Thus, we have
\begin{align*}
T_n(x,y)=&2(n-1)!+\,\sum_{i=2}^{n-1}\,\binom{n-1}{i}\,(n-i-1)!\,H_i(x,y)\\
=&\sum_{k=0}^{n-1}\sum_{\ell=0}^{n-1}\sum_{i=2}^{n-1}\,\binom{n-1}{i}\,(n-i-1)!\,\binom{k+\ell}{k}\,c(i-1,k+\ell)\,x^k\,y^\ell+2(n-1)! \\
=&(n-1)!\sum_{k=0}^{n-1}\sum_{\ell=0}^{n-1}\sum_{i=0}^{n-1}\,\frac{1}{i!}\,\binom{k+\ell}{k}\,c(i-1,k+\ell)\,x^k\,y^\ell+2(n-1)!.    
\end{align*}
 By extracting the coefficients from both sides of \eqref{genera-pair 8 17}, we can directly obtain \eqref{recur-pair 8-17}.
\end{proof}

The computer experiments also suggests that the mesh patterns of pairs $\{X^{(1)}_i\}_{i=1}^{18}$ and $\{Y^{(1)}_i\}_{i=1}^{18}$ are Wilf-equivalent. We will proof this fact in the following corollary. 
\begin{coro}
The pairs $\{X^{(1)}_i\}_{i=1}^{18}$ and $\{Y^{(1)}_i\}_{i=1}^{18}$ are Wilf-equivalent. 
\end{coro}

\begin{proof}
Setting $k=\ell=0$ in \eqref{recur-pair 8-17} and by Subsection~\ref{subsec-X_1-1}, it remains to prove that pairs $\{\pattern{scale = 0.5}{3}{1/1,2/2,3/3}{0/0,0/1,0/2,0/3,1/1,1/2,1/3,2/2,3/2}, \pattern{scale = 0.5}{3}{1/1,2/3,3/2}{0/0,0/1,0/2,0/3,1/1,1/2,1/3,2/2,3/2}\}$ ($X^{(1)}_{1}$) and $\{\pattern{scale = 0.5}{3}{1/1,2/2,3/3}{0/0,0/1,0/2,0/3,1/1,1/2,2/2,3/1,3/2}, \pattern{scale = 0.5}{3}{1/1,2/3,3/2}{0/0,0/1,0/2,0/3,1/1,1/2,2/2,3/1,3/2}\}$ ($X^{(1)}_{5}$) satisfy $T_{n,0,0}=2(n-1)!$. 

For the pair $X^{(1)}_{1}$, consider any $n$-permutation $\pi=\pi_1\cdots\pi_n$. Let $A=\{\pi_i:\pi_i>\pi_1, 2\leq i \leq n\}$ and $B=\{\pi_i:\pi_i<\pi_1, 2\leq i \leq n\}$. The permutation $\pi$ counted by $T_{n,0,0}$ can be obtained from the following two cases:
\begin{itemize}
    \item If $A=\emptyset$, then the remaining $n-1$ elements involved in $B$ can form any permutation, resulting in $(n-1)!$ permutations.
    \item If $A\neq \emptyset$, then $|A|=1$. Otherwise, there would be at least one occurrence of $\pattern{scale = 0.5}{2}{1/1,2/2}{0/0,0/1,0/2,1/1,2/1}$ or $\pattern{scale = 0.5}{2}{1/2,2/1}{0/0,0/1,0/2,1/1,2/1}$ in $\pi$. The remaining elements in $B$ have no restrictions. Thus, this case gives $(n-1)!$ permutations.
\end{itemize}
By summing over the two cases above, we obtain $T_{n,0,0}=2(n-1)!$. The same approach can be applied to the pair $X^{(1)}_{5}$, thus, we omit the details here. The proof is complete.
\end{proof}

\begin{thm}[Pairs $\{X^{(1)}_i\}_{i=19}^{22}$ and $\{Y^{(1)}_i\}_{i=19}^{22}$]\label{thm-pair-Ding-all}
We have $\pattern{scale = 0.5}{3}{1/1,2/2,3/3}{0/0,0/1,0/2,0/3,1/1,1/3,2/1,2/2,2/3,3/3}\sim_{jd}\hspace{-1mm}\pattern{scale = 0.5}{3}{1/1,2/3,3/2}{0/0,0/1,0/2,0/3,1/1,1/3,2/1,2/2,2/3,3/3}$, $\pattern{scale = 0.5}{3}{1/1,2/2,3/3}{0/0,0/1,0/2,0/3,1/1,2/1,2/2,2/3,3/1,3/3}\sim_{jd}\hspace{-1mm}\pattern{scale = 0.5}{3}{1/1,2/3,3/2}{0/0,0/1,0/2,0/3,1/1,2/1,2/2,2/3,3/1,3/3}$, $\pattern{scale = 0.5}{3}{1/1,2/2,3/3}{0/0,0/1,0/2,0/3,1/1,1/2,1/3,2/2,3/2,3/3}\sim_{jd}\hspace{-1mm}\pattern{scale = 0.5}{3}{1/1,2/3,3/2}{0/0,0/1,0/2,0/3,1/1,1/2,1/3,2/2,3/2,3/3}$,
$\pattern{scale = 0.5}{3}{1/1,2/2,3/3}{0/0,0/1,0/2,0/3,1/1,1/2,2/2,3/1,3/2,3/3}\sim_{jd}\hspace{-1mm}\pattern{scale = 0.5}{3}{1/1,2/3,3/2}{0/0,0/1,0/2,0/3,1/1,1/2,2/2,3/1,3/2,3/3}$, $\pattern{scale = 0.5}{3}{1/1,2/2,3/3}{0/0,1/0,1/1,1/2,2/0,2/2,3/0,3/1,3/2,3/3}\sim_{jd}\hspace{-1mm}\pattern{scale = 0.5}{3}{1/1,2/3,3/2}{0/0,1/0,1/1,1/2,2/0,2/2,3/0,3/1,3/2,3/3}$, $\pattern{scale = 0.5}{3}{1/1,2/2,3/3}{0/0,1/0,1/1,1/2,1/3,2/0,2/2,3/2,3/0,3/3}\sim_{jd}\hspace{-1mm}\pattern{scale = 0.5}{3}{1/1,2/3,3/2}{0/0,1/0,1/1,1/2,1/3,2/0,2/2,3/2,3/0,3/3}$, $\pattern{scale = 0.5}{3}{1/1,2/2,3/3}{0/0,1/0,1/1,2/1,2/0,2/2,2/3,3/0,3/1,3/3}\sim_{jd}\hspace{-1mm}\pattern{scale = 0.5}{3}{1/1,2/3,3/2}{0/0,1/0,1/1,2/1,2/0,2/2,2/3,3/0,3/1,3/3}$, and $\pattern{scale = 0.5}{3}{1/1,2/2,3/3}{0/0,1/0,1/1,1/3,2/0,2/1,2/2,2/3,3/0,3/3}\sim_{jd}\hspace{-1mm}\pattern{scale = 0.5}{3}{1/1,2/3,3/2}{0/0,1/0,1/1,1/3,2/0,2/1,2/2,2/3,3/0,3/3}$. Moreover, any pair of these eight pairs has the same joint distribution. 
\end{thm}

\begin{proof}
Since 
\[
\pattern{scale = 0.5}{2}{1/1,2/2}{0/0,0/2,1/0,1/1,1/2,2/2} \pattern{scale = 0.5}{2}{1/2,2/1}{0/0,0/2,1/0,1/1,1/2,2/2}\xrightarrow{cr}\pattern{scale = 0.5}{2}{1/1,2/2}{0/0,1/0,1/1,1/2,2/0,2/2} \pattern{scale = 0.5}{2}{1/2,2/1}{0/0,1/0,1/1,1/2,2/0,2/2}\xrightarrow{i}  \pattern{scale = 0.5}{2}{1/1,2/2}{0/0,0/1,0/2,1/1,2/1,2/2} \pattern{scale = 0.5}{2}{1/2,2/1}{0/0,0/1,0/2,1/1,2/1,2/2}\xrightarrow{cr}\pattern{scale = 0.5}{2}{1/1,2/2}{0/0,0/1,1/1,2/0,2/1,2/2} \pattern{scale = 0.5}{2}{1/2,2/1}{0/0,0/1,1/1,2/0,2/1,2/2},
\]
by Lemma~\ref{reduction-lemma}, we have that any pair of the eight pairs has the same joint distribution. Therefore, to complete the proof, we only need to prove joint equidistributions for the pair 
$\{p_1\hspace{-0.1cm}=\hspace{-0.1cm}\pattern{scale = 0.5}{2}{1/1,2/2}{0/0,0/2,1/0,1/1,1/2,2/2},\hspace{0.1cm}p_2\hspace{-0.1cm}=\hspace{-0.1cm}\pattern{scale = 0.5}{2}{1/2,2/1}{0/0,0/2,1/0,1/1,1/2,2/2}\}$ associated with the pair $X^{(1)}_{19}$.

Observe that any occurrence of $p_1$ or $p_2$ in a permutation $\pi=\pi_1\cdots\pi_n\in S_n$ must involve two consecutive elements in $\pi$. Additionally, $\pi_i\pi_{i+1}$ is an occurrence of $p_1$ (resp., $p_2$) if and only if $\pi_i$ (resp., $\pi_{i+1}$) is a left-to-right minimum and $\pi_{i+1}=n$ (resp., $\pi_i=n$). Thus, either pattern appears at most once in a permutation. We let
\[
f(\pi)=\left\{  
\begin{array}{ll}  
\pi, & \text{if } \pi \in \mathcal{H}(n,0,0)\cup \mathcal{H}(n,1,1), \\[6pt]
\pi_1\cdots \pi_{i-1}\pi_{i+1}\pi_i\pi_{i+2}\cdots \pi_n, & \text{if } \pi \in \mathcal{H}(n,1,0)\cup \mathcal{H}(n,0,1).
\end{array}  
\right.
\]
Clearly, $f$ is a bijection that maps $\mathcal{H}(n,k,\ell)$ to $\mathcal{H}(n,k,\ell)$, proving the joint equidistribution for the pair $\{p_1,p_2\}$, and thereby for the pair $X^{(1)}_{19}$.
The proof is complete.
\end{proof}

\section{Mesh patterns of types $X^{(2)}$ and $Y^{(2)}$}\label{sec-X_2 Y_2}

\begin{figure}
\begin{center}
\begin{tabular}{rr}
\begin{tikzpicture}[scale=0.7]

\tikzset{      
   grid/.style={        
      draw=gray!100,  
      thin,        
    },   
    cell/.style={      
      draw,      
      anchor=center,    
      text centered,    
     },    
    graycell/.style={   
      fill=gray!20,
      draw=none,     
      minimum width=1cm,   
      minimum height=1cm,     
      anchor=south west,            
    }    
} 

\foreach \y in {0,1,2,3,4} {
  \draw[grid] (0,\y) -- (4,\y); 
}

\fill[graycell] (0,0) rectangle (1,3);
\fill[graycell] (1,0) rectangle (2,2);
\fill[graycell] (2,0) rectangle (3,1);
\draw(0,0)--(0,4);
\draw(1,0)--(1,3);
\draw(2,0)--(2,2);
\draw(3,0)--(3,1);
\draw(4,0)--(4,4);
\draw(0,4)--(4,4);
\draw(0,3)--(4,3);
\draw(0,2)--(4,2);
\draw(0,1)--(4,1);
\draw(0,0)--(4,0);

\node at (2 ,3.5) {\scriptsize$A_1$};  
\node at (2.5,2.5) {\scriptsize$A_2$};  
\node at (3.5,0.5) {\scriptsize$A_t$};
\node[anchor=east] at (0,2.7) {\footnotesize{$x_1$}}; 
\node[anchor=east] at (1,1.7) {\footnotesize{$x_2$}}; 
\node[anchor=east] at (3,-0.3) {\footnotesize{$x_t$}}; 

\filldraw[black] (0,3) circle (2.5pt);
\filldraw[black] (1,2) circle (2.5pt);
\filldraw[thick] (3,0) circle (2.5pt);
\node[anchor=center, rotate=-45] at (2.5,1.5) {$\ldots$};
\end{tikzpicture} 
\vspace{-0.8cm}
\end{tabular}
\end{center}
\caption{The structure of permutations containing occurrences of patterns in Section~\ref{sec-X_2 Y_2}.}\label{fig-secon}
\end{figure}
In this section, we consider the patterns of types $X^{(2)}=\raisebox{0.5ex}{\begin{tikzpicture}[scale = 0.2, baseline=(current bounding box.center)]
	\foreach \x/\y in {0/0,0/1,0/2,3/0}		
		\fill[gray!50] (\x,\y) rectangle +(1,1); 	
        \draw (0,1)--(4,1);
        \draw (1,0)--(1,4);
        \draw (0,2)--(1,2);
         \draw (0,3)--(1,3);
         \draw (2,0)--(2,1);
         \draw (3,0)--(3,1);
 			\node  at (2.6,2.6) {$p$};
 			\filldraw (1,1) circle (4pt) node[above left] {};
 			\end{tikzpicture}}$ and $Y^{(2)}=\raisebox{0.5ex}{\begin{tikzpicture}[scale = 0.2, baseline=(current bounding box.center)]
	\foreach \x/\y in {0/0,1/0,2/0,0/3}		
		\fill[gray!50] (\x,\y) rectangle +(1,1); 	
        \draw (0,1)--(4,1);
        \draw (1,0)--(1,4);
        \draw (0,2)--(1,2);
         \draw (0,3)--(1,3);
         \draw (2,0)--(2,1);
         \draw (3,0)--(3,1);
 			\node  at (2.6,2.6) {$p$};
 			\filldraw (1,1) circle (4pt) node[above left] {};
 			\end{tikzpicture}}$.
For the patterns of type $X^{(2)}$,
let $\pi \in S_n$, and $x_1>x_2>\cdots>x_t$ be the sequence of left-to-right minima in $\pi$. Observe that any occurrence of each pair of these patterns begin with an $x_i$. Moreover, referring to Figure~\ref{fig-secon}, such an occurrence has the pattern $p$ appearing in $A_i$, which means occurrences of mesh patterns of type $X^{(2)}$  must be entirely within $\{x_i\}\cup A_i$ for some $i$.  

\subsection{Cases proved via complementation}\label{subsec-X_2-1}
In this subsection, we consider the following 16 pairs presented in Table~\ref{tab-X_2-1}.

\begin{table}[!ht]
 	{
 		\renewcommand{\arraystretch}{1.7}
 		\setlength{\tabcolsep}{5pt}
 \begin{center} 
 		\begin{tabular}{c | c || c | c || c | c || c | c}
 			\hline
 		\footnotesize{Nr.}	& \footnotesize {Patterns} & \footnotesize{Nr.}	& \footnotesize{Patterns} & \footnotesize {Nr.}  & \footnotesize{Patterns} &\footnotesize {Nr.}  & \footnotesize{Patterns}\\
 			\hline
 			\hline
$X^{(2)}_{1}$ & $\pattern{scale = 0.5}{3}{1/1,2/2,3/3}{0/0,0/1,0/2,3/0,1/1,1/2,1/3,2/2,3/2}\pattern{scale = 0.5}{3}{1/1,2/3,3/2}{0/0,0/1,0/2,3/0,1/1,1/2,1/3,2/2,3/2}$ &
$X^{(2)}_2$ & $\pattern{scale = 0.5}{3}{1/1,2/2,3/3}{0/0,0/1,0/2,3/0,1/2,2/2,3/1,3/2,3/3}\pattern{scale = 0.5}{3}{1/1,2/3,3/2}{0/0,0/1,0/2,3/0,1/2,2/2,3/1,3/2,3/3}$& 
$X^{(2)}_3$ & $\pattern{scale = 0.5}{3}{1/1,2/2,3/3}{0/0,0/1,0/2,3/0,1/1,1/3,2/1,2/2,2/3}\pattern{scale = 0.5}{3}{1/1,2/3,3/2}{0/0,0/1,0/2,3/0,1/1,1/3,2/1,2/2,2/3}$ & 
$X^{(2)}_4$ & $\pattern{scale = 0.5}{3}{1/1,2/2,3/3}{0/0,0/1,0/2,3/0,2/1,2/2,2/3,3/1,3/3}\pattern{scale = 0.5}{3}{1/1,2/3,3/2}{0/0,0/1,0/2,3/0,2/1,2/2,2/3,3/1,3/3}$ 
\\[5pt]
$X^{(2)}_5$ & $\pattern{scale = 0.5}{3}{1/1,2/2,3/3}{0/0,0/1,0/2,3/0,1/1,1/3,2/1,2/3}\pattern{scale = 0.5}{3}{1/1,2/3,3/2}{0/0,0/1,0/2,3/0,1/1,1/3,2/1,2/3}$ &
$X^{(2)}_{6}$ & $\pattern{scale = 0.5}{3}{1/1,2/2,3/3}{0/0,0/1,0/2,3/0,2/1,2/3,3/1,3/3}\pattern{scale = 0.5}{3}{1/1,2/3,3/2}{0/0,0/1,0/2,3/0,2/1,2/3,3/1,3/3}$ &
$X^{(2)}_{7}$ & $\pattern{scale = 0.5}{3}{1/1,2/2,3/3}{0/0,0/1,0/2,3/0,1/1,1/2,1/3,3/2}\pattern{scale = 0.5}{3}{1/1,2/3,3/2}{0/0,0/1,0/2,3/0,1/1,1/2,1/3,3/2}$  &
$X^{(2)}_{8}$ & $\pattern{scale = 0.5}{3}{1/1,2/2,3/3}{0/0,0/1,0/2,3/0,1/2,3/1,3/2,3/3}\pattern{scale = 0.5}{3}{1/1,2/3,3/2}{0/0,0/1,0/2,3/0,1/2,3/1,3/2,3/3}$ 
\\[5pt]
\hline
$Y^{(2)}_1$ & $\pattern{scale = 0.5}{3}{1/1,2/2,3/3}{0/0,1/0,1/1,2/0,2/1,2/2,2/3,0/3,3/1}\pattern{scale = 0.5}{3}{1/1,2/3,3/2}{0/0,1/0,1/1,2/0,2/1,2/2,2/3,0/3,3/1}$ &
$Y^{(2)}_2$ &  $\pattern{scale = 0.5}{3}{1/1,2/2,3/3}{0/0,1/0,1/3,2/0,2/1,2/2,2/3,0/3,3/3}\pattern{scale = 0.5}{3}{1/1,2/3,3/2}{0/0,1/0,1/3,2/0,2/1,2/2,2/3,0/3,3/3}$&
$Y^{(2)}_3$ & $\pattern{scale = 0.5}{3}{1/1,2/2,3/3}{0/0,1/0,1/1,1/2,2/0,2/2,0/3,3/1,3/2}\pattern{scale = 0.5}{3}{1/1,2/3,3/2}{0/0,1/0,1/1,1/2,2/0,2/2,0/3,3/1,3/2}$&
$Y^{(2)}_4$ & $\pattern{scale = 0.5}{3}{1/1,2/2,3/3}{0/0,1/0,1/2,1/3,2/0,2/2,0/3,3/3,3/2}\pattern{scale = 0.5}{3}{1/1,2/3,3/2}{0/0,1/0,1/2,1/3,2/0,2/2,0/3,3/3,3/2}$
\\[5pt]
$Y^{(2)}_5$ & $\pattern{scale = 0.5}{3}{1/1,2/2,3/3}{0/0,1/0,1/1,1/2,2/0,0/3,3/2,3/3}\pattern{scale = 0.5}{3}{1/1,2/3,3/2}{0/0,1/0,1/1,1/2,2/0,0/3,3/2,3/3}$ &
$Y^{(2)}_{6}$ & $\pattern{scale = 0.5}{3}{1/1,2/2,3/3}{0/0,1/0,1/2,2/0,0/3,1/3,3/2,3/3}\pattern{scale = 0.5}{3}{1/1,2/3,3/2}{0/0,1/0,1/2,2/0,0/3,1/3,3/2,3/3}$ &
$Y^{(2)}_{7}$ & $\pattern{scale = 0.5}{3}{1/1,2/2,3/3}{0/0,1/0,1/1,2/0,2/1,2/3,0/3,3/1}\pattern{scale = 0.5}{3}{1/1,2/3,3/2}{0/0,1/0,1/1,2/0,2/1,2/3,0/3,3/1}$& 
$Y^{(2)}_{8}$ & $\pattern{scale = 0.5}{3}{1/1,2/2,3/3}{0/0,1/0,1/3,2/0,2/1,2/3,0/3,3/3}\pattern{scale = 0.5}{3}{1/1,2/3,3/2}{0/0,1/0,1/3,2/0,2/1,2/3,0/3,3/3}$ \\[5pt]
\hline
	\end{tabular}
\end{center} 
}
\vspace{-0.5cm}
 	\caption{Jointly equidistributed patterns of types $X^{(2)}$ and $Y^{(2)}$ via complementation.}\label{tab-X_2-1}
\end{table}

Observe that the shadings of the patterns of the pairs $\{p_1,\hspace{0.1cm} p_2\}$ associated with the pairs $\{X^{(2)}_i\}_{i=1}^{8}$ exhibit horizontal symmetry. 
We establish a bijective map $f$ as follows: for a given $\pi$, apply the complement operation on each $A_i$ in $\pi$, that is, inside each $A_i$ replace its largest element by its smallest element (in the same position), its next largest element by next smallest element (in the same position), and so on. Clearly, as shown in Figure~\ref{fig-secon}, occurrences of the patterns in any pair $\{X^{(2)}_i\}_{i=1}^{8}$  will be swapped and no new occurrences of the patterns from the same pair are introduced. This proves the desired joint equidistributions. 

\subsection{Cases proved via swapping elements}\label{subsec-X_2-2}
\begin{table}[!ht]
 	{
 		\renewcommand{\arraystretch}{1.7}
 		\setlength{\tabcolsep}{5pt}
 \begin{center} 
 		\begin{tabular}{c | c || c | c || c | c || c | c}
 			\hline
 		\footnotesize{Nr.}	& \footnotesize {Patterns} & \footnotesize{Nr.}	& \footnotesize{Patterns} & \footnotesize {Nr.}  & \footnotesize{Patterns} &\footnotesize {Nr.}  & \footnotesize{Patterns}\\
 			\hline
 			\hline
 $X^{(2)}_9$& $\pattern{scale = 0.5}{3}{1/1,2/2,3/3}{0/0,0/1,0/2,3/0,1/1,2/1,2/2,2/3,3/1}\pattern{scale = 0.5}{3}{1/1,2/3,3/2}{0/0,0/1,0/2,3/0,1/1,2/1,2/2,2/3,3/1}$ & 
$X^{(2)}_{10}$& $\pattern{scale = 0.5}{3}{1/1,2/2,3/3}{0/0,0/1,0/2,3/0,1/3,2/1,2/2,2/3,3/3}\pattern{scale = 0.5}{3}{1/1,2/3,3/2}{0/0,0/1,0/2,3/0,1/3,2/1,2/2,2/3,3/3}$ &
$X^{(2)}_{11}$ & $\pattern{scale = 0.5}{3}{1/1,2/2,3/3}{0/0,0/1,0/2,3/0,1/1,1/2,2/2,3/1,3/2}\pattern{scale = 0.5}{3}{1/1,2/3,3/2}{0/0,0/1,0/2,3/0,1/1,1/2,2/2,3/1,3/2}$ & 
$X^{(2)}_{12}$ & $\pattern{scale = 0.5}{3}{1/1,2/2,3/3}{0/0,0/1,0/2,3/0,1/2,1/3,2/2,3/2,3/3}\pattern{scale = 0.5}{3}{1/1,2/3,3/2}{0/0,0/1,0/2,3/0,1/2,1/3,2/2,3/2,3/3}$
\\[5pt]
$X^{(2)}_{13}$ & $\pattern{scale = 0.5}{3}{1/1,2/2,3/3}{0/0,0/1,0/2,3/0,1/2,1/3,2/2,2/3}\pattern{scale = 0.5}{3}{1/1,2/3,3/2}{0/0,0/1,0/2,3/0,1/2,1/3,2/2,2/3}$ &
$X^{(2)}_{14}$ & $\pattern{scale = 0.5}{3}{1/1,2/2,3/3}{0/0,0/1,0/2,3/0,1/1,1/3,2/1,2/2,2/3,3/3}\pattern{scale = 0.5}{3}{1/1,2/3,3/2}{0/0,0/1,0/2,3/0,1/1,1/3,2/1,2/2,2/3,3/3}$ &
$X^{(2)}_{15}$ & $\pattern{scale = 0.5}{3}{1/1,2/2,3/3}{0/0,0/1,0/2,3/0,1/1,2/1,2/2,2/3,3/1,3/3}\pattern{scale = 0.5}{3}{1/1,2/3,3/2}{0/0,0/1,0/2,3/0,1/1,2/1,2/2,2/3,3/1,3/3}$ &
$X^{(2)}_{16}$ & $\pattern{scale = 0.5}{3}{1/1,2/2,3/3}{0/0,0/1,0/2,3/0,1/1,1/2,1/3,2/2,3/2,3/3}\pattern{scale = 0.5}{3}{1/1,2/3,3/2}{0/0,0/1,0/2,3/0,1/1,1/2,1/3,2/2,3/2,3/3}$
\\[5pt]
$X^{(2)}_{17}$ & $\pattern{scale = 0.5}{3}{1/1,2/2,3/3}{0/0,0/1,0/2,3/0,1/1,1/2,2/2,3/1,3/2,3/3}\pattern{scale = 0.5}{3}{1/1,2/3,3/2}{0/0,0/1,0/2,3/0,1/1,1/2,2/2,3/1,3/2,3/3}$ &&& &&&
\\[5pt]
\hline
$Y^{(2)}_{9}$ & $\pattern{scale = 0.5}{3}{1/1,2/2,3/3}{0/0,1/0,1/1,1/2,1/3,2/0,2/2,0/3,3/2}\pattern{scale = 0.5}{3}{1/1,2/3,3/2}{0/0,1/0,1/1,1/2,1/3,2/0,2/2,0/3,3/2}$ & 
$Y^{(2)}_{10}$ & $\pattern{scale = 0.5}{3}{1/1,2/2,3/3}{0/0,1/0,1/2,2/0,2/2,0/3,3/3,3/1,3/2}\pattern{scale = 0.5}{3}{1/1,2/3,3/2}{0/0,1/0,1/2,2/0,2/2,0/3,3/3,3/1,3/2}$ &
$Y^{(2)}_{11}$ & $\pattern{scale = 0.5}{3}{1/1,2/2,3/3}{0/0,1/0,1/1,1/3,2/0,2/1,2/2,2/3,0/3}\pattern{scale = 0.5}{3}{1/1,2/3,3/2}{0/0,1/0,1/1,1/3,2/0,2/1,2/2,2/3,0/3}$& 
$Y^{(2)}_{12}$ & $\pattern{scale = 0.5}{3}{1/1,2/2,3/3}{0/0,1/0,2/0,2/1,2/2,2/3,0/3,3/1,3/3}\pattern{scale = 0.5}{3}{1/1,2/3,3/2}{0/0,1/0,2/0,2/1,2/2,2/3,0/3,3/1,3/3}$
\\[5pt]
$Y^{(2)}_{13}$ & $\pattern{scale = 0.5}{3}{1/1,2/2,3/3}{0/0,1/0,2/1,2/2,2/0,0/3,3/1,3/2}\pattern{scale = 0.5}{3}{1/1,2/3,3/2}{0/0,1/0,2/1,2/2,2/0,0/3,3/1,3/2}$ &
$Y^{(2)}_{14}$ & $\pattern{scale = 0.5}{3}{1/1,2/2,3/3}{0/0,1/0,1/1,1/2,2/0,2/2,0/3,3/1,3/2,3/3}\pattern{scale = 0.5}{3}{1/1,2/3,3/2}{0/0,1/0,1/1,1/2,2/0,2/2,0/3,3/1,3/2,3/3}$ &
$Y^{(2)}_{15}$ & $\pattern{scale = 0.5}{3}{1/1,2/2,3/3}{0/0,1/0,1/1,1/2,1/3,2/0,2/2,3/2,0/3,3/3}\pattern{scale = 0.5}{3}{1/1,2/3,3/2}{0/0,1/0,1/1,1/2,1/3,2/0,2/2,3/2,0/3,3/3}$& 
$Y^{(2)}_{16}$ & $\pattern{scale = 0.5}{3}{1/1,2/2,3/3}{0/0,1/0,1/1,2/1,2/0,2/2,2/3,0/3,3/1,3/3}\pattern{scale = 0.5}{3}{1/1,2/3,3/2}{0/0,1/0,1/1,2/1,2/0,2/2,2/3,0/3,3/1,3/3}$ \\[5pt]
$Y^{(2)}_{17}$ &$\pattern{scale = 0.5}{3}{1/1,2/2,3/3}{0/0,1/0,1/1,1/3,2/0,2/1,2/2,2/3,0/3,3/3}\pattern{scale = 0.5}{3}{1/1,2/3,3/2}{0/0,1/0,1/1,1/3,2/0,2/1,2/2,2/3,0/3,3/3}$& && &&&\\[5pt]
\hline
	\end{tabular}
\end{center} 
}
\vspace{-0.5cm}
 	\caption{Jointly equidistributed patterns of types $X^{(2)}$ and $Y^{(2)}$ via swapping elements.}\label{tab-X_2-2}
\end{table}

In this subsection, we consider the following 18 pairs presented in Table~\ref{tab-X_2-2}.

\begin{thm}[Pairs $X^{(2)}_{9}$ and $Y^{(2)}_{9}$]\label{thm-X^2_9}
  We have $\pattern{scale = 0.5}{3}{1/1,2/2,3/3}{0/0,0/1,0/2,3/0,1/1,2/1,2/2,2/3,3/1}\sim_{jd}\hspace{-1mm}\pattern{scale = 0.5}{3}{1/1,2/3,3/2}{0/0,0/1,0/2,3/0,1/1,2/1,2/2,2/3,3/1}$ and $\pattern{scale = 0.5}{3}{1/1,2/2,3/3}{0/0,1/0,1/1,1/2,1/3,2/0,2/2,0/3,3/2}\sim_{jd}\hspace{-1mm}\pattern{scale = 0.5}{3}{1/1,2/3,3/2}{0/0,1/0,1/1,1/2,1/3,2/0,2/2,0/3,3/2}$.
\end{thm}
\begin{proof}
For $X^{(2)}_{9}$, $x_iab$ is an occurrence of $q_1$ (resp., $q_2$) in $\pi \in S_n$, where $1\leq i\leq t$, if and only if $a=x_i+1$ (resp., $b=x_i+1$), and $b$ ($a$) is immediately to the right (resp., left) of $a$ ($b$) with $b>a$ (resp., $a>b$) in $A_i$. Therefore, each $A_i$ contains at most one occurrence of 
$p_1$ and one occurrence of $p_2$. We construct the map $f$ as follows:  
\begin{enumerate}
    \item [(i)] If $A_i \in \mathcal{H}(n_i,0,0)\cup \mathcal{H}(n_i,1,1)$, then $A_i$ remains unchanged. 
     \item [(ii)] If $A_i \in \mathcal{H}(n_i,1,0)\cup \mathcal{H}(n_i,0,1)$, then we swap the elements $a$ and $b$,
\end{enumerate}
where $n_i=|A_i|$ and $i\in \{1, 2, \ldots, t\}$. Clearly, $f$ is a bijection, and the theorem is proved.
\end{proof}

\begin{thm}[Pairs $X^{(2)}_{10}$ and $Y^{(2)}_{10}$]\label{thm-X^2_10}
  We have $\pattern{scale = 0.5}{3}{1/1,2/2,3/3}{0/0,0/1,0/2,3/0,1/3,2/1,2/2,2/3,3/3}\sim_{jd}\hspace{-1mm}\pattern{scale = 0.5}{3}{1/1,2/3,3/2}{0/0,0/1,0/2,3/0,1/3,2/1,2/2,2/3,3/3}$ and $\pattern{scale = 0.5}{3}{1/1,2/2,3/3}{0/0,1/0,1/2,2/0,2/2,0/3,3/3,3/1,3/2}\sim_{jd}\hspace{-1mm}\pattern{scale = 0.5}{3}{1/1,2/3,3/2}{0/0,1/0,1/2,2/0,2/2,0/3,3/3,3/1,3/2}$. 
\end{thm}
\begin{proof}
For $X^{(2)}_{10}$, $x_iab$ is an occurrence of $q_1$ (resp., $q_2$) in $\pi \in S_n$, where $1\leq i\leq t$, if and only if $b$ (resp., $a$) is the largest element to the right of $x_i$, and $a$ (resp., $b$) is immediately to the left (resp., right) of $b$ (resp., $a$) with $a<b$ (resp., $b<a$) in $A_i$. We claim that each permutation contains at most one occurrence of $q_1$ (resp., $q_2$). Otherwise, suppose $x_j a'b'$ is an occurrence of $q_1$ that differs from $x_i ab$ (the case of $q_2$ can be considered in a similar way). 

\vspace{-0.5cm}

\begin{itemize}
    \item If $x_i=x_j$, then $b'=b$, which also leads to $a'=a$, a contradiction.
    \item If $x_i\neq x_j$, then assume, w.l.o.g., that $x_i>x_j$. We see that $x_j$ is positioned between $x_i$ and $b$ because of the shaded area $(3,0)$. It must be the case that $b'=b$, which also leads to a contradiction due to the element $x_i$.
\end{itemize}

\vspace{-0.5cm}

\noindent
The claim is proved. Next, we construct the map $f$ as follows:
\begin{enumerate}
    \item [(i)] If $\pi \in \mathcal{T}(n,0,0)\cup \mathcal{T}(n,1,1)$, then $f(\pi)=\pi$. 
     \item [(ii)] If $\pi \in \mathcal{T}(n,1,0)\cup \mathcal{T}(n,0,1)$, then we swap the elements $a$ and $b$,
\end{enumerate}
where $i\in\{1, 2, \ldots, t\}$. Clearly, $f$ is a bijection, and the theorem is proved.
\end{proof}

\begin{thm}[Pairs $X^{(2)}_{11}$ and $Y^{(2)}_{11}$]
    We have $\pattern{scale = 0.5}{3}{1/1,2/2,3/3}{0/0,0/1,0/2,3/0,1/1,1/2,2/2,3/1,3/2}\sim_{jd}\hspace{-1mm}\pattern{scale = 0.5}{3}{1/1,2/3,3/2}{0/0,0/1,0/2,3/0,1/1,1/2,2/2,3/1,3/2}$ and $\pattern{scale = 0.5}{3}{1/1,2/2,3/3}{0/0,1/0,1/1,1/3,2/0,2/1,2/2,2/3,0/3}\sim_{jd}\hspace{-1mm}\pattern{scale = 0.5}{3}{1/1,2/3,3/2}{0/0,1/0,1/1,1/3,2/0,2/1,2/2,2/3,0/3}$.
\end{thm}
\begin{proof}
For $X^{(2)}_{11}$, $x_iab$ is an occurrence of $q_1$(resp., $q_2$) in $\pi \in S_n$, where $1\leq i\leq t$, if and only if $a$ is a left-to-right minimum in $A_i$, and $b=a+1$ (resp., $b=a-1$) is a right-to-left minimum. Therefore, each $A_i$ contains at most one occurrence of $p_1$ and one occurrence of $p_2$, which allows us to use the same map $f$ as in Theorem~\ref{thm-X^2_9}. 
\end{proof}

\begin{thm}[Pairs $X^{(2)}_{12}$ and $Y^{(2)}_{12}$]
    We have  $\pattern{scale = 0.5}{3}{1/1,2/2,3/3}{0/0,0/1,0/2,3/0,1/2,1/3,2/2,3/2,3/3}\sim_{jd}\hspace{-1mm}\pattern{scale = 0.5}{3}{1/1,2/3,3/2}{0/0,0/1,0/2,3/0,1/2,1/3,2/2,3/2,3/3}$ and $\pattern{scale = 0.5}{3}{1/1,2/2,3/3}{0/0,1/0,2/0,2/1,2/2,2/3,0/3,3/1,3/3}\sim_{jd}\hspace{-1mm}\pattern{scale = 0.5}{3}{1/1,2/3,3/2}{0/0,1/0,2/0,2/1,2/2,2/3,0/3,3/1,3/3}$.
\end{thm}
\begin{proof}
For $X^{(2)}_{12}$, $x_i ab$ is an occurrence of $q_1$ (resp., $q_2$) in $\pi \in S_n$, where $1\leq i\leq t$, if and only if $a$ is a left-to-right maximum in $A_i$, and $b=a+1$ (resp., $b=a-1$) is a right-to-left maximum. Thus, we can use the same map $f$ as in Theorem~\ref{thm-X^2_9}.
\end{proof}

The following lemma can be easily deduced from a result in~\cite{KLv}, but we provide the details of its proof in order to introduce the procedure of swapping elements, which will be used in the proof of Theorem~\ref{thm-box-X_2} below.

\begin{lem}\label{lem-box-X_2}
    We have $\pattern{scale = 0.5}{2}{1/1,2/2}{0/1,0/2,1/1,1/2}\sim_{jd}\hspace{-1mm}\pattern{scale = 0.5}{2}{1/2,2/1}{0/1,0/2,1/1,1/2}$.
\end{lem}
\begin{proof}
    Our arguments here are analogous to those in \cite[Theorem 2.3]{KLv}. 
Suppose that $\pi \in S_n$ contains $k$ occurrences of $\pattern{scale = 0.5}{2}{1/1,2/2}{0/1,0/2,1/1,1/2}$ and $\ell$ occurrences of $\pattern{scale = 0.5}{2}{1/2,2/1}{0/1,0/2,1/1,1/2}$. 
Let $\pi^{(1)}= \pi$ and $\pi^{(1)}_{i_1}\pi^{(1)}_{i_2}\ldots\pi^{(1)}_{i_t}$ be the subsequence of $\pi^{(1)}$ consisting of all elements that participate in any pattern occurrence.
Observe that, once the second element of a pattern occurrence is fixed,   the first element   is uniquely determined as the largest element preceding it in~$\pi$.
We choose the occurrence whose second element is $\pi^{(1)}_{i_t}$ and swap these two elements to obtain $\pi^{(2)}$.
At the $j$-th step,  where $2\le j\le t-1$, we let $\pi^{(j)}_{i_1}\pi^{(j)}_{i_2}\ldots\pi^{(j)}_{i_t}$ be the subsequence of $\pi^{(j)}$ containing all elements in the occurrences of the patterns.  We  select the occurrence whose second element is $\pi^{(j)}_{i_{t-j+1}}$ and then swap the two elements to obtain $\pi^{(j+1)}$.
Eventually, we obtain the resulting permutation $\pi'=\pi^{(t)}$.

Since every element except the first in the subsequence $\pi^{(1)}_{i_1}\pi^{(1)}_{i_2}\dots\pi^{(1)}_{i_t}$ plays the role of the second element of a pattern, so $t-1=k+\ell$.
Moreover, each swap changes exactly one occurrence of a pattern into an occurrence of the other pattern and preserves the number of the remaining occurrences.
Thus, $\pi'$ contains $\ell$ occurrences of $\pattern{scale = 0.5}{2}{1/1,2/2}{0/1,0/2,1/1,1/2}$ and $k$ occurrences of $\pattern{scale = 0.5}{2}{1/2,2/1}{0/1,0/2,1/1,1/2}$, which proves their joint equidistribution. 

\begin{figure}[h]
	\begin{center}
		
		\begin{tabular}{ccc}
			
			\begin{tikzpicture}[scale=0.25] 
			\tikzset{    
				grid/.style={      
					draw,      
					step=1cm,      
					gray!100,     
					very thin,      
				}, 
				cell/.style={    
					draw,    
					anchor=center,  
					text centered,    
				},  
				graycell/.style={ 
					fill=gray!40,   
					draw=none,   
					minimum width=1cm, 
					minimum height=1cm,   
					anchor=south west,   
				}
			}  
			
			\draw[grid] (0,0) grid (8,8);  
			%
			\draw[thick] (0,1) circle (8pt); \draw[thick] (0,1) circle (2pt); 
			\draw[thick] (1,5) circle (8pt); \draw[thick] (1,5) circle (2pt); 
			\draw[thick] (2,2) circle (8pt);  \filldraw[thick] (2,2) circle (2pt);
			\draw[thick] (3,4) circle (8pt);  \filldraw[thick] (3,4) circle (2pt);
			\filldraw[black] (4,0) circle (8pt);
			\draw[thick] (5,7) circle (8pt);
			\draw[thick] (5,7) circle (2pt);
			\filldraw[black] (6,3) circle (8pt); 
			\draw[thick] (7,8) circle (8pt); \filldraw[thick] (7,8) circle (2pt); 
			\filldraw[black] (8,6) circle (8pt); 
			
			\end{tikzpicture} 
			
			& \ & 
			\begin{tikzpicture}[scale=0.25] 
			\tikzset{    
				grid/.style={      
					draw,      
					step=1cm,      
					gray!100,     
					very thin,      
				}, 
				cell/.style={    
					draw,    
					anchor=center,  
					text centered,    
				},  
				graycell/.style={ 
					fill=gray!40,   
					draw=none,   
					minimum width=1cm, 
					minimum height=1cm,   
					anchor=south west,   
				}
			}  
			
			\draw[grid] (0,0) grid (8,8);  
			\draw[thick] (0,2) circle (8pt); \draw[thick] (0,2) circle (2pt); 
			\draw[thick] (1,1) circle (8pt);  \filldraw[thick] (1,1) circle (2pt);
			\draw[thick] (2,4) circle (8pt);  \draw[thick] (2,4) circle (2pt);
			\draw[thick] (3,8) circle (8pt);  \draw[thick] (3,8) circle (2pt);
			\filldraw[black] (4,0) circle (8pt);  
			\draw[thick] (5,5) circle (8pt);  \filldraw[thick] (5,5) circle (2pt);
			\filldraw[black] (6,3) circle (8pt);
			\draw[thick] (7,7) circle (8pt); \filldraw[thick] (7,7) circle (2pt); 
			\filldraw[black] (8,6) circle (8pt); 
			
						\end{tikzpicture}
		\end{tabular}
	\end{center}
	
	\vspace{-0.5cm}
	
	\caption{Permutations $\pi=263518497 $  (to the left) and $\pi'=325916487$ (to the right) illustrating the proof of Lemma~\ref{lem-box-X_2}. The circled dots represent the elements involved in occurrences of the patterns.}\label{fig-box-2}
\end{figure}
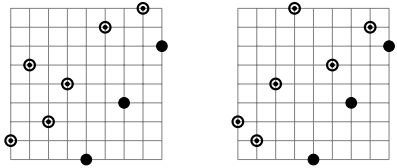

For example, for the permutation $\pi$ in Figure~\ref{fig-box-2}, the procedure is as follows: $\pi=26351\uline{8}4\uline{9}7\xrightarrow{\text{swap 8 \& 9}}2\uline{6}351\uline{9}487\xrightarrow{\text{swap 6 \& 9}}2\uline{9}3\uline{5}16487\xrightarrow{\text{swap 9 \& 5}}2\uline{53}916487\xrightarrow{\text{swap 5 \& 3}}\uline{2}\uline{3}5916487\xrightarrow{\text{swap 2 \& 3}}325916487$, so that $\pi'= 325916487$.
Observe that the occurrences 26, 68 and 89 of $\pattern{scale = 0.5}{2}{1/1,2/2}{0/1,0/2,1/1,1/2}$ in $\pi$ are replaced, respectively, by the occurrence 32, 96 and 98 of $\pattern{scale = 0.5}{2}{1/2,2/1}{0/1,0/2,1/1,1/2}$ in $\pi'$, while the occurrences 63 and 65 of $\pattern{scale = 0.5}{2}{1/2,2/1}{0/1,0/2,1/1,1/2}$ in $\pi$ are replaced by the occurrences 35 and 59 of $\pattern{scale = 0.5}{2}{1/1,2/2}{0/1,0/2,1/1,1/2}$ in $\pi'$. 
\end{proof}

\begin{thm}[Pairs $X^{(2)}_{13}$ and $Y^{(2)}_{13}$]\label{thm-box-X_2}
    We have $\pattern{scale = 0.5}{3}{1/1,2/2,3/3}{0/0,0/1,0/2,3/0,1/2,1/3,2/2,2/3}\sim_{jd}\hspace{-1mm}\pattern{scale = 0.5}{3}{1/1,2/3,3/2}{0/0,0/1,0/2,3/0,1/2,1/3,2/2,2/3}$, and $\pattern{scale = 0.5}{3}{1/1,2/2,3/3}{0/0,1/0,2/0,0/3,2/1,2/2,3/1,3/2}\sim_{jd}\hspace{-1mm}\pattern{scale = 0.5}{3}{1/1,2/3,3/2}{0/0,1/0,2/0,0/3,2/1,2/2,3/1,3/2}$.
\end{thm}
\begin{proof}
For $X^{(2)}_{13}$, note that the rightmost element in $A_{i+1}$ must be to the left of the leftmost element in $A_i$, as depicted in Figure~\ref{fig-box-X_2}. Therefore, the applications of the process described in Lemma~\ref{lem-box-X_2} in each $A_i$ do not affect each other.
\end{proof}

\begin{figure}[h]
\begin{center}  
\begin{tabular}{rr}
\begin{tikzpicture}[scale=0.7]

\tikzset{      
   grid/.style={        
      draw,        
      step=1cm,        
      gray!100,    
      thin,        
    },   
    cell/.style={      
      draw,      
      anchor=center,    
      text centered,    
     },    
    graycell/.style={   
      fill=gray!20,
      draw=none,     
      minimum width=1cm,   
      minimum height=1cm,     
      anchor=south west,            
    }    
  } 
    \fill[graycell] (0,0) rectangle (1,3);
  \fill[graycell] (1,0) rectangle (2,2);
 
 \draw[grid] (0,0) grid (6,4);  
  
  \node at (4.5,2.5) {\scriptsize$A_i$};  
  \node at (3.5,1.5) {\scriptsize$A_{i+1}$};  
  \node[anchor=east] at (1,1.7) {\scriptsize{$x_i$}}; 
  \node[anchor=east] at (2.1,0.7) {\scriptsize{$x_{i+1}$}}; 
 
   \filldraw[black] (1,2) circle (2.5pt);
   \filldraw[black] (2,1) circle (2.5pt);
   \node[anchor=center, rotate=-45] at (0.5,3.5) {$\ldots$};
   \node[anchor=center, rotate=-45] at (2.5,0.5) {$\ldots$};

\end{tikzpicture} 
 \end{tabular}
\end{center}
\caption{The structure of permutations containing occurrences of patterns in the proof of Theorem~\ref{thm-box-X_2}.}\label{fig-box-X_2}
\end{figure}
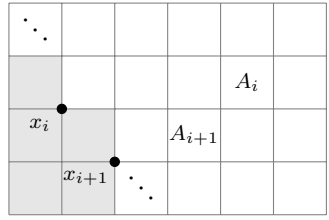

For each pair $\{X^{(2)}_i\}_{i=14}^{17}$, similarly to the proof of Theorem~\ref{thm-X^2_10}, observe that there is at most one occurrence of either pattern in a permutation. This observation allows us to use the bijection $f$ described in the proof of Theorem~\ref{thm-X^2_10} and provides the proofs of the following theorems, which are omitted.
 
\begin{thm}[Pairs $X^{(2)}_{14}$ and $Y^{(2)}_{14}$]
We have
$\pattern{scale = 0.5}{3}{1/1,2/2,3/3}{0/0,0/1,0/2,3/0,1/1,1/3,2/1,2/2,2/3,3/3}\sim_{jd}\hspace{-1mm}\pattern{scale = 0.5}{3}{1/1,2/3,3/2}{0/0,0/1,0/2,3/0,1/1,1/3,2/1,2/2,2/3,3/3}$ and $\pattern{scale = 0.5}{3}{1/1,2/2,3/3}{0/0,1/0,1/1,1/2,2/0,2/2,0/3,3/1,3/2,3/3}\sim_{jd}\hspace{-1mm}\pattern{scale = 0.5}{3}{1/1,2/3,3/2}{0/0,1/0,1/1,1/2,2/0,2/2,0/3,3/1,3/2,3/3}$.
\end{thm}

\begin{thm}[Pairs $X^{(2)}_{15}$ and $Y^{(2)}_{15}$]
We have $\pattern{scale = 0.5}{3}{1/1,2/2,3/3}{0/0,0/1,0/2,3/0,1/1,2/1,2/2,2/3,3/1,3/3}\sim_{jd}\hspace{-1mm}\pattern{scale = 0.5}{3}{1/1,2/3,3/2}{0/0,0/1,0/2,3/0,1/1,2/1,2/2,2/3,3/1,3/3}$ and $\pattern{scale = 0.5}{3}{1/1,2/2,3/3}{0/0,1/0,1/1,1/2,1/3,2/0,2/2,3/2,0/3,3/3}\sim_{jd}\hspace{-1mm}\pattern{scale = 0.5}{3}{1/1,2/3,3/2}{0/0,1/0,1/1,1/2,1/3,2/0,2/2,3/2,0/3,3/3}$.
\end{thm}

\begin{thm}[Pairs $X^{(2)}_{16}$ and $Y^{(2)}_{16}$]
We have $\pattern{scale = 0.5}{3}{1/1,2/2,3/3}{0/0,0/1,0/2,3/0,1/1,1/2,1/3,2/2,3/2,3/3}\sim_{jd}\hspace{-1mm}\pattern{scale = 0.5}{3}{1/1,2/3,3/2}{0/0,0/1,0/2,3/0,1/1,1/2,1/3,2/2,3/2,3/3}$ and $\pattern{scale = 0.5}{3}{1/1,2/2,3/3}{0/0,1/0,1/1,2/1,2/0,2/2,2/3,0/3,3/1,3/3}\sim_{jd}\hspace{-1mm}\pattern{scale = 0.5}{3}{1/1,2/3,3/2}{0/0,1/0,1/1,2/1,2/0,2/2,2/3,0/3,3/1,3/3}$.
\end{thm}

\begin{thm}[Pairs $X^{(2)}_{17}$ and $Y^{(2)}_{17}$]
We have $\pattern{scale = 0.5}{3}{1/1,2/2,3/3}{0/0,0/1,0/2,3/0,1/1,1/2,2/2,3/1,3/2,3/3}\sim_{jd}\hspace{-1mm}\pattern{scale = 0.5}{3}{1/1,2/3,3/2}{0/0,0/1,0/2,3/0,1/1,1/2,2/2,3/1,3/2,3/3}$ and $\pattern{scale = 0.5}{3}{1/1,2/2,3/3}{0/0,1/0,1/1,1/3,2/0,2/1,2/2,2/3,0/3,3/3}\sim_{jd}\hspace{-1mm}\pattern{scale = 0.5}{3}{1/1,2/3,3/2}{0/0,1/0,1/1,1/3,2/0,2/1,2/2,2/3,0/3,3/3}$. 
\end{thm}

\begin{remark}
Note that the proofs of the joint equidistribution results of type $X^{(2)}$ in this section are independent of whether  the area $(3,0)$ is shaded or not. This implies that the joint equistribution also holds for the type $\raisebox{0.5ex}{\begin{tikzpicture}[scale = 0.2, baseline=(current bounding box.center)]
	\foreach \x/\y in {0/0,0/1,0/2}		
		\fill[gray!50] (\x,\y) rectangle +(1,1); 	
        \draw (0,1)--(4,1);
        \draw (1,0)--(1,4);
        \draw (0,2)--(1,2);
         \draw (0,3)--(1,3);
         \draw (2,0)--(2,1);
         \draw (3,0)--(3,1);
 			\node  at (2.6,2.6) {$p$};
 			\filldraw (1,1) circle (4pt) node[above left] {};
 			\end{tikzpicture}}$,
with $p$ replaced by the mesh patterns of $p$ that are associated with $X^{(2)}$, and therefore also holds for the type $\raisebox{0.5ex}{\begin{tikzpicture}[scale = 0.2, baseline=(current bounding box.center)]
	\foreach \x/\y in {0/0,1/0,2/0}		
		\fill[gray!50] (\x,\y) rectangle +(1,1); 	
        \draw (0,1)--(4,1);
        \draw (1,0)--(1,4);
        \draw (0,2)--(1,2);
         \draw (0,3)--(1,3);
         \draw (2,0)--(2,1);
         \draw (3,0)--(3,1);
 			\node  at (2.6,2.6) {$p$};
 			\filldraw (1,1) circle (4pt) node[above left] {};
 			\end{tikzpicture}}$,
 with $p$ replaced by the mesh patterns $p$ associated with $Y^{(2)}$. We state these joint equidistribution results in Table~\ref{tab-extend}.
\end{remark}

\section{Mesh patterns of types $X^{(3)}$ and $Y^{(3)}$}\label{sec-X_3 and Y_3}
In this section, we consider the patterns with types $X^{(3)}=\raisebox{0.5ex}{\begin{tikzpicture}[scale = 0.2, baseline=(current bounding box.center)]
	\foreach \x/\y in {0/0,0/1,0/3,2/0}		
		\fill[gray!50] (\x,\y) rectangle +(1,1); 	
        \draw (0,1)--(4,1);
        \draw (1,0)--(1,4);
        \draw (0,2)--(1,2);
         \draw (0,3)--(1,3);
         \draw (2,0)--(2,1);
         \draw (3,0)--(3,1);
 			\node  at (2.6,2.6) {$p$};
 			\filldraw (1,1) circle (4pt) node[above left] {};
 			\end{tikzpicture}}$
and $Y^{(3)}=\raisebox{0.5ex}{\begin{tikzpicture}[scale = 0.2, baseline=(current bounding box.center)]
	\foreach \x/\y in {0/0,1/0,3/0,0/2}		
		\fill[gray!50] (\x,\y) rectangle +(1,1); 	
        \draw (0,1)--(4,1);
        \draw (1,0)--(1,4);
        \draw (0,2)--(1,2);
         \draw (0,3)--(1,3);
         \draw (2,0)--(2,1);
         \draw (3,0)--(3,1);
 			\node  at (2.6,2.6) {$p$};
 			\filldraw (1,1) circle (4pt) node[above left] {};
 			\end{tikzpicture}}$.
 Let $\pi=\pi_1\cdots \pi_n\in S_n$, and $x_1>x_2\cdots >x_s$ be the sequence of left-to-right minima in $\pi$.
Referring to Figure~\ref{fig-X_3 and Y_3}, each occurrence of any pattern of type $X^{(3)}$ in $\pi$ begins with an $x_i$, where $1\leq i\leq s$. Moreover, for each occurrence of $p_1$ (resp., $p_2$) in $\pi$, the first element must be in $A_{ij}$ (resp., $A_{1j}\cup B_{1j}\cup B_{2 j}\cup \cdots \cup B_{jj}$) because of the shaded area $(0,1)$ (resp., $(0,3)$) and the second element must be in $A_{1j}\cup B_{1j}\cup B_{2 j}\cup \cdots \cup B_{jj}$ (resp., $A_{ij}$) because of the shaded areas $(0,3)$ (resp., $(0,1)$) and $(2,0)$, where $1\leq i\leq s$ and $i\leq j\leq s$. That is, occurrences of the patterns must be entirely contained within $\{x_i\}\cup A_{ij}\cup A_{1j}$ or $\{x_i\}\cup A_{ij}\cup B_{mj}$ for $1\leq m\leq j$.  
 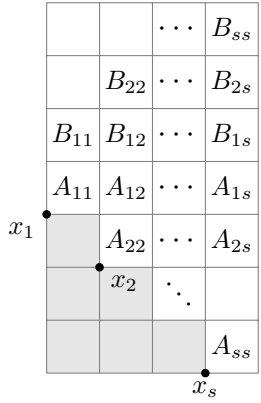
\begin{figure}
\begin{center}
\begin{tikzpicture}[scale=0.7]    
\tikzset{      
   grid/.style={        
      draw,        
      step=1cm,        
      gray!100,    
      thin,        
    },   
    cell/.style={      
      draw,      
      anchor=center,    
      text centered,    
     },    
    graycell/.style={   
      fill=gray!20,
      draw=none,     
      minimum width=1cm,   
      minimum height=1cm,     
      anchor=south west,            
    }    
  } 
      
  \fill[graycell] (0,0) rectangle (1,3);
  \fill[graycell] (1,0) rectangle (2,2);
  \fill[graycell] (2,0) rectangle (3,1);
  
     \draw[grid] (0,0) grid (4,7);
     \node at (0.5,3.5) {\footnotesize{$A_{11}$}};
     \node at (0.5,4.5) {\footnotesize{$B_{11}$}};
     \node at (1.5,3.5) {\footnotesize{$A_{12}$}};
     \node at (1.5,4.5) {\footnotesize{$B_{12}$}};
     \node at (3.5,3.5) {\footnotesize{$A_{1s}$}};
     \node at (3.5,4.5) {\footnotesize{$B_{1s}$}};

     \node at (1.5,2.5) {\footnotesize{$A_{22}$}};
     \node at (1.5,5.5) {\footnotesize{$B_{22}$}};
     \node at (3.5,2.5) {\footnotesize{$A_{2s}$}};
     \node at (3.5,5.5) {\footnotesize{$B_{2s}$}};
     
     \node at (3.5,0.5) {\footnotesize{$A_{ss}$}};
     \node at (3.5,6.5) {\footnotesize{$B_{ss}$}};
  \node[anchor=east] at (0,2.7) {\footnotesize{$x_1$}};
  \node[anchor=west] at (1,1.7) {\footnotesize{$x_2$}};
  \node[anchor=center] at (3,-0.3) {\footnotesize{$x_s$}};
    \filldraw (0,3) circle (2pt);
    \filldraw (1,2) circle (2pt);
    \filldraw (3,0) circle (2pt);
     \node[anchor=center, rotate=-45] at (2.5,1.5) {$\ldots$};
     \node[anchor=center, rotate=0] at (2.5,3.5) {$\cdots$};
     \node[anchor=center, rotate=0] at (2.5,2.5) {$\cdots$};
     \node[anchor=center, rotate=0] at (2.5,4.5) {$\cdots$};
     \node[anchor=center, rotate=0] at (2.5,5.5) {$\cdots$};
     \node[anchor=center, rotate=0] at (2.5,6.5) {$\cdots$};
\end{tikzpicture}
\end{center}

\vspace{-0.8cm}

\caption{The structure of permutations containing occurrences of pairs in Section~\ref{sec-X_3 and Y_3}.}\label{fig-X_3 and Y_3}
\end{figure}

\subsection{Cases proved via reversal}\label{subsec-X_3-1}
In this subsection, we consider the following 16 pairs presented in Table~\ref{tab-X_3-1}. Each shading of the pattern $p$ associated with $\{X^{(3)}_i\}_{i=1}^{8}$ exhibits vertical symmetry.
 Referring to Figure~\ref{fig-X_3 and Y_3}, let $C_j=A_{1j}\cup A_{2j}\cup\cdots\cup A_{jj}\cup B_{1j}\cup B_{2j}\cup\cdots \cup B_{jj}$, where $j=\{1,2,\ldots,s\}$. We establish the bijection $f$ by applying the reverse operation to $C_i$'s, which is a bijection that maps $\T(n,k,\ell)$ to $\T(n,\ell,k)$.
 
\begin{table}[!ht]
 	{
 		\renewcommand{\arraystretch}{1.7}
 		\setlength{\tabcolsep}{5pt}
 \begin{center} 
 		\begin{tabular}{c | c || c | c || c | c || c | c}
 			\hline
 		\footnotesize{Nr.}	& \footnotesize {Patterns} & \footnotesize{Nr.}	& \footnotesize{Patterns} & \footnotesize {Nr.}  & \footnotesize{Patterns} &\footnotesize {Nr.}  & \footnotesize{Patterns}\\
 			\hline
 			\hline
$X^{(3)}_{1}$ &  $\pattern{scale = 0.5}{3}{1/1,2/2,3/3}{0/0,0/1,0/3,2/0,1/1,2/1,2/2,2/3,3/1}\pattern{scale = 0.5}{3}{1/1,2/3,3/2}{0/0,0/1,0/3,2/0,1/1,2/1,2/2,2/3,3/1}$ &
$X^{(3)}_2$ & $\pattern{scale = 0.5}{3}{1/1,2/2,3/3}{0/0,0/1,0/3,2/0,1/3,2/1,2/2,2/3,3/3}\pattern{scale = 0.5}{3}{1/1,2/3,3/2}{0/0,0/1,0/3,2/0,1/3,2/1,2/2,2/3,3/3}$& 
$X^{(3)}_3$ & $\pattern{scale = 0.5}{3}{1/1,2/2,3/3}{0/0,0/1,0/3,2/0,1/1,1/2,2/2,3/1,3/2}\pattern{scale = 0.5}{3}{1/1,2/3,3/2}{0/0,0/1,0/3,2/0,1/1,1/2,2/2,3/1,3/2}$& 
$X^{(3)}_4$ & $\pattern{scale = 0.5}{3}{1/1,2/2,3/3}{0/0,0/1,0/3,2/0,1/2,1/3,2/2,3/2,3/3}\pattern{scale = 0.5}{3}{1/1,2/3,3/2}{0/0,0/1,0/3,2/0,1/2,1/3,2/2,3/2,3/3}$ 
\\[5pt]
$X^{(3)}_5$ & $\pattern{scale = 0.5}{3}{1/1,2/2,3/3}{0/0,0/1,1/1,1/2,0/3,2/0,3/1,3/2}\pattern{scale = 0.5}{3}{1/1,2/3,3/2}{0/0,0/1,1/1,1/2,0/3,2/0,3/1,3/2}$ &
$X^{(3)}_{6}$ &  $\pattern{scale = 0.5}{3}{1/1,2/2,3/3}{0/0,0/1,1/2,1/3,0/3,2/0,3/2,3/3}\pattern{scale = 0.5}{3}{1/1,2/3,3/2}{0/0,0/1,1/2,1/3,0/3,2/0,3/2,3/3}$ &
$X^{(3)}_{7}$ & $\pattern{scale = 0.5}{3}{1/1,2/2,3/3}{0/0,0/1,1/1,2/1,0/3,2/0,3/1,2/3}\pattern{scale = 0.5}{3}{1/1,2/3,3/2}{0/0,0/1,1/1,2/1,0/3,2/0,3/1,2/3}$ &
$X^{(3)}_{8}$ & $\pattern{scale = 0.5}{3}{1/1,2/2,3/3}{0/0,0/1,1/3,0/3,2/0,2/1,2/3,3/3}\pattern{scale = 0.5}{3}{1/1,2/3,3/2}{0/0,0/1,1/3,0/3,2/0,2/1,2/3,3/3}$
\\[5pt]
\hline
$Y^{(3)}_1$ &  $\pattern{scale = 0.5}{3}{1/1,2/2,3/3}{0/0,1/0,1/1,1/2,1/3,3/0,2/2,0/2,3/2}\pattern{scale = 0.5}{3}{1/1,2/3,3/2}{0/0,1/0,1/1,1/2,1/3,3/0,0/2,2/2,3/2}$ &
$Y^{(3)}_2$ &  $\pattern{scale = 0.5}{3}{1/1,2/2,3/3}{0/0,1/0,1/2,2/2,3/0,0/2,3/3,3/1,3/2}\pattern{scale = 0.5}{3}{1/1,2/3,3/2}{0/0,1/0,1/2,3/0,0/2,2/2,3/3,3/1,3/2}$&
$Y^{(3)}_3$ & $\pattern{scale = 0.5}{3}{1/1,2/2,3/3}{0/0,1/0,1/1,1/3,3/0,0/2,2/1,2/2,2/3}\pattern{scale = 0.5}{3}{1/1,2/3,3/2}{0/0,1/0,1/1,1/3,2/1,2/2,2/3,3/0,0/2}$&
$Y^{(3)}_4$ & $\pattern{scale = 0.5}{3}{1/1,2/2,3/3}{0/0,1/0,2/1,2/2,2/3,3/0,0/2,3/1,3/3}\pattern{scale = 0.5}{3}{1/1,2/3,3/2}{0/0,1/0,2/1,2/2,2/3,3/0,0/2,3/1,3/3}$
\\[5pt]
$Y^{(3)}_5$ & $\pattern{scale = 0.5}{3}{1/1,2/2,3/3}{0/0,1/0,1/1,2/1,2/3,3/0,0/2,1/3}\pattern{scale = 0.5}{3}{1/1,2/3,3/2}{0/0,1/0,1/1,2/1,2/3,3/0,0/2,1/3}$ &
$Y^{(3)}_{6}$ & $\pattern{scale = 0.5}{3}{1/1,2/2,3/3}{0/0,1/0,3/1,2/1,2/3,3/0,0/2,3/3}\pattern{scale = 0.5}{3}{1/1,2/3,3/2}{0/0,1/0,3/1,2/1,2/3,3/0,0/2,3/3}$ &
$Y^{(3)}_{7}$ &$\pattern{scale = 0.5}{3}{1/1,2/2,3/3}{0/0,1/0,1/1,1/2,1/3,3/0,0/2,3/2}\pattern{scale = 0.5}{3}{1/1,2/3,3/2}{0/0,1/0,1/1,1/2,1/3,3/0,0/2,3/2}$& 
$Y^{(3)}_{8}$ &  $\pattern{scale = 0.5}{3}{1/1,2/2,3/3}{0/0,1/0,1/2,3/1,3/2,3/0,0/2,3/3}\pattern{scale = 0.5}{3}{1/1,2/3,3/2}{0/0,1/0,1/2,3/1,3/2,3/0,0/2,3/3}$ \\[5pt]
\hline
	\end{tabular}
\end{center} 
}
\vspace{-0.5cm}
 	\caption{Jointly equidistributed patterns of types $X^{(3)}$ and $Y^{(3)}$ via reversal.}\label{tab-X_3-1}
\end{table} 

\subsection{Cases proved via swapping elements}\label{subsec-X_3-2}
In this subsection, we consider the following two pairs shown in Table~\ref{tab-X_3-2}.

\begin{table}[!ht]
 	{
 		\renewcommand{\arraystretch}{1.7}
 		\setlength{\tabcolsep}{5pt}
 \begin{center} 
 		\begin{tabular}{c | c || c | c }
 			\hline
 		\footnotesize{Nr.}	& \footnotesize {Patterns} & \footnotesize{Nr.}	& \footnotesize{Patterns} \\
 			\hline
 			\hline
$X^{(3)}_{9}$ &  $\pattern{scale = 0.5}{3}{1/1,2/2,3/3}{0/0,0/1,0/3,2/0,1/1,1/3,2/1,2/2,2/3,3/3}\pattern{scale = 0.5}{3}{1/1,2/3,3/2}{0/0,0/1,0/3,2/0,1/1,1/3,2/1,2/2,2/3,3/3}$ &
$Y^{(3)}_9$ & $\pattern{scale = 0.5}{3}{1/1,2/2,3/3}{0/0,1/0,3/0,0/2,1/1,1/2,3/1,2/2,3/2,3/3}\pattern{scale = 0.5}{3}{1/1,2/3,3/2}{0/0,1/0,3/0,0/2,1/1,1/2,3/1,2/2,3/2,3/3}$
\\[5pt]
\hline
	\end{tabular}
\end{center} 
}
\vspace{-0.5cm}
 	\caption{Jointly equidistributed patterns of types $X^{(3)}$ and $Y^{(3)}$ via swapping elements.}\label{tab-X_3-2}
\end{table}

\begin{thm}[Pairs $X^{(3)}_{9}$ and $Y^{(3)}_{9}$]\label{thm-Ding-X_3}
    We have $\pattern{scale = 0.5}{3}{1/1,2/2,3/3}{0/0,0/1,0/3,2/0,1/1,1/3,2/1,2/2,2/3,3/3}\sim_{jd}\hspace{-1mm}\pattern{scale = 0.5}{3}{1/1,2/3,3/2}{0/0,0/1,0/3,2/0,1/1,1/3,2/1,2/2,2/3,3/3}$ and $\pattern{scale = 0.5}{3}{1/1,2/2,3/3}{0/0,1/0,3/0,0/2,1/1,1/2,3/1,2/2,3/2,3/3}\sim_{jd}\hspace{-1mm}\pattern{scale = 0.5}{3}{1/1,2/3,3/2}{0/0,1/0,3/0,0/2,1/1,1/2,3/1,2/2,3/2,3/3}$.
\end{thm}
\begin{proof}
Considering the pair  $X^{(3)}_9$, let $x_i \pi_j\pi_t$ be an occurrence of $q_1$ or $q_2$ in a permutation $\pi\in S_n$, where $1\leq i\leq s$, $2\leq j,t\leq n$. We have $t=j+1$ and $\pi_t=n$ (resp., $\pi_j=n$) for $q_1$ (resp., $q_2$). Therefore, there is at most one occurrence of $p_1$ or $p_2$ in $\pi$. Furthermore, w.l.o.g., suppose $x_\ell\pi_j\pi_t$ is an occurrence of $q_1$ or $q_2$ with $i<\ell\leq s$; then $x_\ell$ is positioned between $x_i$ and $\pi_j$ because of the shaded area $(2,0)$, which contradicts our supposition due to the element $x_i$. Thus, each permutation contains at most one occurrence of either pattern in $X^{(3)}_{9}$. Therefore, the map $f$ is constructed in the same way as in Theorem~\ref{thm-X^2_10}.\end{proof}

\section{Mesh patterns of types $X^{(4)}$ and $Y^{(4)}$}\label{sec-X_4}

In this section, we consider the patterns with types $X^{(4)}=\raisebox{0.6ex}{\begin{tikzpicture}[scale = 0.2, baseline=(current bounding box.center)]
	\foreach \x/\y in {0/0,0/1,2/0,3/0}		
		\fill[gray!50] (\x,\y) rectangle +(1,1); 	
        \draw (0,1)--(4,1);
        \draw (1,0)--(1,4);
        \draw (0,2)--(1,2);
         \draw (0,3)--(1,3);
         \draw (2,0)--(2,1);
         \draw (3,0)--(3,1);
 			\node  at (2.6,2.6) {$p$};
 			\filldraw (1,1) circle (4pt) node[above left] {};
 			\end{tikzpicture}}$ and $Y^{(4)}=\raisebox{0.6ex}{\begin{tikzpicture}[scale = 0.2, baseline=(current bounding box.center)]
	\foreach \x/\y in {0/0,1/0,0/2,0/3}		
		\fill[gray!50] (\x,\y) rectangle +(1,1); 	
        \draw (0,1)--(4,1);
        \draw (1,0)--(1,4);
        \draw (0,2)--(1,2);
         \draw (0,3)--(1,3);
         \draw (2,0)--(2,1);
         \draw (3,0)--(3,1);
 			\node  at (2.6,2.6) {$p$};
 			\filldraw (1,1) circle (4pt) node[above left] {};
 			\end{tikzpicture}}$. 
For the patterns of type $X^{(4)}$, let $\pi \in S_n$, and let $x_1>x_2>\cdots>x_t$ be the sequence of left-to-right minima in $\pi$. Any occurrence of a pattern in each pair of these patterns begins with an $x_i$.

\subsection{Cases proved via swapping elements}\label{subsec-5.1}
This subsection is about the 10 pairs presented in Table~\ref{tab-X_4-1}.
\begin{table}[!ht]
 	{
 		\renewcommand{\arraystretch}{1.7}
 		\setlength{\tabcolsep}{5pt}
 \begin{center} 
 		\begin{tabular}{c | c || c | c || c | c || c | c || c | c}
 			\hline
 		\footnotesize{Nr.}	& \footnotesize {Patterns} & \footnotesize{Nr.}	& \footnotesize{Patterns} & \footnotesize {Nr.}  & \footnotesize{Patterns} &\footnotesize {Nr.}  & \footnotesize{Patterns} &\footnotesize {Nr.}  & \footnotesize{Patterns} \\
 			\hline
 			\hline
$X^{(4)}_{1}$ &  $\pattern{scale = 0.5}{3}{1/1,2/2,3/3}{0/0,0/1,2/0,3/0,1/3,2/1,2/2,2/3,3/3}\pattern{scale = 0.5}{3}{1/1,2/3,3/2}{0/0,0/1,2/0,3/0,1/3,2/1,2/2,2/3,3/3}$&
$X^{(4)}_2$ & $\pattern{scale = 0.5}{3}{1/1,2/2,3/3}{0/0,0/1,1/2,2/2,2/0,3/0,3/3,3/1,3/2}\pattern{scale = 0.5}{3}{1/1,2/3,3/2}{0/0,0/1,1/2,2/0,3/0,2/2,3/3,3/1,3/2}$ &  
$X^{(4)}_3$ &$\pattern{scale = 0.5}{3}{1/1,2/2,3/3}{0/0,0/1,3/0,2/0,1/1,1/3,2/1,2/2,2/3,3/3}\pattern{scale = 0.5}{3}{1/1,2/3,3/2}{0/0,0/1,3/0,2/0,1/1,1/3,2/1,2/2,2/3,3/3}$&
$X^{(4)}_{4}$ &  $\pattern{scale = 0.5}{3}{1/1,2/2,3/3}{0/0,0/1,1/1,2/0,2/1,2/2,2/3,3/0,3/1,3/3}\pattern{scale = 0.5}{3}{1/1,2/3,3/2}{0/0,0/1,1/1,2/0,2/1,2/2,2/3,3/0,3/1,3/3}$ &
$X^{(4)}_{5}$ &  $\pattern{scale = 0.5}{3}{1/1,2/2,3/3}{0/0,0/1,3/0,2/0,1/1,1/2,2/2,3/1,3/2,3/3}\pattern{scale = 0.5}{3}{1/1,2/3,3/2}{0/0,0/1,3/0,2/0,1/1,1/2,2/2,3/1,3/2,3/3}$ \\[5pt]
\hline
$Y^{(4)}_1$ &  $\pattern{scale = 0.5}{3}{1/1,2/2,3/3}{0/0,0/2,0/3,1/0,1/2,2/2,3/1,3/2,3/3}\pattern{scale = 0.5}{3}{1/1,2/3,3/2}{0/0,0/2,0/3,1/0,1/2,2/2,3/1,3/2,3/3}$ & 
$Y^{(4)}_2$ &  $\pattern{scale = 0.5}{3}{1/1,2/2,3/3}{0/0,1/0,1/3,2/1,2/2,2/3,0/3,0/2,3/3}\pattern{scale = 0.5}{3}{1/1,2/3,3/2}{0/0,1/0,1/3,2/1,2/2,2/3,0/3,0/2,3/3}$&
$Y^{(4)}_3$ & $\pattern{scale = 0.5}{3}{1/1,2/2,3/3}{0/0,0/2,0/3,1/0,1/1,1/2,2/2,3/1,3/2,3/3}\pattern{scale = 0.5}{3}{1/1,2/3,3/2}{0/0,0/2,0/3,1/0,1/1,1/2,2/2,3/1,3/2,3/3}$ &
$Y^{(4)}_{4}$ &$\pattern{scale = 0.5}{3}{1/1,2/2,3/3}{0/0,0/2,0/3,1/0,1/1,1/2,1/3,2/2,3/2,3/3}\pattern{scale = 0.5}{3}{1/1,2/3,3/2}{0/0,0/2,0/3,1/0,1/1,1/2,1/3,2/2,3/2,3/3}$ &
$Y^{(4)}_{5}$ &$\pattern{scale = 0.5}{3}{1/1,2/2,3/3}{0/0,0/2,0/3,1/0,1/1,1/3,2/1,2/2,2/3,3/3}\pattern{scale = 0.5}{3}{1/1,2/3,3/2}{0/0,0/2,0/3,1/0,1/1,1/3,2/1,2/2,2/3,3/3}$
\\[5pt]
\hline
	\end{tabular}
\end{center} 
}
\vspace{-0.5cm}
 	\caption{Jointly equidistributed patterns of types $X^{(4)}$ and $Y^{(4)}$ via swapping elements.}\label{tab-X_4-1}
\end{table} 

      


\begin{thm}[Pairs $X^{(4)}_1$ and $Y^{(4)}_1$]\label{thm-X^4-1}
We have $\pattern{scale = 0.5}{3}{1/1,2/2,3/3}{0/0,0/1,2/0,3/0,1/3,2/1,2/2,2/3,3/3}\sim_{jd}\hspace{-1mm}\pattern{scale = 0.5}{3}{1/1,2/3,3/2}{0/0,0/1,2/0,3/0,1/3,2/1,2/2,2/3,3/3}$ and $\pattern{scale = 0.5}{3}{1/1,2/2,3/3}{0/0,0/2,0/3,1/0,1/2,2/2,3/1,3/2,3/3}\sim_{jd}\hspace{-1mm}\pattern{scale = 0.5}{3}{1/1,2/3,3/2}{0/0,0/2,0/3,1/0,1/2,2/2,3/1,3/2,3/3}$.
\end{thm}
\begin{proof}
 For the pair $X^{(4)}_1$, if $x_i ab$ is an occurrence of $q_1$ (resp., $q_2$) in a permutation, where $1\leq i\leq t$, then $b$ (resp., $a$) is the maximal element following $x_i$, and $a$ is immediately to the left of $b$ (resp., $b$ is immediately to the right of $a$). We claim that each permutation contains at most one occurrence of $q_1$ or $q_2$. Otherwise, suppose $x_j a'b'$ is an occurrence of $q_1$ that differs from $x_i ab$ (a similar consideration applies for the case of $q_2$).
 \begin{enumerate}
 \item[(i)] If $x_i=x_j$, then $b'=b$, which also leads to $a'=a$, a contradiction.
 \item[(ii)] If $x_i\neq x_j$, then assume, w.l.o.g.,  $i<j\leq t$. We see that $x_j$ is positioned between $x_i$ and $a$ due to the shaded areas $(2,0)$ and $(3,0)$. This will result in $b'=b$, and thereby $a'=a$, also a contradiction. 
\end{enumerate}
We proved our claim. Thus, we can use the same map $f$ as described in the proof of Theorem~\ref{thm-X^2_10} to complete our proof.
\end{proof}

\begin{thm}[Pairs $X^{(4)}_2$ and $Y^{(4)}_2$]
We have $\pattern{scale = 0.5}{3}{1/1,2/2,3/3}{0/0,0/1,1/2,2/2,2/0,3/0,3/3,3/1,3/2}\sim_{jd}\hspace{-1mm}\pattern{scale = 0.5}{3}{1/1,2/3,3/2}{0/0,0/1,1/2,2/0,3/0,2/2,3/3,3/1,3/2}$ and $\pattern{scale = 0.5}{3}{1/1,2/2,3/3}{0/0,1/0,1/3,2/1,2/2,2/3,0/3,0/2,3/3}\sim_{jd}\hspace{-1mm}\pattern{scale = 0.5}{3}{1/1,2/3,3/2}{0/0,1/0,1/3,2/1,2/2,2/3,0/3,0/2,3/3}$.
\end{thm}
\begin{proof}
 For the pair $X^{(4)}_2$, similarly to the proof of Theorem~\ref{thm-X^4-1}, there is at most one occurrence of $q_1$ or $q_2$ in $\pi$. Suppose $x_i ab$ and $x_j a'b'$ are two different occurrences of $q_1$ in a permutation $\pi\in S_n$. It must be the case that $b'=b=\pi_n$, which also results in $x_i\neq x_j$; otherwise, $a'=a$. This will also lead to a contradiction because of the element $x_i$. Therefore, we also use the same map $f$ as in the proof of Theorem~\ref{thm-X^2_10}, which completes our proof.
\end{proof}

For each pair $\{X^{(4)}_i\}_{i=3}^{5}$, similarly to the proof of Theorem~\ref{thm-X^4-1}, observe that there is at most one occurrence of either pattern in a permutation. This observation allows us to use the bijection $f$ described in the proof of Theorem~\ref{thm-X^2_10} and provides the proofs of the following theorems, which are omitted.

\begin{thm}[Pairs $X^{(4)}_3$ and $Y^{(4)}_3$]\label{thm-Ding-4-1}
We have $\pattern{scale = 0.5}{3}{1/1,2/2,3/3}{0/0,0/1,3/0,2/0,1/1,1/3,2/1,2/2,2/3,3/3}\sim_{jd}\hspace{-1mm}\pattern{scale = 0.5}{3}{1/1,2/3,3/2}{0/0,0/1,3/0,2/0,1/1,1/3,2/1,2/2,2/3,3/3}$ and $\pattern{scale = 0.5}{3}{1/1,2/2,3/3}{0/0,0/2,0/3,1/0,1/1,1/2,2/2,3/1,3/2,3/3}\sim_{jd}\hspace{-1mm}\pattern{scale = 0.5}{3}{1/1,2/3,3/2}{0/0,0/2,0/3,1/0,1/1,1/2,2/2,3/1,3/2,3/3}$.
\end{thm}

\begin{thm}[Pairs $X^{(4)}_4$ and $X^{(4)}_4$]
We have $\pattern{scale = 0.5}{3}{1/1,2/2,3/3}{0/0,0/1,1/1,2/0,2/1,2/2,2/3,3/0,3/1,3/3}\sim_{jd}\hspace{-1mm}\pattern{scale = 0.5}{3}{1/1,2/3,3/2}{0/0,0/1,1/1,2/0,2/1,2/2,2/3,3/0,3/1,3/3}$ and $\pattern{scale = 0.5}{3}{1/1,2/2,3/3}{0/0,0/2,0/3,1/0,1/1,1/2,1/3,2/2,3/2,3/3}\sim_{jd}\hspace{-1mm}\pattern{scale = 0.5}{3}{1/1,2/3,3/2}{0/0,0/2,0/3,1/0,1/1,1/2,1/3,2/2,3/2,3/3}$.
\end{thm}

\begin{thm}[Pairs $X^{(4)}_5$ and $X^{(4)}_5$]
We have $\pattern{scale = 0.5}{3}{1/1,2/2,3/3}{0/0,0/1,3/0,2/0,1/1,1/2,2/2,3/1,3/2,3/3}\sim_{jd}\hspace{-1mm}\pattern{scale = 0.5}{3}{1/1,2/3,3/2}{0/0,0/1,3/0,2/0,1/1,1/2,2/2,3/1,3/2,3/3}$ and $\pattern{scale = 0.5}{3}{1/1,2/2,3/3}{0/0,0/2,0/3,1/0,1/1,1/3,2/1,2/2,2/3,3/3}\sim_{jd}\hspace{-1mm}\pattern{scale = 0.5}{3}{1/1,2/3,3/2}{0/0,0/2,0/3,1/0,1/1,1/3,2/1,2/2,2/3,3/3}$.
\end{thm}

\subsection{Cases proved via recurrence relations}\label{subsec-X_4-2}
This subsection is about the two pairs presented in Table~\ref{tab-X_4-3}.

\begin{table}[!ht]
 	{
 		\renewcommand{\arraystretch}{1.7}
 		\setlength{\tabcolsep}{5pt}
 \begin{center} 
 		\begin{tabular}{c | c || c | c }
 			\hline
 		\footnotesize{Nr.}	& \footnotesize {Patterns} & \footnotesize{Nr.}	& \footnotesize{Patterns} \\
 			\hline
 			\hline
$X^{(4)}_6$&  $\pattern{scale = 0.5}{3}{1/1,2/2,3/3}{0/0,0/1,3/0,2/0,2/1,2/2,3/1,3/2}\pattern{scale = 0.5}{3}{1/1,2/3,3/2}{0/0,0/1,3/0,2/0,2/1,2/2,3/1,3/2}$ &
$Y^{(4)}_6$ & $\pattern{scale = 0.5}{3}{1/1,2/2,3/3}{0/0,0/2,1/0,0/3,1/2,1/3,2/2,2/3}\pattern{scale = 0.5}{3}{1/1,2/3,3/2}{0/0,0/2,1/0,0/3,1/2,1/3,2/2,2/3}$
\\[5pt]
\hline
	\end{tabular}
\end{center} 
}
\vspace{-0.5cm}
 	\caption{Jointly equidistributed patterns of types $X^{(4)}$ and $Y^{(4)}$ via recurrence relations.}\label{tab-X_4-3}
\end{table}

\begin{thm}[Pairs $X^{(4)}_6$ and $Y^{(4)}_6$]
We have $\pattern{scale = 0.5}{3}{1/1,2/2,3/3}{0/0,0/1,2/0,2/1,2/2,3/0,3/1,3/2}\sim_{jd}\hspace{-1mm}\pattern{scale = 0.5}{3}{1/1,2/3,3/2}{0/0,0/1,2/0,2/1,2/2,3/0,3/1,3/2}$, and $\pattern{scale = 0.5}{3}{1/1,2/2,3/3}{0/0,0/2,1/0,0/3,1/2,1/3,2/2,2/3}\sim_{jd}\hspace{-1mm}\pattern{scale = 0.5}{3}{1/1,2/3,3/2}{0/0,0/2,1/0,0/3,1/2,1/3,2/2,2/3}$. Moreover, both pairs have the same joint distribution with the pairs $\{X^{(1)}_i\}_{i=9}^{18}$ and $\{Y^{(1)}_i\}_{i=9}^{18}$.
\end{thm}
\begin{proof}
By Theorem~\ref{coro-pairs-17-36}, it suffices to prove $X^{(4)}_6$ satisfies~\eqref{recur-pair 8-17} and \eqref{genera-pair 8 17}.
Let
\[
T^{(1)}_{n,k,\ell}:=\{\pi\in \T_{n,k,\ell}, \text{ }\pi_n=1\},
\]
\[
T^{(2)}_{n,k,\ell}:=\{\pi\in \T_{n,k,\ell}, \text{ }\pi_n\neq1\}.
\]
It is clear that 
\begin{align}\label{T_n=T1+T2-pair 18}
T_n(x,y)=T^{(1)}_n(x,y)+T^{(2)}_n(x,y).
\end{align}
We claim that
\begin{align}
T^{(1)}_n(x,y)=&(n-1)!,\label{recur-T1-pair 18}\\
T^{(2)}_n(x,y)=&(n+x+y-2)T_{n-1}(x,y)+(3-n-x-y)(n-2)!.\label{recur-T2-pair 18}
\end{align} 
To explain the two equations, we consider generating all $n$-permutations from $(n-1)$-permutations by inserting the largest element $n$. We have:
\begin{itemize}
    \item For~\eqref{recur-T1-pair 18}, each permutation $\pi$ counted by the left-hand side can only be obtained from the case where $\pi_n=1$, by inserting $n$ in any slot except immediately after the last entry. Therefore, this case is given by $(n-1)T^{(1)}_{n-1}(x,y)$. Since $T^{(1)}_2(x,y)=1$, we obtain~\eqref{recur-T1-pair 18}. 
    \item For~\eqref{recur-T2-pair 18}, each permutation $\pi$ counted by the left-hand side can be derived from the following cases: 
    \begin{enumerate}
        \item [(a)] If $\pi_n=1$, then inserting $n$ immediately after the last entry results in $\pi_n\neq1$. This insertion preserves the number of the occurrence of both patterns, thereby giving $T^{(1)}_{n-1}(x,y)$.
        \item [(b)] If $\pi_n\neq 1$, then inserting $n$  immediately after the last entry leads to an extra occurrence of $q_1$, in the form of $\pi_j\pi_{n-1}n$, where $\pi_j$ represents the leftmost element less than $\pi_{n-1}$. Thus, this case gives $xT^{(2)}_{n-1}(x,y)$. Similarly, $yT^{(2)}_{n-1}(x,y)$ corresponds to the insertion of $n$ immediately before the last entry, which results in an additional occurrence of $q_1$ formed by $\pi_jn\pi_n$. In any other slot, the insertion of $n$ preserves the number of occurrences of both patterns, contributing to $(n-2)T^{(2)}_{n-1}(x,y)$. 
    \end{enumerate} 
    Summing over the above two cases and using~\eqref{recur-T1-pair 18}, we obtain~\eqref{recur-T2-pair 18}.
\end{itemize}
Combining \eqref{T_n=T1+T2-pair 18}, \eqref{recur-T1-pair 18}, and \eqref{recur-T2-pair 18}, we obtain 
\begin{align}\label{genera-box-4}
T_n(x,y)=(n+x+y-2)\,T_{n-1}(x,y)+(2-x-y)\,(n-2)!.
\end{align}
It is not difficult to verify that \eqref{genera-pair 8 17} satisfies \eqref{genera-box-4} for any nonnegative integer \( n \). Thus, the proof is complete.
\end{proof}

\section{Mesh patterns with other minus antipodal shadings}\label{sec-other}
In this section, we consider four pairs with other shadings as listed in Table~\ref{tab-other-proved-shadings}.
\begin{table}[!ht]
 	{
 		\renewcommand{\arraystretch}{1.7}
 		\setlength{\tabcolsep}{5pt}
 \begin{center} 
 		\begin{tabular}{c | c ||  c | c || c | c || c | c}
 			\hline
 		\footnotesize{Nr.}	& \footnotesize {Patterns} & \footnotesize{Nr.}	& \footnotesize{Patterns} & \footnotesize{Nr.}	& \footnotesize {Patterns} & \footnotesize{Nr.}	& \footnotesize{Patterns} \\
 			\hline
 			\hline
113& $\pattern{scale = 0.5}{3}{1/1,2/2,3/3}{0/2,0/1,3/0,1/1,1/2,2/2,3/1,3/2,3/3}\pattern{scale = 0.5}{3}{1/1,2/3,3/2}{0/2,0/1,3/0,1/1,1/2,2/2,3/1,3/2,3/3}$ &
114& $\pattern{scale = 0.5}{3}{1/1,2/2,3/3}{0/3,1/0,1/1,1/3,2/0,2/1,2/2,2/3,3/3}\pattern{scale = 0.5}{3}{1/1,2/3,3/2}{0/3,1/0,1/1,1/3,2/0,2/1,2/2,2/3,3/3}$ &
115& $\pattern{scale = 0.5}{3}{1/1,2/2,3/3}{0/2,1/0,1/1,1/2,2/2,3/0,3/1,3/2,3/3}\pattern{scale = 0.5}{3}{1/1,2/3,3/2}{0/2,1/0,1/1,1/2,2/2,3/0,3/1,3/2,3/3}$ &
116& $\pattern{scale = 0.5}{3}{1/1,2/2,3/3}{0/3,0/1,2/0,1/1,1/3,2/1,2/2,2/3,3/3}\pattern{scale = 0.5}{3}{1/1,2/3,3/2}{0/3,0/1,2/0,1/1,1/3,2/1,2/2,2/3,3/3}$
\\[5pt]
\hline
	\end{tabular}
\end{center} 
}
\vspace{-0.5cm}
 	\caption{Joint equidistributions of other shadings.}\label{tab-other-proved-shadings}
\end{table}

\begin{thm}[Pairs~113 and 114]
    We have $\pattern{scale = 0.5}{3}{1/1,2/2,3/3}{0/2,0/1,3/0,1/1,1/2,2/2,3/1,3/2,3/3}\sim_{jd}\hspace{-1mm}\pattern{scale = 0.5}{3}{1/1,2/3,3/2}{0/2,0/1,3/0,1/1,1/2,2/2,3/1,3/2,3/3}$ and $\pattern{scale = 0.5}{3}{1/1,2/2,3/3}{0/3,1/0,1/1,1/3,2/0,2/1,2/2,2/3,3/3}\sim_{jd}\hspace{-1mm}\pattern{scale = 0.5}{3}{1/1,2/3,3/2}{0/3,1/0,1/1,1/3,2/0,2/1,2/2,2/3,3/3}$. 
\end{thm}
\begin{proof}
 Observe that pairs~113 and~114 can be obtained from each other by the inverse operation. Hence, it is sufficient to consider the pair~113. Any occurrence of a pattern in the pair in a permutation $\pi=\pi_1\cdots\pi_n$ ends with $\pi_n$. Moreover, $\pi_i\pi_j\pi_n$ is an occurrence of $q_1$ (resp., $q_2$) if and only if $\pi_j=\pi_n-1$ (resp., $\pi_j=\pi_n+1$), and $\pi_i$ is the largest element to the left of, and smaller than $\pi_j$, {\em i.e., $\pi_i=\max\{\pi_s: \pi_s<\pi_j, s<j\}$}. 

Hence, each permutation contains at most one occurrence of either pattern. The map $f$ can then be constructed in the same way as in the proof of Theorem~\ref{thm-X^2_10} to complete our proof. \end{proof}

\begin{thm}[Pairs~115 and 116]
We have $\pattern{scale = 0.5}{3}{1/1,2/2,3/3}{0/2,1/0,1/1,1/2,2/2,3/0,3/1,3/2,3/3}\sim_{jd}\hspace{-1mm}\pattern{scale = 0.5}{3}{1/1,2/3,3/2}{0/2,1/0,1/1,1/2,2/2,3/0,3/1,3/2,3/3}$ and $\pattern{scale = 0.5}{3}{1/1,2/2,3/3}{0/3,0/1,2/0,1/1,1/3,2/1,2/2,2/3,3/3}\sim_{jd}\hspace{-1mm}\pattern{scale = 0.5}{3}{1/1,2/3,3/2}{0/3,0/1,2/0,1/1,1/3,2/1,2/2,2/3,3/3}$.
\end{thm}

\begin{proof}
Observe that pairs~115 and~116 can be obtained from each other by the inverse operation. Hence, it is sufficient to consider the pair~115. Any occurrence of a pattern in the pair in a permutation $\pi=\pi_1\cdots\pi_n$ ends with $\pi_n$. Moreover, $\pi_i\pi_j\pi_n$ is an occurrence of $q_1$ (resp., $q_2$) if and only if $\pi_j=\pi_n-1$ (resp., $\pi_j=\pi_n+1$), and $\pi_i$ is the rightmost element to the left of, and smaller than $\pi_j$, {\em i.e., $i=\max\{s: \pi_s<\pi_j, s<j\}$}. 

    Hence, each permutation contains at most one occurrence of each pattern. The map $f$ can be constructed in the same way as in the proof of Theorem~\ref{thm-X^2_10}  to complete our proof.
\end{proof}

 \section{Concluding remark}\label{sec-con}

In this paper, we proved 112 joint equidistribution results for mesh patterns with minus antipodal shadings. We achieve this by constructing bijections, finding recurrence relations, and generating functions. In fact, computer experiments suggest that pairs 113 to 116 have the same joint distribution as pairs$\{X^{(1)}_i\}_{i=9}^{18}$ and $\{Y^{(1)}_i\}_{i=9}^{18}$, satisfying \eqref{recur-pair 8-17} and \eqref{genera-pair 8 17}; we leave justifying this observation as an open problem. Also, Table~\ref{tab-conj} presents 14 conjectured equidistributions, which we state below as separate conjectures.

\begin{table}[!ht]
 	{
 		\renewcommand{\arraystretch}{1.7}
 		\setlength{\tabcolsep}{5pt}
 \begin{center} 
 		\begin{tabular}{c | c ||  c |  c || c | c || c | c}
 			\hline
 		\footnotesize{Nr.}	& \footnotesize {Patterns} & \footnotesize{Nr.}	& \footnotesize{Patterns}&\footnotesize{Nr.}	& \footnotesize {Patterns} & \footnotesize{Nr.}	& \footnotesize{Patterns} \\
 			\hline
 			\hline
			
			$X^{(4)}_7$ & $\pattern{scale = 0.5}{3}{1/1,2/2,3/3}{0/0,0/1,3/0,2/0,2/1,2/2,2/3,3/1,3/3}\pattern{scale = 0.5}{3}{1/1,2/3,3/2}{0/0,0/1,3/0,2/0,2/1,2/2,2/3,3/1,3/3}$ & 
$X^{(4)}_8$ & $\pattern{scale = 0.5}{3}{1/1,2/2,3/3}{0/0,2/0,2/1,2/3,3/0,0/1,3/1,3/3}\pattern{scale = 0.5}{3}{1/1,2/3,3/2}{0/0,2/0,2/1,2/3,3/0,0/1,3/1,3/3}$ &
$Y^{(4)}_7$ & $\pattern{scale = 0.5}{3}{1/1,2/2,3/3}{0/0,0/2,0/3,1/0,1/2,1/3,2/2,3/2,3/3}\pattern{scale = 0.5}{3}{1/1,2/3,3/2}{0/0,0/2,0/3,1/0,1/2,1/3,2/2,3/2,3/3}$ &
$Y^{(4)}_8$ & $\pattern{scale = 0.5}{3}{1/1,2/2,3/3}{0/0,1/0,1/2,1/3,0/2,0/3,3/2,3/3}\pattern{scale = 0.5}{3}{1/1,2/3,3/2}{0/0,1/0,1/2,1/3,0/2,0/3,3/2,3/3}$\\[5pt]

 \hline

  117 & $\pattern{scale = 0.5}{3}{1/1,2/2,3/3}{0/1,0/2,0/3,1/1,1/2,1/3,2/2,3/2,3/3}\pattern{scale = 0.5}{3}{1/1,2/3,3/2}{0/1,0/2,0/3,1/1,1/2,1/3,2/2,3/2,3/3}$ & 
  118 & $\pattern{scale = 0.5}{3}{1/1,2/2,3/3}{1/0,1/1,2/0,2/1,2/2,2/3,3/0,3/1,3/3}\pattern{scale = 0.5}{3}{1/1,2/3,3/2}{1/0,1/1,2/0,2/1,2/2,2/3,3/0,3/1,3/3}$ & 
  119 & $\pattern{scale = 0.5}{3}{1/1,2/2,3/3}{0/1,0/2,0/3,1/1,1/2,1/3,2/2,2/3}\pattern{scale = 0.5}{3}{1/1,2/3,3/2}{0/1,0/2,0/3,1/1,1/2,1/3,2/2,2/3}$ &
  120 & $\pattern{scale = 0.5}{3}{1/1,2/2,3/3}{1/0,1/1,2/0,2/1,2/2,3/0,3/1,3/2}\pattern{scale = 0.5}{3}{1/1,2/3,3/2}{1/0,1/1,2/0,2/1,2/2,3/0,3/1,3/2}$\\[5pt]
  			\hline
 121 & $\pattern{scale = 0.5}{3}{1/1,2/2,3/3}{0/1,0/2,0/3,1/1,1/2,1/3,3/2,3/3}\pattern{scale = 0.5}{3}{1/1,2/3,3/2}{0/1,0/2,0/3,1/1,1/2,1/3,3/2,3/3}$ & 
  122 &$\pattern{scale = 0.5}{3}{1/1,2/2,3/3}{1/0,1/1,2/0,2/1,2/3,3/0,3/1,3/3}\pattern{scale = 0.5}{3}{1/1,2/3,3/2}{1/0,1/1,2/0,2/1,2/3,3/0,3/1,3/3}$ &
  123 & $\pattern{scale = 0.5}{3}{1/1,2/2,3/3}{0/1,0/3,1/1,1/3,2/0,2/1,2/3,3/3}\pattern{scale = 0.5}{3}{1/1,2/3,3/2}{0/1,0/3,1/1,1/3,2/0,2/1,2/3,3/3}$ &
  124 & $\pattern{scale = 0.5}{3}{1/1,2/2,3/3}{0/2,1/0,1/1,1/2,3/0,3/1,3/2,3/3}\pattern{scale = 0.5}{3}{1/1,2/3,3/2}{0/2,1/0,1/1,1/2,3/0,3/1,3/2,3/3}$\\[5pt]
 			\hline
  125 & $\pattern{scale = 0.5}{3}{1/1,2/2,3/3}{0/1,0/2,1/1,1/2,3/0,3/1,3/2,3/3}\pattern{scale = 0.5}{3}{1/1,2/3,3/2}{0/1,0/2,1/1,1/2,3/0,3/1,3/2,3/3}$ &
  126 & $\pattern{scale = 0.5}{3}{1/1,2/2,3/3}{0/3,1/0,1/1,1/3,2/0,2/1,2/3,3/3}\pattern{scale = 0.5}{3}{1/1,2/3,3/2}{0/3,1/0,1/1,1/3,2/0,2/1,2/3,3/3}$ & & & \\[5pt]
\hline
	\end{tabular}
\end{center} 
}
\vspace{-0.5cm}
 	\caption{Conjectured joint equidistributions. The names of the patterns in row 1 are consistent with the names of patterns of the same type considered in Section~\ref{sec-X_4}.}\label{tab-conj}
\end{table}

\begin{conj}[Pairs $X^{(4)}_7$ and $Y^{(4)}_7$]
We have $\pattern{scale = 0.5}{3}{1/1,2/2,3/3}{0/0,0/1,3/0,2/0,2/1,2/2,2/3,3/1,3/3}\sim_{jd}\hspace{-1mm}\pattern{scale = 0.5}{3}{1/1,2/3,3/2}{0/0,0/1,3/0,2/0,2/1,2/2,2/3,3/1,3/3}$ and $\pattern{scale = 0.5}{3}{1/1,2/2,3/3}{0/0,0/2,0/3,1/0,1/2,1/3,2/2,3/2,3/3}\sim_{jd}\hspace{-1mm}\pattern{scale = 0.5}{3}{1/1,2/3,3/2}{0/0,0/2,0/3,1/0,1/2,1/3,2/2,3/2,3/3}$.
\end{conj}

\begin{conj}[Pairs $X^{(4)}_8$ and $Y^{(4)}_8$]
We have$\pattern{scale = 0.5}{3}{1/1,2/2,3/3}{0/0,2/0,2/1,2/3,3/0,0/1,3/1,3/3}\sim_{jd}\hspace{-1mm}\pattern{scale = 0.5}{3}{1/1,2/3,3/2}{0/0,2/0,2/1,2/3,3/0,0/1,3/1,3/3}$ and $\pattern{scale = 0.5}{3}{1/1,2/2,3/3}{0/0,1/0,1/2,1/3,0/2,0/3,3/2,3/3}\sim_{jd}\hspace{-1mm}\pattern{scale = 0.5}{3}{1/1,2/3,3/2}{0/0,1/0,1/2,1/3,0/2,0/3,3/2,3/3}$.
\end{conj}

\begin{conj}[Pairs 117 and 118]
   We have $\pattern{scale = 0.5}{3}{1/1,2/2,3/3}{0/1,0/2,0/3,1/1,1/2,1/3,2/2,3/2,3/3}\sim_{jd}\hspace{-1mm}\pattern{scale = 0.5}{3}{1/1,2/3,3/2}{0/1,0/2,0/3,1/1,1/2,1/3,2/2,3/2,3/3}$ and $\pattern{scale = 0.5}{3}{1/1,2/2,3/3}{1/0,1/1,2/0,2/1,2/2,2/3,3/0,3/1,3/3}\sim_{jd}\hspace{-1mm}\pattern{scale = 0.5}{3}{1/1,2/3,3/2}{1/0,1/1,2/0,2/1,2/2,2/3,3/0,3/1,3/3}$, and these two pairs have the same joint distribution as $\{X^{(1)}_i\}_{i=5}^{8}$ and $\{Y^{(1)}_i\}_{i=5}^{8}$.
   \end{conj}
   
   \begin{conj}[Pairs 119 to 126]
      We have $\pattern{scale = 0.5}{3}{1/1,2/2,3/3}{0/1,0/2,0/3,1/1,1/2,1/3,2/2,2/3}\sim_{jd}\hspace{-1mm}\pattern{scale = 0.5}{3}{1/1,2/3,3/2}{0/1,0/2,0/3,1/1,1/2,1/3,2/2,2/3}$,
  $\pattern{scale = 0.5}{3}{1/1,2/2,3/3}{1/0,1/1,2/0,2/1,2/2,3/0,3/1,3/2}\sim_{jd}\hspace{-1mm}\pattern{scale = 0.5}{3}{1/1,2/3,3/2}{1/0,1/1,2/0,2/1,2/2,3/0,3/1,3/2}$, $\pattern{scale = 0.5}{3}{1/1,2/2,3/3}{0/1,0/2,0/3,1/1,1/2,1/3,3/2,3/3}\sim_{jd}\hspace{-1mm}\pattern{scale = 0.5}{3}{1/1,2/3,3/2}{0/1,0/2,0/3,1/1,1/2,1/3,3/2,3/3}$, $\pattern{scale = 0.5}{3}{1/1,2/2,3/3}{1/0,1/1,2/0,2/1,2/3,3/0,3/1,3/3}\sim_{jd}\hspace{-1mm}\pattern{scale = 0.5}{3}{1/1,2/3,3/2}{1/0,1/1,2/0,2/1,2/3,3/0,3/1,3/3}$, $\pattern{scale = 0.5}{3}{1/1,2/2,3/3}{0/1,0/3,1/1,1/3,2/0,2/1,2/3,3/3}\sim_{jd}\hspace{-1mm}\pattern{scale = 0.5}{3}{1/1,2/3,3/2}{0/1,0/3,1/1,1/3,2/0,2/1,2/3,3/3}$, $\pattern{scale = 0.5}{3}{1/1,2/2,3/3}{0/2,1/0,1/1,1/2,3/0,3/1,3/2,3/3}\sim_{jd}\hspace{-1mm}\pattern{scale = 0.5}{3}{1/1,2/3,3/2}{0/2,1/0,1/1,1/2,3/0,3/1,3/2,3/3}$, $\pattern{scale = 0.5}{3}{1/1,2/2,3/3}{0/1,0/2,1/1,1/2,3/0,3/1,3/2,3/3}\sim_{jd}\hspace{-1mm}\pattern{scale = 0.5}{3}{1/1,2/3,3/2}{0/1,0/2,1/1,1/2,3/0,3/1,3/2,3/3}$, and $\pattern{scale = 0.5}{3}{1/1,2/2,3/3}{0/3,1/0,1/1,1/3,2/0,2/1,2/3,3/3}\sim_{jd}\hspace{-1mm}\pattern{scale = 0.5}{3}{1/1,2/3,3/2}{0/3,1/0,1/1,1/3,2/0,2/1,2/3,3/3}$. Moreover, these eight pairs have the same joint distribution as $\{X^{(1)}_i\}_{i=9}^{18}$ and $\{Y^{(1)}_i\}_{i=9}^{18}$.
 \end{conj}
 
Although the subject of our study in this paper is the minus antipodal shadings of mesh patterns, we conclude by presenting in Table~\ref{tab-extend} non-minus-antipodal (and non-symmetric) shadings and their types that give joint equidistributions for the mesh patterns 123 and 132. Each of these results can be easily justified through the complement and/or reverse operations, and can essentially be found in one of the Subsections~\ref{subsec-X_1-1}, \ref{subsec-X_2-1}, and \ref{subsec-X_3-1}; hence, we omit the proofs.
 
 \begin{table}[!ht]
 	{
 		\renewcommand{\arraystretch}{1.2}
 		\setlength{\tabcolsep}{5pt}
 \begin{center} 
 		\begin{tabular}{|c |c | c|}
 			\hline
 		\footnotesize{Type}	 & \footnotesize{Form}	& {\footnotesize Mesh pattern $p$}   \\
 			\hline
 			\hline

\multirow{6.5}{1.7cm}{$X^{(1)}$/$Y^{(1)}$}  & 
\multirow{6.2}{1.9cm}{\raisebox{0.5ex}{\begin{tikzpicture}[scale = 0.2, baseline=(current bounding box.center)]
	\foreach \x/\y in {0/0,0/1,0/2,0/3}		
		\fill[gray!50] (\x,\y) rectangle +(1,1); 	
        \draw (0,1)--(4,1);
        \draw (1,0)--(1,4);
        \draw (0,2)--(1,2);
         \draw (0,3)--(1,3);
         \draw (2,0)--(2,1);
         \draw (3,0)--(3,1);
 			\node  at (2.6,2.6) {$p$};
 			\filldraw (1,1) circle (4pt) node[above left] {};
 			\end{tikzpicture}}/\raisebox{0.5ex}{\begin{tikzpicture}[scale = 0.2, baseline=(current bounding box.center)]
	\foreach \x/\y in {0/0,1/0,2/0,3/0}		
		\fill[gray!50] (\x,\y) rectangle +(1,1); 	
        \draw (0,1)--(4,1);
        \draw (1,0)--(1,4);
        \draw (0,2)--(1,2);
         \draw (0,3)--(1,3);
         \draw (2,0)--(2,1);
         \draw (3,0)--(3,1);
 			\node  at (2.6,2.6) {$p$};
 			\filldraw (1,1) circle (4pt) node[above left] {};
 			\end{tikzpicture}}} &  $\pattern{scale = 0.4}{2}{}{}$
 $\pattern{scale = 0.4}{2}{}{0/1}$
 $\pattern{scale =0.4}{2}{}{1/1}$
 $\pattern{scale =0.4}{2}{}{2/1}$
 $\pattern{scale = 0.4}{2}{}{0/1,1/1}$
 $\pattern{scale =0.4}{2}{}{0/1,2/1}$
 $\pattern{scale =0.4}{2}{}{1/1,2/1}$
 $\pattern{scale =0.4}{2}{}{0/0,0/2}$
 $\pattern{scale =0.4}{2}{}{1/0,1/2}$
 $\pattern{scale =0.4}{2}{}{2/0,2/2}$
 $\pattern{scale =0.4}{2}{}{0/0,0/1,0/2}$
 $\pattern{scale =0.4}{2}{}{0/0,1/1,0/2}$
 $\pattern{scale =0.4}{2}{}{0/0,0/2,2/1}$
 $\pattern{scale =0.4}{2}{}{1/0,1/2,0/1}$
 $\pattern{scale =0.4}{2}{}{1/0,1/2,1/1}$
 $\pattern{scale =0.4}{2}{}{1/0,1/2,2/1}$\\
&& $\pattern{scale =0.4}{2}{}{2/0,2/2,0/1}$
 $\pattern{scale =0.4}{2}{}{2/0,2/2,1/1}$
 $\pattern{scale =0.4}{2}{}{2/0,2/2,2/1}$
  $\pattern{scale =0.4}{2}{}{0/1,1/1,2/1}$ $\pattern{scale =0.4}{2}{}{0/0,0/2,2/0,2/2}$
 $\pattern{scale =0.4}{2}{}{0/0,0/1,0/2,1/1}$
 $\pattern{scale =0.4}{2}{}{0/0,0/2,1/1,2/1}$
 $\pattern{scale =0.4}{2}{}{1/0,1/1,1/2,0/1}$
 $\pattern{scale =0.4}{2}{}{0/1,1/0,1/2,2/1}$
 $\pattern{scale =0.4}{2}{}{1/0,1/1,1/2,2/1}$
 $\pattern{scale =0.4}{2}{}{0/1,1/1,2/0,2/2}$
 $\pattern{scale =0.4}{2}{}{1/1,2/0,2/1,2/2}$  $\pattern{scale =0.4}{2}{}{0/1,1/0,1/1,1/2,2/1}$  
$\pattern{scale =0.4}{2}{}{0/1,0/0,0/2,1/0,1/2}$
$\pattern{scale =0.4}{2}{}{0/0,0/2,1/0,1/2,2/1}$
$\pattern{scale =0.4}{2}{}{0/0,0/1,0/2,2/0,2/2}$\\
&&
$\pattern{scale =0.4}{2}{}{0/0,0/2,1/1,2/0,2/2}$
$\pattern{scale =0.4}{2}{}{0/0,0/2,2/0,2/1,2/2}$
$\pattern{scale =0.4}{2}{}{0/1,1/0,1/2,2/0,2/2}$
$\pattern{scale =0.4}{2}{}{1/0,1/2,2/0,2/1,2/2}$
$\pattern{scale=0.4}{2}{}{0/0,0/2,1/0,1/2,2/0,2/2}$
$\pattern{scale =0.4}{2}{}{0/0,0/1,0/2,1/0,1/1,1/2}$
$\pattern{scale =0.4}{2}{}{0/0,0/1,0/2,1/0,1/2,2/1}$
$\pattern{scale =0.4}{2}{}{0/0,0/2,1/0,1/1,1/2,2/1}$
$\pattern{scale =0.4}{2}{}{0/0,0/1,0/2,1/1,2/0,2/2}$
$\pattern{scale =0.4}{2}{}{0/0,0/1,0/2,2/0,2/1,2/2}$
$\pattern{scale =0.4}{2}{}{0/0,0/2,1/1,2/0,2/1,2/2}$
$\pattern{scale =0.4}{2}{}{0/1,1/0,1/1,1/2,2/0,2/2}$
$\pattern{scale =0.4}{2}{}{0/1,1/0,1/2,2/0,2/1,2/2}$
$\pattern{scale =0.4}{2}{}{1/0,1/1,1/2,2/0,2/1,2/2}$
$\pattern{scale =0.4}{2}{}{0/0,0/1,0/2,1/0,1/1,1/2,2/1}$
$\pattern{scale =0.4}{2}{}{0/0,0/1,0/2,1/1,2/0,2/1,2/2}$\\
&&
$\pattern{scale =0.4}{2}{}{0/1,1/0,1/1,1/2,2/0,2/1,2/2}$
$\pattern{scale =0.4}{2}{}{0/0,0/1,0/2,1/0,1/2,2/0,2/2}$
$\pattern{scale =0.4}{2}{}{0/0,0/2,1/0,1/1,1/2,2/0,2/2}$
$\pattern{scale =0.4}{2}{}{0/0,0/2,1/0,1/2,2/0,2/1,2/2}$
$\pattern{scale =0.4}{2}{}{0/0,0/1,0/2,1/0,1/1,1/2,2/0,2/2}$
$\pattern{scale =0.4}{2}{}{0/0,0/1,0/2,1/0,1/2,2/0,2/1,2/2}$
$\pattern{scale =0.4}{2}{}{0/0,0/2,1/0,1/1,1/2,2/0,2/1,2/2}$
$\pattern{scale =0.4}{2}{}{0/0,0/1,0/2,1/0,1/1,1/2,2/1,2/0,2/2}$
 \\
  && 
 $\pattern{scale = 0.4}{2}{}{1/2}$
 $\pattern{scale =0.4}{2}{}{1/0}$
 $\pattern{scale = 0.4}{2}{}{1/1,1/2}$
 $\pattern{scale =0.4}{2}{}{1/2,1/0}$
 $\pattern{scale =0.4}{2}{}{1/0,1/1}$
 $\pattern{scale =0.4}{2}{}{0/2,2/2}$
 $\pattern{scale =0.4}{2}{}{0/1,2/1}$
 $\pattern{scale =0.4}{2}{}{0/0,2/0}$
 $\pattern{scale =0.4}{2}{}{0/2,1/2,2/2}$
 $\pattern{scale =0.4}{2}{}{0/2,1/1,2/2}$
 $\pattern{scale =0.4}{2}{}{0/2,1/0,2/2}$
 $\pattern{scale =0.4}{2}{}{0/1,1/2,2/1}$
 $\pattern{scale =0.4}{2}{}{0/1,1/0,2/1}$
 $\pattern{scale =0.4}{2}{}{0/0,1/2,2/0}$
 $\pattern{scale =0.4}{2}{}{0/0,1/1,2/0}$
 $\pattern{scale =0.4}{2}{}{0/0,1/0,2/0}$ \\
&&
 $\pattern{scale =0.4}{2}{}{0/2,1/1,1/2,2/2}$
 $\pattern{scale =0.4}{2}{}{0/2,1/0,1/1,2/2}$
 $\pattern{scale =0.4}{2}{}{0/1,1/1,1/2,2/1}$
 $\pattern{scale =0.4}{2}{}{0/1,1/1,1/0,2/1}$
 $\pattern{scale =0.4}{2}{}{0/0,1/1,1/2,2/0}$
 $\pattern{scale =0.4}{2}{}{0/0,1/0,2/0,1/1}$  $\pattern{scale =0.4}{2}{}{0/1,0/2,1/2,2/2,2/1}$  
$\pattern{scale =0.4}{2}{}{0/1,0/2,1/0,2/1,2/2}$
$\pattern{scale =0.4}{2}{}{0/0,0/2,1/2,2/0,2/2}$
$\pattern{scale =0.4}{2}{}{0/0,0/2,1/0,2/0,2/2}$
$\pattern{scale =0.4}{2}{}{0/0,0/1,1/2,2/1,2/0}$
$\pattern{scale =0.4}{2}{}{0/0,0/1,1/0,2/1,2/0}$
$\pattern{scale =0.4}{2}{}{0/1,0/2,1/1,1/2,2/1,2/2}$
$\pattern{scale =0.4}{2}{}{0/1,0/2,1/0,1/2,2/1,2/2}$
$\pattern{scale =0.4}{2}{}{0/1,0/2,1/0,1/1,2/1,2/2}$
$\pattern{scale =0.4}{2}{}{0/0,0/2,1/1,1/2,2/0,2/2}$\\
&&
$\pattern{scale =0.4}{2}{}{0/0,0/2,1/0,1/1,2/0,2/2}$
$\pattern{scale =0.4}{2}{}{0/0,0/1,1/1,1/2,2/1,2/0}$
$\pattern{scale =0.4}{2}{}{0/0,0/1,1/0,1/2,2/0,2/1}$
$\pattern{scale =0.4}{2}{}{0/0,0/1,1/0,1/1,2/0,2/1}$
$\pattern{scale =0.4}{2}{}{0/1,0/2,1/0,1/1,1/2,2/1,2/2}$
$\pattern{scale =0.4}{2}{}{0/0,0/1,1/0,1/1,1/2,2/0,2/1}$
$\pattern{scale =0.4}{2}{}{0/0,0/1,0/2,1/2,2/0,2/1,2/2}$
$\pattern{scale =0.4}{2}{}{0/0,0/1,0/2,1/0,2/0,2/1,2/2}$
$\pattern{scale =0.4}{2}{}{0/0,0/1,0/2,1/1,1/2,2/0,2/1,2/2}$
$\pattern{scale =0.4}{2}{}{0/0,0/1,0/2,1/0,1/1,2/0,2/1,2/2}$\\[5pt]
\hline

\multirow{4.6}{1.7cm}{$X^{(2)}$/$Y^{(3)}$}  & \multirow{4.2}{1.9cm}{\raisebox{0.6ex}{\begin{tikzpicture}[scale = 0.2, baseline=(current bounding box.center)]
	\foreach \x/\y in {0/0,0/1,0/2,3/0}		
		\fill[gray!50] (\x,\y) rectangle +(1,1); 	
        \draw (0,1)--(4,1);
        \draw (1,0)--(1,4);
        \draw (0,2)--(1,2);
         \draw (0,3)--(1,3);
         \draw (2,0)--(2,1);
         \draw (3,0)--(3,1);
 			\node  at (2.6,2.6) {$p$};
 			\filldraw (1,1) circle (4pt) node[above left] {};
 			\end{tikzpicture}}/\raisebox{0.5ex}{\begin{tikzpicture}[scale = 0.2, baseline=(current bounding box.center)]
	\foreach \x/\y in {0/0,1/0,3/0,0/2}		
		\fill[gray!50] (\x,\y) rectangle +(1,1); 	
        \draw (0,1)--(4,1);
        \draw (1,0)--(1,4);
        \draw (0,2)--(1,2);
         \draw (0,3)--(1,3);
         \draw (2,0)--(2,1);
         \draw (3,0)--(3,1);
 			\node  at (2.6,2.6) {$p$};
 			\filldraw (1,1) circle (4pt) node[above left] {};
 			\end{tikzpicture}}}
   &  $\pattern{scale = 0.4}{2}{}{}$
 $\pattern{scale = 0.4}{2}{}{0/1}$
 $\pattern{scale =0.4}{2}{}{1/1}$
 $\pattern{scale =0.4}{2}{}{2/1}$
 $\pattern{scale =0.4}{2}{}{0/1,1/1}$
 $\pattern{scale =0.4}{2}{}{0/1,2/1}$
 $\pattern{scale =0.4}{2}{}{1/1,2/1}$
 $\pattern{scale =0.4}{2}{}{0/0,0/2}$
 $\pattern{scale =0.4}{2}{}{1/0,1/2}$
 $\pattern{scale =0.4}{2}{}{2/0,2/2}$
 $\pattern{scale =0.4}{2}{}{0/0,0/1,0/2}$
 $\pattern{scale =0.4}{2}{}{0/0,1/1,0/2}$
 $\pattern{scale =0.4}{2}{}{0/0,0/2,2/1}$
 $\pattern{scale =0.4}{2}{}{1/0,1/2,0/1}$
 $\pattern{scale =0.4}{2}{}{1/0,1/2,1/1}$
 $\pattern{scale =0.4}{2}{}{1/0,1/2,2/1}$\\
&&
 $\pattern{scale =0.4}{2}{}{2/0,2/2,0/1}$
 $\pattern{scale =0.4}{2}{}{2/0,2/2,1/1}$
 $\pattern{scale =0.4}{2}{}{2/0,2/2,2/1}$
  $\pattern{scale =0.4}{2}{}{0/1,1/1,2/1}$ $\pattern{scale =0.4}{2}{}{0/0,0/2,2/0,2/2}$
 $\pattern{scale =0.4}{2}{}{0/0,0/1,0/2,1/1}$
 $\pattern{scale =0.4}{2}{}{0/0,0/2,1/1,2/1}$
 $\pattern{scale =0.4}{2}{}{1/0,1/1,1/2,0/1}$
 $\pattern{scale =0.4}{2}{}{0/1,1/0,1/2,2/1}$
 $\pattern{scale =0.4}{2}{}{1/0,1/1,1/2,2/1}$
 $\pattern{scale =0.4}{2}{}{0/1,1/1,2/0,2/2}$
 $\pattern{scale =0.4}{2}{}{1/1,2/0,2/1,2/2}$  $\pattern{scale =0.4}{2}{}{0/1,1/0,1/1,1/2,2/1}$  
$\pattern{scale =0.4}{2}{}{0/1,0/0,0/2,1/0,1/2}$
$\pattern{scale =0.4}{2}{}{0/0,0/2,1/0,1/2,2/1}$
$\pattern{scale =0.4}{2}{}{0/0,0/1,0/2,2/0,2/2}$\\
&&
$\pattern{scale =0.4}{2}{}{0/0,0/2,1/1,2/0,2/2}$
$\pattern{scale =0.4}{2}{}{0/0,0/2,2/0,2/1,2/2}$
$\pattern{scale =0.4}{2}{}{0/1,1/0,1/2,2/0,2/2}$
$\pattern{scale =0.4}{2}{}{1/0,1/2,2/0,2/1,2/2}$
$\pattern{scale=0.4}{2}{}{0/0,0/2,1/0,1/2,2/0,2/2}$
$\pattern{scale =0.4}{2}{}{0/0,0/1,0/2,1/0,1/1,1/2}$
$\pattern{scale =0.4}{2}{}{0/0,0/1,0/2,1/0,1/2,2/1}$
$\pattern{scale =0.4}{2}{}{0/0,0/2,1/0,1/1,1/2,2/1}$
$\pattern{scale =0.4}{2}{}{0/0,0/1,0/2,1/1,2/0,2/2}$
$\pattern{scale =0.4}{2}{}{0/0,0/1,0/2,2/0,2/1,2/2}$
$\pattern{scale =0.4}{2}{}{0/0,0/2,1/1,2/0,2/1,2/2}$
$\pattern{scale =0.4}{2}{}{0/1,1/0,1/1,1/2,2/0,2/2}$
$\pattern{scale =0.4}{2}{}{0/1,1/0,1/2,2/0,2/1,2/2}$
$\pattern{scale =0.4}{2}{}{1/0,1/1,1/2,2/0,2/1,2/2}$
$\pattern{scale =0.4}{2}{}{0/0,0/1,0/2,1/0,1/1,1/2,2/1}$
$\pattern{scale =0.4}{2}{}{0/0,0/1,0/2,1/1,2/0,2/1,2/2}$\\
&&
$\pattern{scale =0.4}{2}{}{0/1,1/0,1/1,1/2,2/0,2/1,2/2}$
$\pattern{scale =0.4}{2}{}{0/0,0/1,0/2,1/0,1/2,2/0,2/2}$
$\pattern{scale =0.4}{2}{}{0/0,0/2,1/0,1/1,1/2,2/0,2/2}$
$\pattern{scale =0.4}{2}{}{0/0,0/2,1/0,1/2,2/0,2/1,2/2}$
$\pattern{scale =0.4}{2}{}{0/0,0/1,0/2,1/0,1/1,1/2,2/0,2/2}$
$\pattern{scale =0.4}{2}{}{0/0,0/1,0/2,1/0,1/2,2/0,2/1,2/2}$
$\pattern{scale =0.4}{2}{}{0/0,0/2,1/0,1/1,1/2,2/0,2/1,2/2}$
$\pattern{scale =0.4}{2}{}{0/0,0/1,0/2,1/0,1/1,1/2,2/1,2/0,2/2}$
 \\[5pt]
\hline
\multirow{4.5}{1.7cm}{$Y^{(2)}$/$X^{(3)}$} & \multirow{4.2}{1.9cm}{\raisebox{0.6ex}{\begin{tikzpicture}[scale = 0.2, baseline=(current bounding box.center)]
	\foreach \x/\y in {0/0,1/0,2/0,0/3}		
		\fill[gray!50] (\x,\y) rectangle +(1,1); 	
        \draw (0,1)--(4,1);
        \draw (1,0)--(1,4);
        \draw (0,2)--(1,2);
         \draw (0,3)--(1,3);
         \draw (2,0)--(2,1);
         \draw (3,0)--(3,1);
 			\node  at (2.6,2.6) {$p$};
 			\filldraw (1,1) circle (4pt) node[above left] {};
 			\end{tikzpicture}}/\raisebox{0.5ex}{\begin{tikzpicture}[scale = 0.2, baseline=(current bounding box.center)]
	\foreach \x/\y in {0/0,0/1,0/3,2/0}		
		\fill[gray!50] (\x,\y) rectangle +(1,1); 	
        \draw (0,1)--(4,1);
        \draw (1,0)--(1,4);
        \draw (0,2)--(1,2);
         \draw (0,3)--(1,3);
         \draw (2,0)--(2,1);
         \draw (3,0)--(3,1);
 			\node  at (2.6,2.6) {$p$};
 			\filldraw (1,1) circle (4pt) node[above left] {};
 			\end{tikzpicture}}}
  &  $\pattern{scale =0.4}{2}{}{}$
 $\pattern{scale =0.4}{2}{}{1/2}$
 $\pattern{scale =0.4}{2}{}{1/1}$
 $\pattern{scale =0.4}{2}{}{1/0}$
 $\pattern{scale =0.4}{2}{}{1/1,1/2}$
 $\pattern{scale =0.4}{2}{}{1/2,1/0}$
 $\pattern{scale =0.4}{2}{}{1/0,1/1}$
 $\pattern{scale =0.4}{2}{}{0/2,2/2}$
 $\pattern{scale =0.4}{2}{}{0/1,2/1}$
 $\pattern{scale =0.4}{2}{}{0/0,2/0}$
 $\pattern{scale =0.4}{2}{}{0/2,1/2,2/2}$
 $\pattern{scale =0.4}{2}{}{0/2,1/1,2/2}$
 $\pattern{scale =0.4}{2}{}{0/2,1/0,2/2}$
 $\pattern{scale =0.4}{2}{}{0/1,1/2,2/1}$
 $\pattern{scale =0.4}{2}{}{0/1,1/1,2/1}$
 $\pattern{scale =0.4}{2}{}{0/1,1/0,2/1}$\\
&&
 $\pattern{scale =0.4}{2}{}{0/0,1/2,2/0}$
 $\pattern{scale =0.4}{2}{}{0/0,1/1,2/0}$
 $\pattern{scale =0.4}{2}{}{0/0,1/0,2/0}$
  $\pattern{scale =0.4}{2}{}{1/0,1/1,1/2}$ $\pattern{scale =0.4}{2}{}{0/0,0/2,2/0,2/2}$
 $\pattern{scale =0.4}{2}{}{0/2,1/1,1/2,2/2}$
 $\pattern{scale =0.4}{2}{}{0/2,1/0,1/1,2/2}$
 $\pattern{scale =0.4}{2}{}{0/1,1/1,1/2,2/1}$
 $\pattern{scale =0.4}{2}{}{0/1,1/0,1/2,2/1}$
 $\pattern{scale =0.4}{2}{}{0/1,1/1,1/0,2/1}$
 $\pattern{scale =0.4}{2}{}{0/0,1/1,1/2,2/0}$
 $\pattern{scale =0.4}{2}{}{0/0,1/0,2/0,1/1}$  $\pattern{scale =0.4}{2}{}{0/1,0/2,1/2,2/2,2/1}$  
$\pattern{scale =0.4}{2}{}{0/1,0/2,1/0,2/1,2/2}$
$\pattern{scale =0.4}{2}{}{0/0,0/2,1/2,2/0,2/2}$
$\pattern{scale =0.4}{2}{}{0/0,0/2,1/1,2/0,2/2}$\\
&&
$\pattern{scale =0.4}{2}{}{0/0,0/2,1/0,2/0,2/2}$
$\pattern{scale =0.4}{2}{}{0/0,0/1,1/2,2/1,2/0}$
$\pattern{scale =0.4}{2}{}{0/0,0/1,1/0,2/1,2/0}$
$\pattern{scale =0.4}{2}{}{0/1,1/0,1/1,1/2,2/1}$
$\pattern{scale=0.4}{2}{}{0/0,0/1,0/2,2/0,2/1,2/2}$
$\pattern{scale =0.4}{2}{}{0/1,0/2,1/1,1/2,2/1,2/2}$
$\pattern{scale =0.4}{2}{}{0/1,0/2,1/0,1/2,2/1,2/2}$
$\pattern{scale =0.4}{2}{}{0/1,0/2,1/0,1/1,2/1,2/2}$
$\pattern{scale =0.4}{2}{}{0/0,0/2,1/1,1/2,2/0,2/2}$
$\pattern{scale =0.4}{2}{}{0/0,1/0,2/0,0/2,1/2,2/2}$
$\pattern{scale =0.4}{2}{}{0/0,0/2,1/0,1/1,2/0,2/2}$
$\pattern{scale =0.4}{2}{}{0/0,0/1,1/1,1/2,2/1,2/0}$
$\pattern{scale =0.4}{2}{}{0/0,0/1,1/0,1/2,2/0,2/1}$
$\pattern{scale =0.4}{2}{}{0/0,0/1,1/0,1/1,2/0,2/1}$
$\pattern{scale =0.4}{2}{}{0/1,0/2,1/0,1/1,1/2,2/1,2/2}$
$\pattern{scale =0.4}{2}{}{0/0,0/2,1/0,1/1,1/2,2/0,2/2}$\\
&&
$\pattern{scale =0.4}{2}{}{0/0,0/1,1/0,1/1,1/2,2/0,2/1}$
$\pattern{scale =0.4}{2}{}{0/0,0/1,0/2,1/2,2/0,2/1,2/2}$
$\pattern{scale =0.4}{2}{}{0/0,0/1,0/2,1/1,2/0,2/1,2/2}$
$\pattern{scale =0.4}{2}{}{0/0,0/1,0/2,1/0,2/0,2/1,2/2}$
$\pattern{scale =0.4}{2}{}{0/0,0/1,0/2,1/1,1/2,2/0,2/1,2/2}$
$\pattern{scale =0.4}{2}{}{0/0,0/1,0/2,1/0,1/2,2/0,2/1,2/2}$
$\pattern{scale =0.4}{2}{}{0/0,0/1,0/2,1/0,1/1,2/0,2/1,2/2}$
$\pattern{scale =0.4}{2}{}{0/0,0/1,0/2,1/0,1/1,1/2,2/1,2/0,2/2}$
 \\[5pt]
\hline
\multirow{6.5}{0.7cm}{\rule{\dimexpr \linewidth}{0.4pt}\\[0.2cm]}  & \multirow{4.2}{0.9cm}{\raisebox{0.6ex}{\begin{tikzpicture}[scale = 0.2, baseline=(current bounding box.center)]
	\foreach \x/\y in {0/0,0/1,0/2}		
		\fill[gray!50] (\x,\y) rectangle +(1,1); 	
        \draw (0,1)--(4,1);
        \draw (1,0)--(1,4);
        \draw (0,2)--(1,2);
         \draw (0,3)--(1,3);
         \draw (2,0)--(2,1);
         \draw (3,0)--(3,1);
 			\node  at (2.6,2.6) {$p$};
 			\filldraw (1,1) circle (4pt) node[above left] {};
 			\end{tikzpicture}}}
   &  $\pattern{scale = 0.4}{2}{}{}$
 $\pattern{scale = 0.4}{2}{}{0/1}$
 $\pattern{scale =0.4}{2}{}{1/1}$
 $\pattern{scale =0.4}{2}{}{2/1}$
 $\pattern{scale =0.4}{2}{}{0/1,1/1}$
 $\pattern{scale =0.4}{2}{}{0/1,2/1}$
 $\pattern{scale =0.4}{2}{}{1/1,2/1}$
 $\pattern{scale =0.4}{2}{}{0/0,0/2}$
 $\pattern{scale =0.4}{2}{}{1/0,1/2}$
 $\pattern{scale =0.4}{2}{}{2/0,2/2}$
 $\pattern{scale =0.4}{2}{}{0/0,0/1,0/2}$
 $\pattern{scale =0.4}{2}{}{0/0,1/1,0/2}$
 $\pattern{scale =0.4}{2}{}{0/0,0/2,2/1}$
 $\pattern{scale =0.4}{2}{}{1/0,1/2,0/1}$
 $\pattern{scale =0.4}{2}{}{1/0,1/2,1/1}$
 $\pattern{scale =0.4}{2}{}{1/0,1/2,2/1}$\\
&&
 $\pattern{scale =0.4}{2}{}{2/0,2/2,0/1}$
 $\pattern{scale =0.4}{2}{}{2/0,2/2,1/1}$
 $\pattern{scale =0.4}{2}{}{2/0,2/2,2/1}$
  $\pattern{scale =0.4}{2}{}{0/1,1/1,2/1}$
  $\pattern{scale = 0.4}{2}{}{0/0,1/0,0/2,1/2}$
$\pattern{scale = 0.4}{2}{}{1/0,2/0,1/2,2/2}$
$\pattern{scale =0.4}{2}{}{0/0,0/2,2/0,2/2}$
 $\pattern{scale =0.4}{2}{}{0/0,0/1,0/2,1/1}$
 $\pattern{scale =0.4}{2}{}{0/0,0/2,1/1,2/1}$
 $\pattern{scale =0.4}{2}{}{0/0,0/1,0/2,2/1}$
 $\pattern{scale =0.4}{2}{}{0/1,2/0,2/1,2/2}$
 $\pattern{scale =0.4}{2}{}{1/0,1/1,1/2,0/1}$
 $\pattern{scale =0.4}{2}{}{0/1,1/0,1/2,2/1}$
 $\pattern{scale =0.4}{2}{}{1/0,1/1,1/2,2/1}$
 $\pattern{scale =0.4}{2}{}{0/1,1/1,2/0,2/2}$
 $\pattern{scale =0.4}{2}{}{1/1,2/0,2/1,2/2}$ \\
&&
 $\pattern{scale = 0.4}{2}{}{0/0,0/1,1/1,0/2,2/1}$
$\pattern{scale = 0.4}{2}{}{0/1,2/0,1/1,2/1,2/2}$
$\pattern{scale =0.4}{2}{}{0/1,1/0,1/1,1/2,2/1}$  
$\pattern{scale =0.4}{2}{}{0/1,0/0,0/2,1/0,1/2}$
$\pattern{scale =0.4}{2}{}{0/0,0/2,1/0,1/2,2/1}$
$\pattern{scale =0.4}{2}{}{0/0,0/1,0/2,2/0,2/2}$
$\pattern{scale =0.4}{2}{}{0/0,0/2,1/1,2/0,2/2}$
$\pattern{scale =0.4}{2}{}{0/0,0/2,2/0,2/1,2/2}$
$\pattern{scale =0.4}{2}{}{0/1,1/0,1/2,2/0,2/2}$
$\pattern{scale =0.4}{2}{}{1/0,1/2,2/0,2/1,2/2}$
$\pattern{scale=0.4}{2}{}{0/0,0/2,1/0,1/2,2/0,2/2}$
$\pattern{scale =0.4}{2}{}{0/0,0/1,0/2,1/0,1/1,1/2}$
$\pattern{scale =0.4}{2}{}{0/0,0/1,0/2,1/0,1/2,2/1}$
$\pattern{scale =0.4}{2}{}{0/0,0/2,1/0,1/1,1/2,2/1}$
$\pattern{scale =0.4}{2}{}{0/0,0/1,0/2,1/1,2/0,2/2}$
$\pattern{scale =0.4}{2}{}{0/0,0/1,0/2,2/0,2/1,2/2}$\\
&&
$\pattern{scale =0.4}{2}{}{0/0,0/2,1/1,2/0,2/1,2/2}$
$\pattern{scale =0.4}{2}{}{0/1,1/0,1/1,1/2,2/0,2/2}$
$\pattern{scale =0.4}{2}{}{0/1,1/0,1/2,2/0,2/1,2/2}$
$\pattern{scale =0.4}{2}{}{1/0,1/1,1/2,2/0,2/1,2/2}$
$\pattern{scale =0.4}{2}{}{0/0,0/1,0/2,1/0,1/1,1/2,2/1}$
$\pattern{scale =0.4}{2}{}{0/0,0/1,0/2,1/1,2/0,2/1,2/2}$
$\pattern{scale =0.4}{2}{}{0/1,1/0,1/1,1/2,2/0,2/1,2/2}$
$\pattern{scale =0.4}{2}{}{0/0,0/1,0/2,1/0,1/2,2/0,2/2}$
$\pattern{scale =0.4}{2}{}{0/0,0/2,1/0,1/1,1/2,2/0,2/2}$
$\pattern{scale =0.4}{2}{}{0/0,0/2,1/0,1/2,2/0,2/1,2/2}$
$\pattern{scale =0.4}{2}{}{0/0,0/1,0/2,1/0,1/1,1/2,2/0,2/2}$
$\pattern{scale =0.4}{2}{}{0/0,0/1,0/2,1/0,1/2,2/0,2/1,2/2}$
$\pattern{scale =0.4}{2}{}{0/0,0/2,1/0,1/1,1/2,2/0,2/1,2/2}$
$\pattern{scale =0.4}{2}{}{0/0,0/1,0/2,1/0,1/1,1/2,2/1,2/0,2/2}$
 \\
&&$\pattern{scale = 0.4}{2}{}{0/0,1/0,1/1,1/2,2/0}$ $\pattern{scale = 0.4}{2}{}{0/2,1/0,1/1,1/2,2/2}$ $\pattern{scale = 0.4}{2}{}{0/0,0/1,1/1,2/0,2/1}$ $\pattern{scale = 0.4}{2}{}{0/1,0/2,1/1,2/1,2/2}$ $\pattern{scale = 0.4}{2}{}{0/1,0/2,1/1,1/2}$  $\pattern{scale = 0.4}{2}{}{0/0,0/2,1/0,1/1,1/2,2/2}$ $\pattern{scale = 0.4}{2}{}{0/0,1/0,1/1,1/2,2/0,2/2}$ $\pattern{scale = 0.4}{2}{}{0/0,0/1,0/2,1/1,2/1,2/2}$ $\pattern{scale = 0.4}{2}{}{0/0,0/1,1/1,2/0,2/1,2/2}$\\[5pt]
\hline
\multirow{6.5}{0.7cm}{\rule{\dimexpr \linewidth}{0.4pt}\\[0.3cm]} & \multirow{4.2}{0.9cm}{\raisebox{0.6ex}{\begin{tikzpicture}[scale = 0.2, baseline=(current bounding box.center)]
	\foreach \x/\y in {0/0,1/0,2/0}		
		\fill[gray!50] (\x,\y) rectangle +(1,1); 	
        \draw (0,1)--(4,1);
        \draw (1,0)--(1,4);
        \draw (0,2)--(1,2);
         \draw (0,3)--(1,3);
         \draw (2,0)--(2,1);
         \draw (3,0)--(3,1);
 			\node  at (2.6,2.6) {$p$};
 			\filldraw (1,1) circle (4pt) node[above left] {};
 			\end{tikzpicture}}}
  &  $\pattern{scale =0.4}{2}{}{}$
 $\pattern{scale =0.4}{2}{}{1/2}$
 $\pattern{scale =0.4}{2}{}{1/1}$
 $\pattern{scale =0.4}{2}{}{1/0}$
 $\pattern{scale =0.4}{2}{}{1/1,1/2}$
 $\pattern{scale =0.4}{2}{}{1/2,1/0}$
 $\pattern{scale =0.4}{2}{}{1/0,1/1}$
 $\pattern{scale =0.4}{2}{}{0/2,2/2}$
 $\pattern{scale =0.4}{2}{}{0/1,2/1}$
 $\pattern{scale =0.4}{2}{}{0/0,2/0}$
 $\pattern{scale =0.4}{2}{}{0/2,1/2,2/2}$
 $\pattern{scale =0.4}{2}{}{0/2,1/1,2/2}$
 $\pattern{scale =0.4}{2}{}{0/2,1/0,2/2}$
 $\pattern{scale =0.4}{2}{}{0/1,1/2,2/1}$
 $\pattern{scale =0.4}{2}{}{0/1,1/1,2/1}$
 $\pattern{scale =0.4}{2}{}{0/1,1/0,2/1}$\\
&&
 $\pattern{scale =0.4}{2}{}{0/0,1/2,2/0}$
 $\pattern{scale =0.4}{2}{}{0/0,1/1,2/0}$
 $\pattern{scale =0.4}{2}{}{0/0,1/0,2/0}$
  $\pattern{scale =0.4}{2}{}{1/0,1/1,1/2}$
  $\pattern{scale =0.4}{2}{}{0/0,0/1,2/0,2/1}$
  $\pattern{scale =0.4}{2}{}{0/1,0/2,2/1,2/2}$
  $\pattern{scale =0.4}{2}{}{0/0,0/2,2/0,2/2}$
 $\pattern{scale =0.4}{2}{}{0/2,1/1,1/2,2/2}$
 $\pattern{scale =0.4}{2}{}{0/2,1/0,1/1,2/2}$
 $\pattern{scale =0.4}{2}{}{0/1,1/1,1/2,2/1}$
 $\pattern{scale =0.4}{2}{}{0/1,1/0,1/2,2/1}$
 $\pattern{scale =0.4}{2}{}{0/1,1/1,1/0,2/1}$
 $\pattern{scale =0.4}{2}{}{0/0,1/1,1/2,2/0}$
 $\pattern{scale =0.4}{2}{}{0/0,1/0,2/0,1/1}$
  $\pattern{scale =0.4}{2}{}{0/0,1/0,2/0,1/2}$
  $\pattern{scale =0.4}{2}{}{0/2,2/2,1/2,1/0}$\\
&&
  $\pattern{scale =0.4}{2}{}{0/1,0/2,1/2,2/2,2/1}$  
$\pattern{scale =0.4}{2}{}{0/1,0/2,1/0,2/1,2/2}$
$\pattern{scale =0.4}{2}{}{0/0,0/2,1/2,2/0,2/2}$
$\pattern{scale =0.4}{2}{}{0/0,0/2,1/1,2/0,2/2}$
$\pattern{scale =0.4}{2}{}{0/0,0/2,1/0,2/0,2/2}$
$\pattern{scale =0.4}{2}{}{0/0,0/1,1/2,2/1,2/0}$
$\pattern{scale =0.4}{2}{}{0/0,0/1,1/0,2/1,2/0}$
$\pattern{scale =0.4}{2}{}{0/1,1/0,1/1,1/2,2/1}$ $\pattern{scale =0.4}{2}{}{0/0,1/0,2/0,1/1,1/2}$
  $\pattern{scale =0.4}{2}{}{0/2,1/2,2/2,1/1,1/0}$
$\pattern{scale=0.4}{2}{}{0/0,0/1,0/2,2/0,2/1,2/2}$
$\pattern{scale =0.4}{2}{}{0/1,0/2,1/1,1/2,2/1,2/2}$
$\pattern{scale =0.4}{2}{}{0/1,0/2,1/0,1/2,2/1,2/2}$
$\pattern{scale =0.4}{2}{}{0/1,0/2,1/0,1/1,2/1,2/2}$
$\pattern{scale =0.4}{2}{}{0/0,0/2,1/1,1/2,2/0,2/2}$
$\pattern{scale =0.4}{2}{}{0/0,1/0,2/0,0/2,1/2,2/2}$\\
&&
$\pattern{scale =0.4}{2}{}{0/0,0/2,1/0,1/1,2/0,2/2}$
$\pattern{scale =0.4}{2}{}{0/0,0/1,1/1,1/2,2/1,2/0}$
$\pattern{scale =0.4}{2}{}{0/0,0/1,1/0,1/2,2/0,2/1}$
$\pattern{scale =0.4}{2}{}{0/0,0/1,1/0,1/1,2/0,2/1}$
$\pattern{scale =0.4}{2}{}{0/1,0/2,1/0,1/1,1/2,2/1,2/2}$
$\pattern{scale =0.4}{2}{}{0/0,0/2,1/0,1/1,1/2,2/0,2/2}$
$\pattern{scale =0.4}{2}{}{0/0,0/1,1/0,1/1,1/2,2/0,2/1}$
$\pattern{scale =0.4}{2}{}{0/0,0/1,0/2,1/2,2/0,2/1,2/2}$
$\pattern{scale =0.4}{2}{}{0/0,0/1,0/2,1/1,2/0,2/1,2/2}$
$\pattern{scale =0.4}{2}{}{0/0,0/1,0/2,1/0,2/0,2/1,2/2}$
$\pattern{scale =0.4}{2}{}{0/0,0/1,0/2,1/1,1/2,2/0,2/1,2/2}$
$\pattern{scale =0.4}{2}{}{0/0,0/1,0/2,1/0,1/2,2/0,2/1,2/2}$ 
$\pattern{scale =0.4}{2}{}{0/0,0/1,0/2,1/0,1/1,2/0,2/1,2/2}$
$\pattern{scale =0.4}{2}{}{0/0,0/1,0/2,1/0,1/1,1/2,2/1,2/0,2/2}$
\\
 &&$\pattern{scale = 0.4}{2}{}{0/0,0/1,0/2,1/1,2/1}$ $\pattern{scale = 0.4}{2}{}{0/1,1/1,2/0,2/1,2/2}$ $\pattern{scale = 0.4}{2}{}{0/0,0/2,1/0,1/1,1/2}$ $\pattern{scale = 0.4}{2}{}{1/0,1/1,1/2,2/0,2/2}$ $\pattern{scale = 0.4}{2}{}{1/0,1/1,2/0,2/1}$  $\pattern{scale = 0.4}{2}{}{0/0,0/1,1/1,2/0,2/1,2/2}$  $\pattern{scale = 0.4}{2}{}{0/0,0/1,0/2,1/1,2/1,2/2}$ $\pattern{scale = 0.4}{2}{}{0/0,1/0,1/1,1/2,2/0,2/2}$ $\pattern{scale = 0.4}{2}{}{0/0,0/2,1/0,1/1,1/2,2/2}$\\[5pt]
\hline
	\end{tabular}
\end{center} 
}
\vspace{-0.5cm}
 	\caption{Non-minus-antipodal and non-symmetric shadings (and their types) giving joint equidistributions for the mesh patterns 123 and 132.}\label{tab-extend}
\end{table}
 
\section*{Acknowledgements} 
The authors are grateful to Sergey Kitaev for helpful discussion related to this paper. 
The work of Philip B. Zhang was supported by the National Natural Science Foundation of China (No.\ 12171362) and the Tianjin Municipal Natural Science Foundation (No. 25JCYBJC00430).

\end{document}